\documentclass[11pt]{article}
\usepackage[total={6.5in, 9in}]{geometry}
\usepackage{amsmath,amsfonts,mathrsfs,amsthm,amssymb}
\numberwithin{equation}{section}
\usepackage{cases}
\usepackage{latexsym,bm}
\usepackage{indentfirst}
\usepackage{color}
\usepackage{ifpdf}
\usepackage{graphicx}
\usepackage{psfrag}
\usepackage{dsfont}
\usepackage{enumerate}

\usepackage{pgf,tikz,pgfplots}
\usepackage{tikz-3dplot}
\usetikzlibrary{positioning,fit,backgrounds,calc,decorations.pathreplacing}
\usepackage{pifont}
\usepackage{refcount}
\usepackage[linesnumbered,ruled,vlined]{algorithm2e}
\usepackage{subcaption}

\usepackage[
pdfauthor={},
pdftitle={},
pdfstartview=XYZ,
bookmarks=true,
colorlinks=true,
linkcolor=blue,
urlcolor=blue,
citecolor=blue,
linktocpage=true,
hyperindex=true
]{hyperref}

\theoremstyle{definition}
\newtheorem{THM}{\textbf{Theorem}}[section]
\newtheorem{DEF}{\textbf{Definition}}[section]
\newtheorem{LEM}[THM]{\textbf{Lemma}}
\newtheorem{CLA}{\textbf{Claim}}[section]
\newtheorem{CON}[THM]{\textbf{Conjecture}}

\newtheorem{Ope}{\textbf{Operation}}[section]

\providecommand{\phantomsection}{} \makeatletter \newcommand{\CaseLabel}[1]{%
\smallskip \noindent\textbf{Case #1.}\phantomsection \def\@currentlabel{#1}\label{case:#1}\ } 
\makeatother

\newcommand{\CC}{\mathcal{C}}
\DeclareMathOperator{\df}{def}
\newcommand{\ve}{\varepsilon}
\newcommand{\pbar}{\overline{\varphi}}

\DeclareMathOperator{\pr}{Pr}

\begin{document}
	\title{Towards the Overfull  Conjecture II}
	\author{Guantao Chen \thanks{Georgia State
			University, Department of Mathematics and Statistics,  Atlanta, GA 30303. Email: {\tt gchen@gsu.edu}.  Partially supported by NSF grant DMS-2154331.}
		\quad 
		Jessica McDonald\thanks{Auburn University, Department of Mathematics and Statistics, Auburn, AL 36849.
			Email: {\tt mcdonald@auburn.edu}.   
			Partially supported  by Simons Foundation Grant \#845698 and NSF grant DMS-2452103.}	
		\quad 
		Songling Shan \thanks{Auburn University, Department of Mathematics and Statistics, Auburn, AL 36849. 
			Email:  	{\tt szs0398@auburn.edu}. Partially supported by NSF  grant 
			DMS-2451895.}
	}
	
	\date{\today}
	\maketitle

	\emph{\textbf{Abstract}.}
	Let $G$ be a simple graph with maximum degree $\Delta(G)$. A subgraph $H\subseteq G$ is \emph{$\Delta(G)$-overfull} if $|E(H)|>\Delta(G)\left\lfloor |V(H)|/2\right\rfloor$. In any edge coloring of $G$, each color class restricted to $H$ is a matching of size at most $\left\lfloor |V(H)|/2\right\rfloor$. Thus, if $G$ contains a $\Delta(G)$-overfull subgraph, then $G$ cannot be edge-colored with only $\Delta(G)$ colors. By Vizing's Theorem, $\chi'(G)\le \Delta(G)+1$, and hence $G$ is class $2$. In 1986, Chetwynd and Hilton conjectured that whenever $\Delta(G)>|V(G)|/3$, the converse also holds: every class $2$ graph $G$ contains a $\Delta(G)$-overfull subgraph. This statement, commonly known as the \emph{Overfull Conjecture}, is one of the most influential conjectures in graph edge coloring. It would imply a polynomial-time algorithm for determining the chromatic index of graphs $G$ with $\Delta(G)>|V(G)|/3$, and would also imply several other longstanding conjectures in the area, including the Just-overfull Conjecture and the Vertex-splitting Conjecture. In previous work, the third author verified the conjecture for large graphs $G$ with maximum degree at least $13|V(G)|/14$. In this paper, we confirm the conjecture for robust expanders satisfying certain  density constraints. As a consequence, for every $0<\ve<1$, the conjecture holds for all sufficiently large graphs $G$ with maximum degree at least $(1+\ve)|V(G)|/2$.

	
		\newpage 
	
	\tableofcontents

	\newpage 
	
	\section{Introduction}

	In this paper, the term \emph{graph} always refers to a simple graph. When
	multiple edges are allowed, but loops are not, we use the term \emph{multigraph}.
	Since multigraphs arise in the proof of the main result, we state the relevant definitions and terminology in the setting of multigraphs. If $G$ is a
	multigraph, then, for ease of notation, we still refer to any multigraph $H$ with
	$V(H)\subseteq V(G)$ and $E(H)\subseteq E(G)$ as a subgraph of $G$, rather than
	as a submultigraph of $G$.
	
	Let $\mathbb{Z}$ be the set of integers and let $\mathbb{N}$ be the set of
	nonnegative integers. For integers $p,q$, let
	$[p,q]=\{i\in\mathbb{Z}: p\le i\le q\}$, and write $[q]$ for $[1,q]$. For  
 $k\in \mathbb{N}$, an \emph{edge $k$-coloring} of a multigraph $G$ is a
	mapping $\varphi:E(G)\to [k]$, whose elements are called \emph{colors}, such that
	no two adjacent edges receive the same color. Each set of edges receiving the
	same color under $\varphi$ is a \emph{color class}. The \emph{chromatic index} of
	$G$, denoted by $\chi'(G)$, is the smallest integer $k$ such that $G$ admits an
	edge $k$-coloring.
	
	A multigraph $G$ with $|E(G)|>\Delta(G)\lfloor |V(G)|/2\rfloor$ is called
	\emph{overfull}. More generally, a subgraph $H$ of $G$ with
	$|E(H)|>\Delta(G)\lfloor |V(H)|/2\rfloor$ is called a
	\emph{$\Delta(G)$-overfull subgraph} of $G$.   Since $2|E(H)| \le \Delta(G)|V(H)|$,
	it follows that any $\Delta(G)$-overfull subgraph has odd order. 
	In any edge $k$-coloring of $G$, 
	since  each color class  restricted to $H$  is a matching of size at most $\lfloor |V(H)|/2\rfloor$, if
	$G$ contains a $\Delta(G)$-overfull subgraph $H$, then $\chi'(G)\ge \chi'(H)\ge
	\frac{|E(H)|}{\lfloor |V(H)|/2\rfloor}>\Delta(G)$.

	 In the 1960s, Gupta~\cite{Gupta-67} and, independently, Vizing~\cite{Vizing-2-classes}
	 showed that, for every multigraph $G$, $ \Delta(G)\le \chi'(G)\le \Delta(G)+\mu(G)$, 
	 where $\mu(G)$, the \emph{multiplicity} of $G$, is the maximum number of edges
	 joining two vertices in $G$. In the case of simple graphs, where $\mu(G)\le 1$,
	 this yields a natural classification. Following Fiorini and Wilson~\cite{fw},
	 a graph $G$ is of \emph{class 1} if $\chi'(G)=\Delta(G)$, and of
	 \emph{class 2} if $\chi'(G)=\Delta(G)+1$. Holyer~\cite{Holyer} showed that it is
	 NP-complete to determine whether an arbitrary graph is of class 1. However, a
	 conjecture of Chetwynd and Hilton~\cite{MR975994,MR848854} from 1986 would give
	 a polynomial-time algorithm for determining the chromatic index of graphs $G$
	 with $\Delta(G)>\frac{1}{3}|V(G)|$. This conjecture is formulated in terms of
	 overfull subgraphs, and is listed as  one of the 
	 \emph{Twenty Pretty Edge Coloring Conjectures} in the book of Stiebitz, Scheide, Toft, and Favrholdt~\cite[p. 258]{StiebSTF-Book}. 
	 It states the following.

	 \begin{CON}[Overfull Conjecture]\label{overfull-con}
	 	Let $G$ be a graph with $\Delta(G)>\frac{1}{3}|V(G)|$. Then
	 	$\chi'(G)=\Delta(G)$ if and only if $G$ contains no $\Delta(G)$-overfull
	 	subgraph.
	 \end{CON}
	 
	 The only known example demonstrating the sharpness of the degree condition
	 $\Delta(G)>\tfrac{1}{3}|V(G)|$ is the graph $P^*$, obtained from the Petersen
	 graph by deleting one vertex. We have $\Delta(P^*)=\tfrac{1}{3}|V(P^*)|$ and
	 $\chi'(P^*)=4$, yet $P^*$ contains no $3$-overfull subgraph.
	 
	 Using Edmonds' matching polytope theorem, Seymour~\cite{seymour79} showed that
	 the existence of  a $\Delta(G)$-overfull subgraph  in a graph
	 $G$ can be decided in time polynomial in $|V(G)|$. Independently,
	 Niessen~\cite{MR1814514} proved in 2001 that, for graphs $G$ with
	 $\Delta(G)>\tfrac{1}{3}|V(G)|$, there are at most three induced
	 $\Delta(G)$-overfull subgraphs, and that one can be found in polynomial time in
	 $|V(G)|$. Moreover, in this same paper,  Niessen proved that when $\Delta(G)\ge \tfrac{1}{2}|V(G)|$, there is at most one
	 induced $\Delta(G)$-overfull subgraph, and it can be identified in linear time in
	 $|V(G)|+|E(G)|$.
	 
	 Consequently, if the Overfull Conjecture holds, then the NP-complete problem of
	 determining the chromatic index becomes polynomial-time solvable for graphs $G$
	 with $\Delta(G)>\tfrac{1}{3}|V(G)|$. The Overfull Conjecture would also imply
	 several other longstanding conjectures in edge coloring. These include the
	 Just-overfull Conjecture~\cite[Conjecture~4.23, p.72]{StiebSTF-Book}, the
	 Vertex-splitting Conjecture~\cite[Conjecture~1]{MR1460574}, as well as Vizing's
	 Independence Number, $2$-Factor, and Average Degree Conjectures
	 ~\cite{vizing-2factor,vizing-ind} when restricted to graphs $G$ with
	 $\Delta(G)>\tfrac{1}{3}|V(G)|$. The latter three conjectures also appear
	 in the aforementioned list of Stiebitz, Scheide, Toft, and Favrholdt~\cite[pp. 257--258]{StiebSTF-Book}.

	 The well-known \emph{1-Factorization Conjecture}, first stated explicitly
	 in~\cite{MR772711} and often traced back to Dirac in the early 1950s, was
	 verified for all sufficiently large graphs by Csaba, K\"uhn, Lo, Osthus, and
	 Treglown~\cite{MR3545109} in 2016. 
	 
	 \begin{CON}[1-Factorization Conjecture]\label{con:1-factorization}
	 	Let $G$ be a graph of even order. If $G$ is $k$-regular for some
	 	$k\ge 2\lceil |V(G)|/4\rceil-1$, then $G$ is 1-factorable; equivalently,
	 	$\chi'(G)=\Delta(G)$.
	 \end{CON}

	 This conjecture is a special case of the
	 Overfull Conjecture. Indeed, if $G$ is a  $k$-regular graph of even order with 
  $ k\ge 2\lceil |V(G)|/4\rceil-1$,  since every  subset
	 $X\subseteq V(G)$  with $|X|$ odd satisfies
	 $|E(G[X])|\le k \lfloor |X|/2\rfloor$, $G$ contains no $k$-overfull subgraph. 
	 Then the Overfull Conjecture implies $\chi'(G)=k$ and so $G$ has a 1-factorization since 
	 $G$ is regular.

	 Much of the progress on the Overfull Conjecture has followed the perspective of
	 the 1-Factorization Conjecture, whose formulation may be viewed as imposing a
	 minimum-degree condition within the framework of the Overfull Conjecture.
	 Accordingly, most existing results rely on a minimum-degree assumption; see, for
	 example, \cite{MR4394718,MR4563210,MR4735535}. By contrast, considerably less is
	 known when no minimum-degree assumption is imposed.
	 
	 If one restricts only the maximum degree of $G$, the strongest earlier result was
	 obtained by Chetwynd and Hilton in 1988~\cite{MR975994}, who proved that every
	 graph $G$ with $\Delta(G)\ge |V(G)|-3$ satisfies the Overfull Conjecture. More
	 recently, Shan~\cite{overfull-I}  showed that,
	 for sufficiently large $|V(G)|$, the Overfull Conjecture holds for graphs $G$
	 with $\Delta(G)\ge \tfrac{13}{14}|V(G)|$. In this paper, we prove the conjecture
	 for robust expanders satisfying certain density  constraints. As a consequence,
	  the conjecture holds for large graphs $G$ whose maximum degree is at least  $(1+\ve)|V(G)|/2$ for any $0<\ve<1$. 
	 
	 In order to state our main result, we define robust expanders. Given
	 $0<\nu\le \tau<1$, we say that a graph $G$ on $n$ vertices is a
	 \emph{robust $(\nu,\tau)$-expander} if, for every $S\subseteq V(G)$ with
	 $\tau n\le |S|\le (1-\tau)n$, the number of vertices  of $G$ that have at least
	 $\nu n$ neighbors in $S$ is at least $|S|+\nu n$. The \emph{$\nu$-robust
	 	neighborhood} $RN_{\nu,G}(S)$ is the set of all vertices of $G$ with at least
	 $\nu n$ neighbors in $S$. Note that, for
	 $0<\nu'\le \nu\le \tau\le \tau'<1$, every robust $(\nu,\tau)$-expander is also a
	 robust $(\nu',\tau')$-expander.

	We say that a graph $G$ satisfies \emph{weak VAL} if
	$d_G(u)+d_G(v)\ge \Delta(G)+2$ for every edge $uv\in E(G)$. The terminology is motivated by Vizing's Adjacency Lemma, which will be introduced later. 
	
	For a multigraph $G$ and a given $0<\eta <1$, let $$U_\eta(G)=\{v\in V(G): \Delta(G)-|N_G(v)|> \eta |V(G)|\}.$$
	Our main result is the following.

	 \begin{THM}\label{thm:robust-expander}
	 	For every $\alpha>0$, there exists $\tau=\tau(\alpha)>0$ such that, for
	 	every $\nu>0$, there exists $n_0=n_0(\alpha,\nu,\tau)\in\mathbb{N}$ for
	 	which the following holds. Suppose that $G$ is a graph on $n\ge n_0$
	 	vertices satisfying:
	 	\begin{enumerate}[(i)]
	 		\item $\Delta(G)\ge \alpha n$ and $G$ satisfies weak VAL;
	 		\item for some $0<\eta\ll \nu$,   we have
	 		$|U_\eta(G)|<\eta^2 n$;
	 		\item $G$ is a robust $(\nu,\tau)$-expander.
	 	\end{enumerate}
	 	Then $\chi'(G)=\Delta(G)$ if and only if $G$ contains no
	 	$\Delta(G)$-overfull subgraph.
	 \end{THM}
	 
	 As an application of Theorem~\ref{thm:robust-expander}, we obtain the following
	 result.
	 
	 \begin{THM}\label{thm:1}
	 	For every $0<\varepsilon<1$, there exists $n_0\in\mathbb{N}$ for which the
	 	following holds. If $G$ is a graph on $n\ge n_0$ vertices with
	 	$\Delta(G)\ge \frac{1}{2}(1+\varepsilon)n$, then
	 	$\chi'(G)=\Delta(G)$ if and only if $G$ contains no
	 	$\Delta(G)$-overfull subgraph.
	 \end{THM}
	 
	 We will use the standard hierarchy notation $0<a\ll b\le 1$. More precisely, if
	 we claim that a statement holds whenever $0<a\ll b\le 1$, then there exists a
	 non-decreasing function $f:(0,1]\to(0,1]$ such that the statement holds for all
	 $0<a\le f(b)$. Intuitively, this means that, for every fixed $b>0$, the
	 statement holds provided that $a$ is sufficiently small compared to $b$.
	 
	 The remainder of this paper is organized as follows. In the next section, we
	 prove Theorem~\ref{thm:1} by applying Theorem~\ref{thm:robust-expander}, and we
	 give an outline of the proof of Theorem~\ref{thm:robust-expander}. In
	 Section~3, we introduce the necessary notation and preliminaries. In Section~4,
	 we reformulate Theorem~\ref{thm:robust-expander} into an equivalent statement
	 better suited to our proof techniques. This equivalent statement  is proved in the final section.

	\section{Proof of Theorem~\ref{thm:1} and  Proof  Outline of  Theorem~\ref{thm:robust-expander}}\label{sec2}

\subsection{Proof of Theorem~\ref{thm:1}}

We first present several preliminary results that will be needed in the proof.

K\"uhn and Osthus~\cite{KO2014} showed that graphs with large minimum degree are robust expanders.

\begin{LEM}[{\cite[Lemma 3.8]{KO2014}}]\label{lem:expander}
	Suppose that $0<\nu\le \tau\le \varepsilon<1$ and $\varepsilon\ge 4\nu/\tau$. Let $G$ be a graph on $n$ vertices with minimum degree $\delta(G)\ge (1+\varepsilon)n/2$. Then $G$ is a robust $(\nu,\tau)$-expander.
\end{LEM}

The following lemma was stated for  robust outexpanders  due to  Gir\~ao, Granet, K\"uhn, Lo, and Osthus~\cite{MR4550147};
here we use its undirected analogue. It 
states that robust expansion is preserved under the removal of few edges and under the addition or removal of few vertices. 

\begin{LEM}[{\cite[Lemma 4.2]{MR4550147}}]\label{lem:stability-expansion}
	Let $0<\varepsilon\le \nu\le \tau\le 1$, and let $G$ be a robust $(\nu,\tau)$-expander on $n$ vertices.
	\begin{enumerate}[(a)]
		\item If $G'$ is obtained from $G$ by removing at most $\varepsilon n$ edges at each vertex, then $G'$ is a robust $(\nu-\varepsilon,\tau)$-expander.
		\item Suppose that $\tau\ge (1+2\tau)\varepsilon$. If $G'$ is obtained from $G$ by adding or removing at most $\varepsilon n$ vertices, then $G'$ is a robust $(\nu-\varepsilon,2\tau)$-expander.
	\end{enumerate}
\end{LEM}

A graph $G$ is \emph{edge-chromatic critical} if $\chi'(G)=\Delta(G)+1$ but every proper subgraph $H$ of $G$ satisfies $\chi'(H)\le \Delta(G)$. An edge $e\in E(G)$ is \emph{critical} if $\chi'(G)=\Delta(G)+1$ but $\chi'(G-e)= \Delta(G)$. Clearly, a graph is edge-chromatic critical if and only if it is connected and all of its edges are critical. The following two results describe structural properties of edge-chromatic critical graphs; the second one is due to Cao and Chen~\cite{MR4130909}.

\begin{LEM}[Vizing's Adjacency Lemma (VAL), \cite{Vizing-2-classes}]\label{lem:val}
	Let $G$ be a class~2 graph with maximum degree $\Delta$. If $e=xy$ is a critical edge of $G$, then $x$ has at least $\Delta-d_G(y)+1$ neighbors of degree $\Delta$ in $V(G)\setminus\{y\}$. As a consequence, $d_G(x)+d_G(y)\ge \Delta+2$.
\end{LEM}

\begin{LEM}[{\cite[Theorem 3]{MR4130909}}]\label{Thm:average}
	For any $0<\ve<1$, there exists a constant $\Delta_0$ such that if $G$ is an edge-chromatic critical graph with $\Delta(G)\ge \Delta_0$, then the average degree of $G$ is at least $(1-\ve)\Delta(G)$. Moreover, if $\ve<2/3$, then $\Delta_0\le (16/\ve^3)^{8/\ve}$.
\end{LEM}

\begin{proof}[Proof of Theorem~\ref{thm:1} assuming Theorem~\ref{thm:robust-expander}]
	Let $\tau^*=\tau(1/2)$ be defined as in Theorem~\ref{thm:robust-expander}.
	Choose $n_0\in\mathbb{N}$ and $\eta$ such that
	\[
	0<1/n_0\ll \eta\ll \nu\ll \tau\ll \tau^*, \ve<1,
	\]
	with $n_0$ large enough that $(1+\ve)n_0/2$ exceeds the constant $\Delta_0$
	in Lemma~\ref{Thm:average} with parameter $\eta^3$.
	
	If $G$ contains a $\Delta(G)$-overfull subgraph, then $\chi'(G)=\Delta(G)+1$.
	Hence suppose $G$ contains no $\Delta(G)$-overfull subgraph, and suppose for
	a contradiction that $\chi'(G)=\Delta(G)+1$.

By deleting edges or vertices if necessary, we may replace $G$ with an
edge-chromatic critical subgraph having the same maximum degree $\Delta(G)$.
This replacement preserves $\Delta(G)$ and creates no $\Delta(G)$-overfull
subgraph. Since an edge-chromatic critical graph has at least $\Delta(G)+1$
vertices, by choosing the original $|V(G)|$ sufficiently large, we may relabel
this subgraph as $G$ and still assume that $n:=|V(G)|\ge n_0$ and
$\Delta(G)\ge (1+\ve)n/2$.

	By Lemma~\ref{lem:val}, $G$ satisfies weak VAL. Since $\Delta(G)\ge(1+\ve)n/2$
	exceeds $\Delta_0$, Lemma~\ref{Thm:average} gives average degree at least
	$(1-\eta^3)\Delta(G)$, so $\sum_{v\in V(G)}(\Delta(G)-d_G(v))\le
	\eta^3\Delta(G)\cdot n<\eta^3n^2$. Since every vertex of $U_\eta(G)$ has a degree-deficit
	at least $\eta n$, we get $|U_\eta(G)|<\eta^2 n$.
	
	Let $G^*=G-U_\eta(G)$. Then $\delta(G^*)\ge\Delta(G)-\eta n-|U_\eta(G)|
	\ge(1+\ve/2)n/2$, and since $|V(G^*)|\le n$, also
	$\delta(G^*)\ge(1+\ve/2)|V(G^*)|/2$. By the hierarchy $\ve/2\ge 4\nu/\tau$,
	$G^*$ is a robust $(\nu,\tau)$-expander by Lemma~\ref{lem:expander}.
	
	Since $|V(G^*)|\ge(1-\eta^2)n$, we have
	$|V(G)\setminus V(G^*)|<2\eta^2|V(G^*)|$. The hierarchy $\eta\ll\tau$ gives
	$\tau\ge(1+2\tau)2\eta^2$, so Lemma~\ref{lem:stability-expansion}(b) with
	$2\eta^2$ in place of $\varepsilon$ shows $G$ is a robust
	$(\nu-2\eta^2,2\tau)$-expander, and since $2\tau\le\tau^*$, in particular a
	robust $(\nu-2\eta^2,\tau^*)$-expander. Theorem~\ref{thm:robust-expander}
	then gives $\chi'(G)=\Delta(G)$, contradicting $\chi'(G)=\Delta(G)+1$.
\end{proof}

	For the general Conjecture~\ref{overfull-con}, Theorem~\ref{thm:robust-expander}
	rules out counterexamples that are robust expanders, after the standard reduction
	to a critical graph. Indeed, let $G$ be an $n$-vertex graph with
	$\Delta(G)>n/3$, where $n$ is sufficiently large, and suppose that $G$ is a
	minimal counterexample to Conjecture~\ref{overfull-con}. Thus $G$ contains no
	$\Delta(G)$-overfull subgraph, $\chi'(G)=\Delta(G)+1$, and $G$ is
	edge-chromatic critical. By Vizing's Adjacency Lemma, $G$ satisfies weak VAL.
	Moreover, Lemma~\ref{Thm:average} implies that, for every fixed $\eta>0$, we
	have $|U_\eta(G)|<\eta^2 n$. Hence, for any
	choice of parameters satisfying $0<\eta\ll \nu\le \tau<1$, if such a graph $G$
	is a robust $(\nu,\tau)$-expander, then Theorem~\ref{thm:robust-expander}
	implies $\chi'(G)=\Delta(G)$, a contradiction. Consequently, any sufficiently
	large minimal counterexample to Conjecture~\ref{overfull-con} with
	$\Delta(G)>n/3$ must fail to be a robust $(\nu,\tau)$-expander for the relevant
	choice of parameters.

	\subsection{Proof Outline of Theorem~\ref{thm:robust-expander}}

	Let $G$ be a robust $(\nu,\tau)$-expander satisfying the hypotheses of
	Theorem~\ref{thm:robust-expander}. If $G$ contains a $\Delta(G)$-overfull
	subgraph, then $\chi'(G)=\Delta(G)+1$. Thus we may assume that $G$ contains no
	$\Delta(G)$-overfull subgraph, and suppose for a contradiction that
	$\chi'(G)=\Delta(G)+1$.
	
A natural approach is to construct a $\Delta(G)$-regular super-multigraph $G'$
of $G$ such that $\chi'(G)\le \chi'(G')$, and then prove that
$\chi'(G')=\Delta(G')$, thereby obtaining a contradiction. The main difficulty
in applying this approach directly to $G$ is that the resulting multigraph $G'$
may contain a \emph{fat triangle}, that is, a triangle whose total number of
edges is roughly $\Delta(G')$ and whose edge multiplicities are roughly
$\Delta(G')/3$. For instance, this can occur when $|V(G)|$ is even, $G$ has
exactly three vertices of degree $\lceil \Delta(G)/3\rceil$, and all other
vertices have maximum degree. Such a triangle creates a serious risk of
producing overfull subgraphs during the edge-decomposition process, and is quite
challenging to handle; see, for example, the work of the third author
in~\cite{overfull-I}. Moreover, the approach of~\cite{overfull-I} does not
transfer directly to the present setting, since here the minimum degree may be
as low as a constant.
	
	In this paper, we avoid this difficulty by using the following three-stage
	strategy. Assuming that $G$ contains no $\Delta(G)$-overfull subgraph, we
	decompose $E(G)$ into $\Delta(G)$ edge-disjoint matchings.
	
	\emph{\bf Stage One: Trimming.}
	
	We apply a tool developed by Gir\~ao,
	Granet, K\"uhn, Lo, and Osthus~\cite[Lemma~7.3]{MR4550147}. We trim $G$ to a
	robust expander $G^*$ that contains no $\Delta(G^*)$-overfull subgraph. 
	 More precisely, we remove edge-disjoint matchings, each of
	which covers all vertices of maximum degree in the current graph and also
	reduces the number of edges incident with vertices of small degree. The main
	challenge in this trimming process is to ensure that no overfull subgraph with
	the same maximum degree as the current graph is created. Thus this stage ends
	either when all small-degree vertices have been eliminated, or when we find a
	dense subgraph that is ``almost''  overfull, whichever occurs first.
	
	\emph{\bf Stage Two: Regularization.}
	
 Let $G^*$ be the  graph resulting  from the trimming stage. If $G^*$ has no
	almost-overfull subgraph, then we add at most one vertex and suitable edges to
	$G^*$ to obtain a $\Delta(G^*)$-regular multigraph $G^{**}$ of even order. If
	$G^*$ contains an almost-overfull subgraph, then we carry out the same
	regularization either directly, or after identifying some small-degree vertices
	into one or two new vertices; in one case, we also remove some further matchings
	before regularizing.
	
\emph{\bf{Stage Three: Edge Decomposition.}}

In this stage we construct a decomposition of $G^{**}$ into $\Delta(G^{**})$ 1-factors. Such a decomposition, together with the matchings removed in the previous stages, gives a decomposition of $E(G)$ into $\Delta(G)$ edge-disjoint matchings, contradicting $\chi'(G)=\Delta(G)+1$.

We briefly describe the construction. Applying the regularity lemma, we obtain a balanced bipartition of $V(G^{**})$ into parts $A$ and $B$ such that $G^{**}[A,B]$ is a bipartite robust expander, and such that each vertex has approximately the same number of neighbors in its own part as in the other part. We first color $G^{**}[A]$ and $G^{**}[B]$, together with some cross edges between $A$ and $B$, using about $\Delta(G^{**})/2$ colors. We then use the robust expansion between $A$ and $B$ to correct this coloring so that each color class becomes a $1$-factor. Next, we color the few edges inside $A$ or inside $B$ that become uncolored during the correction process, and extend each resulting color class to a $1$-factor. The remaining uncolored edges now  lie only between $A$ and $B$ and form a regular bipartite multigraph. By the classical theorem of K\"onig~\cite{MR1511872}, this remaining multigraph has a $1$-factorization. Combining these three collections of $1$-factors gives a $1$-factorization of $G^{**}$, as required.

	\section{Notation and Preliminaries}\label{section:3}

	Let $G$ be a multigraph. We use $V(G)$ and $E(G)$ to denote the vertex set
	and edge set of $G$, respectively, and write $e(G)=|E(G)|$. 
	For $v\in V(G)$, let $N_G(v)$ be the set of neighbors of $v$ in $G$. The degree
	of $v$ in $G$, denoted by $d_G(v)$, is the number of edges of $G$ incident with
	$v$. We write $d_G^s(v)=|N_G(v)|$ and call this the \emph{simple degree} of
	$v$ in $G$. A vertex $v$ of $G$ is called \emph{simple} in $G$ if
	$d_G(v)=d_G^s(v)$. A neighbor $u$ of $v$ is called a \emph{simple neighbor} of
	$v$ if there is exactly one edge joining $u$ and $v$. We denote by $N_G^s(v)$ the set of simple neighbors of
	$v$. Note that it is possible to have $d_G^s(v)>|N_G^s(v)|$. Let $G^s$ denote the
	underlying simple graph of $G$. Thus $d_{G^s}(v)=d_G^s(v)$ for every
	$v\in V(G)$.

	Let $V_1,V_2\subseteq V(G)$ be disjoint vertex sets. We write $E_G(V_1,V_2)$ for the set of edges of $G$ with one endpoint in $V_1$ and the other in $V_2$, and let $e_G(V_1,V_2)=|E_G(V_1,V_2)|$. If $V_1=\{v\}$, we write $E_G(v,V_2)$ and $e_G(v,V_2)$ instead of $E_G(\{v\},V_2)$ and $e_G(\{v\},V_2)$. We also write $G[V_1,V_2]$ for the bipartite subgraph of $G$ with vertex set $V_1\cup V_2$ and edge set $E_G(V_1,V_2)$.

Let $S\subseteq V(G)$ and $v\in V(G)$. We write $N_G(v,S)=N_G(v)\cap S$, $d_G(v,S)=e_G(v,S\setminus\{v\})$, and $d_G^s(v,S)=|N_G(v,S)|$. The subgraph of $G$ induced by $S$ is denoted by $G[S]$. We write $G-S=G[V(G)\setminus S]$, and use the abbreviation $G-x$ for $G-\{x\}$. We also write $e_G(S)=e(G[S])$.

	If $F\subseteq E(G)$, then $G-F$ denotes the multigraph obtained from $G$ by deleting all edges in $F$, and $G[F]$ denotes the subgraph of $G$ induced by the edge set $F$.  We write $V(F)$ for the set of vertices covered by the edges in $F$.
	For a multiset $F$ of edges of the complete graph on $V(G)$, let $G+F$ denote the multigraph obtained from $G$ by adding all edges in $F$, with multiplicities.
	We also use the abbreviation $G+e$ for $G+\{e\}$.

	We present in the rest of this section some preliminaries that will be needed in the proof of Theorem~\ref{thm:robust-expander}.
	
	\subsection{Degree Sequences}
	
	The following result of the third author will be used to construct a regular super-multigraph from a given multigraph by adding edges. Let $m\ge 2$ be an integer. A sequence of nonnegative non-increasing integers $(d_1,\ldots,d_m)$ is \emph{admissible} if $\sum_{i=1}^m d_i$ is even and $d_1\le \sum_{i=2}^m d_i$.
	
	\begin{LEM}[{\cite[Lemma 2.3]{overfull-I}}]\label{lem:graphical-biparite}
		Let $m\ge 2$, and let $(d_1,d_2,\ldots,d_m)$ be an admissible sequence. Then there is a bipartite multigraph $L$ on $\{v_1,\ldots,v_m\}$ and an even index $p\in [2,m]$ satisfying the following properties, where we set $d_i=0$ whenever $i\notin [1,m]$.
		\begin{enumerate}[(a)]
			\item $d_L(v_i)=d_i-d_{i+1}$ and $d_L(v_{i+1})=0$ for every odd $i\in [1,p]$.
			\item If $p<m$, then $d_L(v_{p+1})\le d_{p+1}$; moreover, $d_L(v_i)=d_i$ for every $i\in [p+2,m]$.
			\item Either $\{v_1,\ldots,v_p\}$ and $\{v_{p+1},\ldots,v_m\}$, or $\{v_1,\ldots,v_p,v_{p+1}\}$ and $\{v_{p+2},\ldots,v_m\}$ form a bipartition of $L$.
		\end{enumerate}
		Furthermore, there is a polynomial-time algorithm that finds $L$ and $p$.
	\end{LEM}
	
	We will also need the following result on realizing a bipartite multigraph with a prescribed degree sequence and additional constraints.
	
	\begin{LEM}\label{lem:graphical-biparite2}
		Let $m\ge 2$, and let $((c_1,d_1),(c_2,d_2),\ldots,(c_m,d_m))$ be a sequence of pairs of nonnegative integers. Suppose that $\sum_{i=1}^m c_i=\sum_{i=1}^m d_i$ and, for each $i\in [m]$, we have $c_i+d_i\le \sum_{j\ne i}(c_j+d_j)$. Then there exists a bipartite multigraph $L$ with bipartition $X=\{x_1,\ldots,x_m\}$ and $Y=\{y_1,\ldots,y_m\}$ such that
		\begin{enumerate}[(a)]
			\item $d_L(x_i)=c_i$ and $d_L(y_i)=d_i$ for each $i\in [m]$;
			\item $e_L(x_i,y_i)=0$ for each $i\in [m]$.
		\end{enumerate}
	\end{LEM}
	
	\begin{proof}
		Since $\sum_{i=1}^m c_i=\sum_{i=1}^m d_i$, there is a bipartite multigraph with the desired degrees. Indeed, give $x_i$ exactly $c_i$ half-edges and $y_i$ exactly $d_i$ half-edges, and pair the half-edges on the $X$-side arbitrarily with the half-edges on the $Y$-side.
		
		Among all such bipartite multigraphs, choose $L$ so that $\sum_{i=1}^m e_L(x_i,y_i)$ is minimum. We claim that this minimum is zero. Suppose not. Then there exists $i\in [m]$ such that $e_L(x_i,y_i)>0$. 
		Since $c_i+d_i\le \sum_{j\ne i}(c_j+d_j)$, and 
		\begin{align*}
		c_i+d_i &=2e_L(x_i,y_i)+e_L(\{x_i,y_i\},V(L)\setminus\{x_i,y_i\})\\
		\sum_{j\ne i}(c_j+d_j) &= 2e_L(X\setminus\{x_i\},Y\setminus\{y_i\})+e_L(\{x_i,y_i\},V(L)\setminus\{x_i,y_i\}), 
		\end{align*}
		 it follows that $e_L(X\setminus\{x_i\},Y\setminus\{y_i\}) \ge  e_L(x_i,y_i)>0$.
		
		Hence there is an edge $e\in E_L(x_p,y_q)$ with $p\ne i$ and $q\ne i$. Choose an edge $f\in E_L(x_i,y_i)$, and replace the two edges $f$ and $e$ by $x_i y_q$ and $x_p y_i$. This preserves the degree of every vertex. Moreover, it removes one edge between $x_i$ and $y_i$ and creates no new edge between $x_j$ and $y_j$ for any $j\ne i$, since $p\ne i$ and $q\ne i$. Therefore $\sum_{j=1}^m e_L(x_j,y_j)$ decreases, contradicting the choice of $L$.
		
		Thus $e_L(x_i,y_i)=0$ for every $i\in [m]$, and the lemma follows.
	\end{proof}

	\subsection{Edge Coloring and Related Concepts}

	Let $G$ be a multigraph, and let $W\subseteq V(G)$ be an independent set. We
	say that $G$ satisfies the \emph{Hall property} with respect to $W$ if, for
	every edge $wv\in E(G)$ with $w\in W$, we have $e_G(v,W)<d_G(w)$.
	
	Notice that if $W$ contains no vertex of maximum degree, then Vizing's Adjacency Lemma (VAL) implies that
	$G$ satisfies the Hall property with respect to $W$. In the proof of
	Theorem~\ref{thm:robust-expander}, we will take $W$ to be a set of vertices of
	degree less than $\Delta(G)/3$, and try to eliminate the vertices of $W$ by
	removing edge-disjoint matchings that saturate both the vertices of maximum
	degree in the current graph and the vertices of $W$. Lemma~\ref{lem:matching-in-critical-graph}
	below shows that, as long as $G$ satisfies the Hall property with respect to
	$W$, the bipartite graph $G[W,V(G)\setminus W]$ contains two edge-disjoint
	matchings, each of which saturates $W$.
	
	The proof uses the following result, which follows from Hall's Theorem and is
	listed as Exercise~16.2.13 in the book of Bondy and Murty~\cite{MR2368647}.
	
	\begin{LEM}\label{lem:matching-in-bipartite-graph}
		Let $G=G[X,Y]$ be a bipartite graph such that $d_G(x)\ge 1$ for every
		$x\in X$, and $d_G(x)\ge d_G(y)$ for every edge $xy\in E(G)$ with
		$x\in X$. Then $G$ has a matching saturating $X$.
	\end{LEM}
	
	\begin{LEM}\label{lem:matching-in-critical-graph}
		Let $G$ be a graph, and let $W\subseteq V(G)$ be an independent set such
		that $d_G(w)\ge 1$ for every $w\in W$. If $G$ satisfies the Hall property
		with respect to $W$, then $G$ has two edge-disjoint matchings, each of
		which saturates $W$.
	\end{LEM}
	
	\begin{proof}
		Let $H=G[W,V(G)\setminus W]$. We first show that $H$ has a matching
		saturating $W$. Let $wv\in E(H)$ with $w\in W$ and
		$v\in V(G)\setminus W$. Since $W$ is independent, $d_H(w)=d_G(w)$. By the
		Hall property, $d_H(v)=e_G(v,W)<d_G(w)=d_H(w)$. In particular, since
		$d_H(v)\ge 1$, we have $d_H(w)\ge 2$ for every $w\in W$.
		
		By Lemma~\ref{lem:matching-in-bipartite-graph}, $H$ has a matching
		$M_1$ saturating $W$. Let $H'=H-M_1$. For every edge $wv\in E(H')$ with
		$w\in W$ and $v\in V(G)\setminus W$, we have
		$d_{H'}(v)\le d_H(v)\le d_H(w)-1=d_{H'}(w)$. Moreover,
		$d_{H'}(w)=d_H(w)-1\ge 1$ for every $w\in W$. Applying
		Lemma~\ref{lem:matching-in-bipartite-graph} to $H'$, we obtain a matching
		$M_2$ saturating $W$. Since $M_2\subseteq E(H')=E(H)\setminus M_1$, the
		matchings $M_1$ and $M_2$ are edge-disjoint.
	\end{proof}
	
	In the 1960s, Gupta~\cite{Gupta-67} and, independently, Vizing~\cite{Vizing-2-classes}
	proved the following upper bound on the chromatic index of a multigraph.
	K\"onig~\cite{MR1511872} determined the exact value of the chromatic index for
	bipartite multigraphs.
	
	\begin{THM}[\cite{Gupta-67,Vizing-2-classes}]\label{thm:chromatic-index}
		Every multigraph $G$ satisfies $\chi'(G)\le \Delta(G)+\mu(G)$.
	\end{THM}
	
	\begin{THM}[\cite{MR1511872}]\label{konig}
		Every bipartite multigraph $G$ satisfies $\chi'(G)=\Delta(G)$.
	\end{THM}
	
	An edge $k$-coloring of a multigraph $G$ is said to be \emph{equalized} if each
	color class contains either $\lfloor e(G)/k\rfloor$ or $\lceil e(G)/k\rceil$
	edges. McDiarmid~\cite{MR300623} observed the following result.
	
	\begin{THM}\label{lem:equa-edge-coloring}
		Let $G$ be a multigraph with chromatic index $\chi'(G)$. Then, for every
		$k\ge \chi'(G)$, there is an equalized edge-coloring of $G$ with $k$
		colors.
	\end{THM}

	Let $k\in\mathbb{N}$, and let $\CC^k(G)$ denote the set of all edge $k$-colorings of a multigraph $G$. For $\varphi\in \CC^k(G)$ and $v\in V(G)$, the set of colors \emph{present} at $v$ is
	\[
	\varphi(v)=\{\varphi(e): \text{$e\in E(G)$ is incident with $v$}\},
	\]
	and the set of colors \emph{missing} at $v$ is $\pbar(v)=[k]\setminus \varphi(v)$. For $X\subseteq V(G)$ and $i\in [k]$, define $\pbar_X^{-1}(i)=\{v\in X:i\in \pbar(v)\}$. We write simply $\pbar^{-1}(i)$ for $\pbar_{V(G)}^{-1}(i)$.
	
	The technique of multifans is widely used in edge-coloring problems. It was introduced in the book of Stiebitz, Scheide, Toft, and Favrholdt~\cite{StiebSTF-Book} as a generalization of Vizing fans. Let $G$ be a multigraph, and let $\varphi$ be a partial edge $k$-coloring of $G$. A \emph{multifan} centered at $r$ with respect to $\varphi$ is a sequence
	\[
	F_\varphi(r,s_1:s_p)=(r,e_1,s_1,e_2,s_2,\ldots,e_p,s_p),
	\]
	where $p\ge 1$, the edges $e_1,\ldots,e_p$ are distinct, $e_i=rs_i$ for each $i\in [p]$, the edge $e_1$ is uncolored under $\varphi$, and for every $i\in [2,p]$, there exists $j\in [i-1]$ such that $\varphi(e_i)\in \pbar(s_j)$.
	
	The following result on multifans can be found in~\cite[Theorem~2.1]{StiebSTF-Book}.
	
	\begin{LEM}[{\cite[Theorem~2.1]{StiebSTF-Book}}]\label{lem:ele-fan}
		Let $G$ be a multigraph with $\chi'(G)=k+1$ for some $k\ge \Delta(G)$, let $e=rs_1\in E(G)$, and let $\varphi\in \CC^k(G-e)$. If $F_\varphi(r,s_1:s_p)$ is a multifan, then $\pbar(u)\cap \pbar(v)=\emptyset$ for any distinct vertices $u,v\in \{r,s_1,\ldots,s_p\}$.
	\end{LEM}

	We will need the following two consequences of Lemma~\ref{lem:ele-fan}.
	
	\begin{LEM}\label{lemma:reduce-to-weak-VAL}
		Let $G$ be a class~2 graph. If there exists an edge $uv\in E(G)$ such that $v$ is adjacent to fewer than $\Delta(G)-d_G(u)+1$ vertices of maximum degree in $V(G)\setminus\{u\}$, then $G-uv$ is class~2 and $\Delta(G-uv)=\Delta(G)$.
	\end{LEM}
	
	\begin{proof}
		Let $\Delta=\Delta(G)$, and suppose not. Then $\chi'(G-uv)\le \Delta$: indeed, this is immediate if $G-uv$ is class~1, while if $G-uv$ is class~2 and $\Delta(G-uv)<\Delta$, then $\Delta(G-uv)=\Delta-1$ and Theorem~\ref{thm:chromatic-index} gives $\chi'(G-uv)\le \Delta$. Let $\varphi$ be an edge $\Delta$-coloring of $G-uv$, viewed as a partial edge-coloring of $G$ in which the edge $uv$ is uncolored.
		
		Let $F=F_\varphi(v,u:s_p)=(v,e_1,s_1,e_2,s_2,\ldots,e_p,s_p)$ be a maximal  multifan centered at $v$, where $e_1=vu$ and $s_1=u$. By maximality of $F$, every color missing at one of $s_1,\ldots,s_p$ appears on one of the colored fan edges $e_2,\ldots,e_p$. Hence 
		$\bigcup_{i=1}^p \pbar(s_i)\subseteq \{\varphi(e_i):2\le i\le p\}$ and so 
$\big|\bigcup_{i=1}^p \pbar(s_i)\big|\le p-1$.
		
		On the other hand, by Lemma~\ref{lem:ele-fan}, the sets $\pbar(s_1),\ldots,\pbar(s_p)$ are pairwise disjoint. Since $s_1=u$, we have $|\pbar(s_1)|=\Delta-d_G(u)+1$. Moreover, by assumption, among the vertices $s_2,\ldots,s_p$, fewer than $\Delta-d_G(u)+1$ have degree $\Delta$ in $G$. Thus at least $p-(\Delta-d_G(u)+1)$ of the vertices $s_2,\ldots,s_p$ have degree less than $\Delta$, and each of them misses at least one color.  It follows that $\big|\bigcup_{i=1}^p \pbar(s_i)\big|\ge p$, a contradiction. Therefore $G-uv$ is class~2 and $\Delta(G-uv)=\Delta(G)$.
	\end{proof}

	\begin{LEM}\label{lem:multi-version-star-multigraph}
		Let $G$ be a multigraph and let $x\in V(G)$. Then
		$\chi'(G)\le \Delta(G)+\mu(G-x)$.
	\end{LEM}
	
	\begin{proof}
		Let $\mu=\mu(G-x)$. If $\mu=0$, then $G-x$ has no edges, and so all edges of
		$G$ are incident with $x$. Hence $\chi'(G)=d_G(x)\le \Delta(G)$, and the result
		follows. Thus we may assume that $\mu\ge 1$.
		
		Suppose, to the contrary, that $\chi'(G)>\Delta(G)+\mu$. Choose such a
		counterexample with $|E(G)|$ minimum. Then every edge of $G$ is critical; that
		is, $\chi'(G-e)<\chi'(G)$ for every edge $e\in E(G)$. Let
		$k=\chi'(G)-1$. Then $k\ge \Delta(G)+\mu$.
		
		First suppose that there exists a vertex $u\in V(G)\setminus\{x\}$ such that
		$e_G(u,x)\le \mu$. Let $v\in N_G(u)$, and let $e=uv$. Since $e$ is critical,
		$G-e$ has an edge $k$-coloring $\varphi$. Let
		$F=F_\varphi(u,v:s_p)=(u,e_1,s_1,e_2,s_2,\ldots,e_p,s_p)$
		be a maximal multifan with respect to $e$ and $\varphi$, centered at $u$, where
		$e_1=e=uv$ and $s_1=v$.
		
		By the maximality of $F$, every color missing at one of
		$s_1,\ldots,s_p$ appears on an edge from $u$ to one of
		$s_1,\ldots,s_p$. Hence $\bigcup_{i=1}^p \pbar(s_i)\subseteq
		\{\varphi(f): f\in E_{G-e}(u,\{s_1,\ldots,s_p\})\}$. 
		For every $i\in[p]$, we have $e_G(u,s_i)\le \mu$: this is true by the definition
		of $\mu=\mu(G-x)$ if $s_i\ne x$, and by the choice of $u$ if $s_i=x$. Since the
		edge $e=us_1$ is uncolored, it follows that $|\{\varphi(f): f\in E_{G-e}(u,\{s_1,\ldots,s_p\})\}|\le p\mu-1$. 
		On the other hand, by Lemma~\ref{lem:ele-fan}, the sets
		$\pbar(s_1),\ldots,\pbar(s_p)$ are pairwise disjoint. Since $e$ is uncolored and
		incident with $s_1$, we have
		$|\pbar(s_1)|\ge k-(d_G(s_1)-1)\ge \mu+1$. For each $i\in[2,p]$, we have
		$|\pbar(s_i)|\ge k-d_G(s_i)\ge \mu$. Therefore 
		$\left|\bigcup_{i=1}^p \pbar(s_i)\right|\ge (\mu+1)+(p-1)\mu=p\mu+1$, 
		a contradiction.
		
		Thus every vertex $u\in V(G)\setminus\{x\}$ satisfies $e_G(u,x)>\mu$. In
		particular, $V(G)=N_G(x)\cup\{x\}$. Write
		$N_G(x)=\{y_1,\ldots,y_t\}$, and let $d_i=e_G(x,y_i)$ for each $i\in[t]$.
		Let $H$ be the spanning subgraph of $G$ obtained by keeping all edges of $G-x$
		and exactly $\mu$ edges between $x$ and each $y_i$. Then $\mu(H)\le \mu$ and
		$\Delta(H)\le t\mu$: indeed, $d_H(x)=t\mu$, while for each $i\in[t]$,
		$d_H(y_i)\le \mu+(t-1)\mu=t\mu$. By Theorem~\ref{thm:chromatic-index},
		$\chi'(H)\le \Delta(H)+\mu(H)\le (t+1)\mu$.
		
		The edges in $E(G)\setminus E(H)$ are all incident with $x$, and 
		$|E(G)\setminus E(H)| =\sum_{i=1}^t (d_i-\mu)=d_G(x)-t\mu$. 
		Assign a distinct new color to each of these edges. Together with an edge
		coloring of $H$, this gives an edge coloring of $G$ using at most
		\[
		(t+1)\mu+d_G(x)-t\mu=d_G(x)+\mu\le \Delta(G)+\mu
		\]
		colors, contradicting the choice of $G$. Therefore
		$\chi'(G)\le \Delta(G)+\mu(G-x)$.
	\end{proof}

Let $G$ be a multigraph. For $X\subseteq V(G)$, let
\[
\df_G(X)=\sum_{x\in X}(\Delta(G)-d_{G[X]}(x)),
\]
and write $\df_G(H)$ for $\df_G(V(H))$ for an induced subgraph $H$ of $G$. We also write $\df(G)$ for
$\df_G(G)$. Thus $\df_G(X)=\Delta(G)|X|-2e_G(X)$.

Recall that $G$ is overfull if $e(G)>\Delta(G)\lfloor |V(G)|/2\rfloor$, and 
 that a subgraph $H$ of $G$ is
$\Delta(G)$-overfull if $e(H)>\Delta(G)\lfloor |V(H)|/2\rfloor$.
Since $e(H)\le \Delta(G)|V(H)|/2$ whenever $|V(H)|$ is even, a
$\Delta(G)$-overfull subgraph must have odd order. Thus, if $H$ is
$\Delta(G)$-overfull, then $2e(H)\ge \Delta(G)(|V(H)|-1)+2$. Since
$e_G(V(H))\ge e(H)$, it follows that
\begin{equation}\label{eqn:equi-def-overfull}
	\df_G(H)=\Delta(G)|V(H)|-2e_G(V(H))\le \Delta(G)-2.
\end{equation}

Consequently, if $H$ is a subgraph of $G$ with $\df_G(H)\ge \Delta(G)-1$, then
$H$ is not $\Delta(G)$-overfull. We will use this observation frequently.

Since $\df(G)=\Delta(G)|V(G)|-2e(G)$, when $|V(G)|$ is odd, the quantities
$\df(G)$ and $\Delta(G)$ have the same parity. Thus we have the following
simple consequence.

\begin{LEM}\label{lem:def-relation-with-Delta}
	Let $G$ be a multigraph of odd order. Then $\df(G)=\Delta(G)+2i$ for some
	$i\in \mathbb{Z}$.
\end{LEM}

\begin{LEM}\label{lem:overfull-delete-one-vertex}
	Let $G$ be a multigraph of even order, and let $\Delta=\Delta(G)$. If there
	exists $v\in V(G)$ such that $G-v$ is $\Delta$-overfull, then
	$d_G(v)=\delta(G)$.
\end{LEM}

\begin{proof}
	Let $u\in V(G)\setminus\{v\}$. Since $G-v$ is $\Delta$-overfull, we have
	\[
	\sum_{x\in V(G)\setminus\{v\}}(\Delta-d_{G-v}(x))\le \Delta-2.
	\]
	On the other hand,
	\[
	\begin{aligned}
		\sum_{x\in V(G)\setminus\{v\}}(\Delta-d_{G-v}(x))
		&=\sum_{x\in V(G)\setminus\{v\}}(\Delta-d_G(x)+e_G(x,v))\\
		&=d_G(v)+\sum_{x\in V(G)\setminus\{v\}}(\Delta-d_G(x))\\
		&\ge d_G(v)+\Delta-d_G(u).
	\end{aligned}
	\]
	Hence $d_G(v)+\Delta-d_G(u)\le \Delta-2$, and so
	$d_G(v)\le d_G(u)-2$. Since this holds for every
	$u\in V(G)\setminus\{v\}$, we have $d_G(v)=\delta(G)$.
\end{proof}

Given an edge-coloring of $G$ and a color $i$, the edges colored $i$ form a
matching, and hence the number of vertices not incident with an edge of color
$i$ has the same parity as $|V(G)|$. We will use the following Parity Lemma of
Gr\"unewald and Steffen~\cite[Lemma~2.1]{MR2028248}.

\begin{LEM}[Parity Lemma]\label{lem:parity-lemma}
	Let $G$ be a multigraph, and let $\varphi\in \CC^k(G)$ for some integer
	$k\ge \Delta(G)$. Then $|\pbar^{-1}(i)|\equiv |V(G)| \pmod 2$ for every
	color $i\in [k]$.
\end{LEM}

\subsection{Robust Expanders}

Let $G$ be a robust $(\nu,\tau)$-expander satisfying the hypotheses of
Theorem~\ref{thm:robust-expander}. To remove edge-disjoint sparse spanning configurations from $G$ 
while ensuring that the remaining graph is still a robust expander, we will use
a result of Gir\~ao, Granet, K\"uhn, Lo, and Osthus~\cite{MR4550147}. The
relevant notation was defined for multidigraphs in~\cite{MR4550147}; here we
state the corresponding definitions for multigraphs.

We first recall the notation needed for this result.
Let $G$ be a multigraph on vertex set $V$, where $|V|=n$. Let $L$ be a multiset
consisting of paths on $V$ and isolated vertices. We write $V(L)$ for the set of
vertices that are isolated in $L$ or lie on a non-trivial path in $L$, and
$E(L)$ for the multiset of edges contained in the paths of $L$. For
$v\in V(L)$, let $d_L(v)$ be the number of edges in $E(L)$ incident with $v$.
Let $F$ be a multiset of edges.   We say that $(L,F)$ is a \emph{layout} if $F$ is a submultiset of $E(L)$ and
$E(L)\setminus F\ne \emptyset$. When $F$ is given as a
multigraph, we also say that $(L,F)$ is a layout, meaning that $(L,E(F))$ is a
layout.

Given an edge $uv$, we say that a $(u,v)$-path has \emph{shape $uv$}. Let
$(L,F)$ be a layout on $V$. A multigraph $\mathcal{H}$ on $V$ is a
\emph{spanning configuration of shape $(L,F)$} if $\mathcal{H}$ can be
decomposed into internally vertex-disjoint paths $\{P_e:e\in E(L)\}$ such that
each $P_e$ has shape $e$, $P_f=f$ for every $f\in F$, and
$\bigcup_{e\in E(L)}V^0(P_e)=V\setminus V(L)$, where $V^0(P)$ denotes the set
of internal vertices of a path $P$. In particular, the last condition implies
that the isolated vertices of $L$ remain isolated in $\mathcal{H}$.

Let $\varepsilon,\delta>0$. We say that a graph $G$ is
\emph{$(\delta,\varepsilon)$-almost regular} if
$d_G(v)=(\delta\pm \varepsilon)n$ for every $v\in V(G)$.

The following result of Gao and the third author~\cite{2405.18494} is an
adaptation to robust expanders of Lemma~7.3 for robust outexpanders due to
Gir\~ao, Granet, K\"uhn, Lo, and Osthus~\cite{MR4550147}.

\begin{LEM}[{\cite[Theorem 2.12]{2405.18494}}]\label{lem:removing-linear-forests}
	Let
	$0<1/n\ll \ve^*\le \ve\ll \nu'\ll \nu\ll \tau\ll \alpha,\eta^*\ll 1$,
	and let $\ell\le (\alpha-\eta^*)n/2$. Suppose that $G$ is an
	$(\alpha,\ve^*)$-almost regular robust $(\nu,\tau)$-expander on $n$ vertices.
	Suppose also that, for each $i\in [\ell]$, $F_i$ is a multiset of edges on
	$V(G)$. Assume that $(L_1,F_1),\ldots,(L_\ell,F_\ell)$ are layouts with
	$V(L_i)\subseteq V(G)$ for each $i\in [\ell]$, and that the following hold,
	where $L=\bigcup_{i\in [\ell]}L_i$.
	\begin{enumerate}[(a)]
		\item For each $i\in [\ell]$, we have $|V(L_i)|\le \varepsilon^2 n$ and
		$|E(L_i)|\le \varepsilon^4 n$.
		\item For each $v\in V(G)$, we have $d_L(v)\le \varepsilon^3 n$, and
		there are at most $\varepsilon^2 n$ indices $i\in [\ell]$ such that
		$v\in V(L_i)$.
	\end{enumerate}
	Then there exist edge-disjoint submultigraphs
	$\mathcal{H}_1,\ldots,\mathcal{H}_\ell\subseteq G+\bigcup_{i\in [\ell]}F_i$
	such that each $\mathcal{H}_i$ is a spanning configuration of shape
	$(L_i,F_i)$, and such that
	$G-\bigcup_{i=1}^{\ell}E(\mathcal{H}_i)$ is a robust $(\nu'/2,4\tau)$-expander.
\end{LEM}

Following~\cite{MR3299598}, for an integer $k\ge 2$, we say that a graph $G$ is
\emph{Hamilton $k$-linked} if, whenever
$x_1,y_1,\ldots,x_k,y_k$ are distinct vertices of $G$, there exist
vertex-disjoint paths $Q_1,\ldots,Q_k$ such that $Q_i$ connects $x_i$ and $y_i$
for each $i\in [k]$, and together the paths $Q_1,\ldots,Q_k$ cover all vertices
of $G$. We will use the following result of K\"uhn, Lo, Osthus, and
Staden~\cite{MR3299598}.

\begin{LEM}[{\cite[Corollary 6.9(ii)]{MR3299598}}]\label{lem:hamiltonicity-of-expander}
	Let $n,k\in \mathbb{N}$, and suppose that
	$0<1/n\ll \nu\ll \tau\ll \alpha<1$ and $k\le \nu^4 n$. If $H$ is a robust
	$(\nu,\tau)$-expander on $n$ vertices with $\delta(H)\ge \alpha n$, then $H$
	is Hamilton $k$-linked.
\end{LEM}

The notion of a bipartite robust expander was introduced by K\"uhn, Lo, Osthus,
and Staden in~\cite{MR3299598}, where it was defined in a one-sided form. Here
we use a two-sided variant, requiring the robust expansion condition to hold
from each side of the bipartition.

Let $0<\nu\le \tau<1$. Suppose that $H$ is a bipartite graph with bipartition $(A,B)$. For
$C\in\{A,B\}$, let $D$ be the other part. For $S\subseteq C$, let
\[
RN_{\nu,H}(S)=\{x\in D:d_H(x,S)\ge \nu |C|\}.
\]
We say that $H$ is a \emph{bipartite robust $(\nu,\tau)$-expander with
	bipartition $(A,B)$} if, for each choice of $C\in\{A,B\}$, every set
$S\subseteq C$ with $\tau |C|\le |S|\le (1-\tau)|C|$ satisfies
\[
|RN_{\nu,H}(S)|\ge |S|+\nu |D|.
\]
In our applications, the bipartition $(A,B)$ is chosen so that $|A|=|B|$.
The following lemma is a bipartite analogue of Lemma~\ref{lem:stability-expansion}.

\begin{LEM}[Stability of bipartite robust expansion]\label{lem:bipartite-stability}
	Let $0<\nu\le\tau<1$ and let $H$ be a bipartite robust $(\nu,\tau)$-expander
	with balanced bipartition $(A,B)$, where $|A|=|B|=m$. Let
	\[
	0<\beta\le\min\Bigl\{\tfrac{\nu}{4},\,\tfrac{\tau}{4},\,\tfrac{1-2\tau}{2}\Bigr\}.
	\]
	\begin{enumerate}[(a)]
		\item If $H'$ is obtained from $H$ by deleting at most $\beta m$ edges at
		each vertex, then $H'$ is a bipartite robust $(\nu-\beta,\tau)$-expander
		with bipartition $(A,B)$. \label{lem:bipartite-stability-a}
		
		\item Let $H'$ be a bipartite graph with bipartition $(A',B')$ obtained
		from $H$ by adding and/or deleting at most $\beta m$ vertices in each part,
		in the following sense: there are sets $W_A\subseteq A$, $W_B\subseteq B$
		and $W_A'\subseteq A'$, $W_B'\subseteq B'$ with
		$|W_A|,|W_B|,|W_A'|,|W_B'|\le\beta m$ such that
		$A\setminus W_A=A'\setminus W_A'$ and $B\setminus W_B=B'\setminus W_B'$, and
		such that $H$ and $H'$ induce the same bipartite graph on
		$(A\setminus W_A,\,B\setminus W_B)$. Then $H'$ is a bipartite robust
		$(\nu-3\beta,2\tau)$-expander with bipartition $(A',B')$.
		\label{lem:bipartite-stability-b}
		
		\item Let $H'$ be the bipartite graph with bipartition $(A',B')$ whose edge
		set consists of all edges of $H$ with one endpoint in $A'$ and the other in
		$B'$, where $(A',B')$ is obtained from $(A,B)$ by relocating at most
		$\beta m$ vertices from $A$ to $B$ and at most $\beta m$ vertices from $B$
		to $A$, so that $|A'|=|B'|=m$. Then $H'$ is a bipartite robust
		$(\nu-3\beta,\tau+2\beta)$-expander with bipartition $(A',B')$.
		\label{lem:bipartite-stability-swap}
	\end{enumerate}
\end{LEM}

\begin{proof}
	In every part we verify the expansion condition only for subsets of the
	$A$-side; the argument for the $B$-side is symmetric. Recall that for a set
	$S$ contained in one part $C$, with $D$ the other part, the robust
	neighborhood is $RN_{\nu,H}(S)=\{x\in D:d_H(x,S)\ge\nu|C|\}$, and the
	expansion requirement reads $|RN_{\nu,H}(S)|\ge|S|+\nu|D|$. The hypothesis
	$\beta\le\min\{\nu/4,\tau/4,(1-2\tau)/2\}$ guarantees in particular that
	$\nu-3\beta>0$, $2\tau<1$, and $\tau+2\beta<1$, so all three conclusions are
	nonvacuous.
	
	\medskip
	\noindent\textbf{(a)} Here $A,B$ and $m$ are unchanged. Let $S\subseteq A$
	with $\tau m\le|S|\le(1-\tau)m$. For every $x\in RN_{\nu,H}(S)$ we have
	$d_H(x,S)\ge\nu m$; since at most $\beta m$ edges are deleted at $x$,
	\[
	d_{H'}(x,S)\ge d_H(x,S)-\beta m\ge(\nu-\beta)m=(\nu-\beta)|A|,
	\]
	so $x\in RN_{\nu-\beta,H'}(S)$. Hence
	$RN_{\nu,H}(S)\subseteq RN_{\nu-\beta,H'}(S)$, and therefore
	\[
	|RN_{\nu-\beta,H'}(S)|\ge|RN_{\nu,H}(S)|\ge|S|+\nu m\ge|S|+(\nu-\beta)m,
	\]
	as required.
	
	\medskip
	\noindent\textbf{(b)} Write $a=|A'|$ and $b=|B'|$; by hypothesis
	$ (1-\beta)m \le a,b \le (1+\beta)m$. Let $S'\subseteq A'$ with
	$2\tau a\le|S'|\le(1-2\tau)a$, and put
	\[
	S=S'\cap(A\setminus W_A)=S'\setminus W_A',
	\]
	so that $|S|\ge|S'|-|W_A'|\ge|S'|-\beta m$. Every vertex of $S$ lies in
	$A\setminus W_A$ and therefore keeps its part in both $H$ and $H'$.
	
	We first check that $S$ is admissible for the expansion of $H$ at level
	$\tau$. For the lower bound,
	\[
	|S|\ge 2\tau a-\beta m\ge 2\tau(1-\beta)m-\beta m
	=\bigl(\tau+\tau(1-2\beta)-\beta\bigr)m\ge\tau m,
	\]
	since $\beta\le\tau/4$ gives $\tau(1-2\beta)\ge\beta$. For the upper bound,
	\[
	|S|\le|S'|\le(1-2\tau)a\le(1-2\tau)(1+\beta)m\le(1-\tau)m,
	\]
	using $\beta(1-2\tau)\le\tau$. Thus $\tau m\le|S|\le(1-\tau)m$, and the robust
	expansion of $H$ gives
	\[
	|RN_{\nu,H}(S)|\ge|S|+\nu m.
	\]
	
	Now let $x\in RN_{\nu,H}(S)$ with $x\in B\setminus W_B$. Then $x$ keeps its
	part in both $H$ and $H'$, and so do all vertices of $S\subseteq A\setminus W_A$;
	since $H$ and $H'$ induce the same bipartite graph on
	$(A\setminus W_A,\,B\setminus W_B)$, the edges between $x$ and $S$ are common
	to $H$ and $H'$. Hence
	\[
	d_{H'}(x,S')\ge d_{H'}(x,S)=d_H(x,S)\ge\nu m.
	\]
	The relevant threshold in $H'$ is $(\nu-3\beta)a$, and
	\[
	(\nu-3\beta)a\le(\nu-3\beta)(1+\beta)m\le(\nu-2\beta)m\le\nu m,
	\]
	so $x\in RN_{\nu-3\beta,H'}(S')$. Thus every vertex of $RN_{\nu,H}(S)$ lying in
	$B\setminus W_B$ belongs to $RN_{\nu-3\beta,H'}(S')$. The vertices of
	$RN_{\nu,H}(S)$ outside $B\setminus W_B$ lie in $W_B$, of which there are at
	most $\beta m$. Therefore
	\[
	|RN_{\nu-3\beta,H'}(S')|
	\ge|RN_{\nu,H}(S)|-|W_B|
	\ge|S|+\nu m-\beta m.
	\]
	Finally, using $|S'|\le|S|+\beta m$ and
	$(\nu-3\beta)b\le(\nu-3\beta)(1+\beta)m\le(\nu-2\beta)m$,
	\[
	|S'|+(\nu-3\beta)b
	\le|S|+\beta m+(\nu-2\beta)m
	=|S|+\nu m-\beta m
	\le|RN_{\nu-3\beta,H'}(S')|.
	\]
	Hence $H'$ is a bipartite robust $(\nu-3\beta,2\tau)$-expander with
	bipartition $(A',B')$.

	\medskip
	\noindent\textbf{(c)} The vertex set and the part sizes are unchanged:
	$|A'|=|B'|=m$. Let $R_A\subseteq A$ be the set of vertices relocated from $A$
	to $B$, and $R_B\subseteq B$ the set relocated from $B$ to $A$, so that
	$|R_A|,|R_B|\le\beta m$,
	\[
	A'=(A\setminus R_A)\cup R_B
	\qquad\text{and}\qquad
	B'=(B\setminus R_B)\cup R_A.
	\]
	By the definition of $H'$, all edges of $H$ between $A\setminus R_A$ and
	$B\setminus R_B$ are kept in $H'$. Thus
	\begin{equation}\label{eqn:c-cross-equality}
		d_{H'}(x,T)=d_H(x,T)
		\qquad\text{for all $x\in B\setminus R_B$ and all $T\subseteq A\setminus R_A$,}
	\end{equation}
	and symmetrically $d_{H'}(x,T)=d_H(x,T)$ for all $x\in A\setminus R_A$ and
	all $T\subseteq B\setminus R_B$.
	
	We verify the expansion condition for subsets of $A'$; the argument for
	subsets of $B'$ is symmetric.
	
	Let $S'\subseteq A'$ with $(\tau+2\beta)m\le |S'|\le (1-\tau-2\beta)m$, and set
	\[
	S=S'\cap (A\setminus R_A).
	\]
	Then $S'\setminus S\subseteq R_B$, so $|S'|-\beta m\le |S|\le |S'|$, and hence
	\[
	|S|\ge(\tau+2\beta)m-\beta m\ge\tau m,
	\qquad
	|S|\le |S'|\le(1-\tau-2\beta)m\le(1-\tau)m.
	\]
	Thus $S$ is admissible for the expansion of $H$, and
	\[
	|RN_{\nu,H}(S)|\ge |S|+\nu m.
	\]
	
	Fix $x\in RN_{\nu,H}(S)\setminus R_B$. Since $S\subseteq A\setminus R_A$ and
	$x\in B\setminus R_B$, \eqref{eqn:c-cross-equality} gives, using $S\subseteq S'$,
	\[
	d_{H'}(x,S')\ge d_{H'}(x,S)=d_H(x,S)\ge\nu m
	\ge(\nu-3\beta)m=(\nu-3\beta)|A'|.
	\]
	Therefore $x\in RN_{\nu-3\beta,H'}(S')$. Hence every vertex of
	$RN_{\nu,H}(S)\setminus R_B$ belongs to $RN_{\nu-3\beta,H'}(S')$, and so
	\[
	|RN_{\nu-3\beta,H'}(S')|
	\ge |RN_{\nu,H}(S)|-|R_B|
	\ge |S|+\nu m-\beta m.
	\]
	Finally, since $|S'|\le |S|+\beta m$ and $|B'|=m$,
	\[
	|S'|+(\nu-3\beta)|B'|
	\le |S|+\beta m+(\nu-3\beta)m
	=|S|+\nu m-2\beta m
	\le |S|+\nu m-\beta m
	\le |RN_{\nu-3\beta,H'}(S')|.
	\]
	Hence $H'$ is a bipartite robust $(\nu-3\beta,\tau+2\beta)$-expander with
	bipartition $(A',B')$.
\end{proof}

	\section{A Reformulation of Theorem~\ref{thm:robust-expander}}

	In this section, we reduce Theorem~\ref{thm:robust-expander} to a statement
	about regular multigraphs. The reduction has two parts. First, starting from a
	class~2 counterexample $G$, we trim $G$ by removing carefully chosen sparse
	subgraphs. The layouts and spanning configurations from
	Lemma~\ref{lem:removing-linear-forests} are used only in this trimming step:
	they realize, inside the robust expander, the auxiliary edge patterns needed
	for the iterative deletion process, while leaving a robust expander behind.
	Second, after the trimming process stops, we regularize the remaining graph, or
	a slight modification of it, by adding a multiset of edges chosen according to
	its deficiency sequence. This is where Lemma~\ref{lem:graphical-biparite}
	enters.
	
	The regularized  object is a multigraph whose underlying simple graph remains a
	robust expander and whose multiple edges are sparse away from at most one
	exceptional vertex. The regularization is completed through the following
	construction, where the multigraph $D$ is obtained from a reduced graph of $G$
	arising from the trimming stage.
	
	\begin{DEF}[Canonical completion]\label{def:canonical-completion}
		Let $D$ be a multigraph on $\{v_1,\ldots,v_m\}$, where $m$ is even, and let
		$\Delta_0=\Delta(D)$. Suppose that $d_D(v_1)\le d_D(v_2)\le\cdots\le d_D(v_m)$, 
		set $d_i=\Delta_0-d_D(v_i)$ for each $i\in[m]$, and assume that
		$(d_1,\ldots,d_m)$ is admissible.
		Let $L$ be the bipartite
		multigraph obtained by applying Lemma~\ref{lem:graphical-biparite} to
		$(d_1,\ldots,d_m)$, and let $p$ be the even index given by that lemma. Let
		$D^+=D\cup L$.
		
		We define a multigraph $N$ on $\{v_1,\ldots,v_m\}$ as follows. If $p=m$,
		then, for every $i\in[m/2]$, add
		$\Delta_0-d_{D^+}(v_{2i})$ edges between $v_{2i-1}$ and $v_{2i}$ to $N$.
		
		Suppose now that $p<m$. Then
		$\Delta_0-d_{D^+}(v_{p+1})$ is even. Let
		\[
		q=\frac{\Delta_0-d_{D^+}(v_{p+1})}{2}.
		\]
		Add $q$ edges between $v_{p+1}$ and $v_p$, add $q$ edges between
		$v_{p+1}$ and $v_{p-1}$, and add
		$\Delta_0-d_{D^+}(v_p)-q$ edges between $v_p$ and $v_{p-1}$. Moreover, for
		each $i\in[p/2-1]$, add $\Delta_0-d_{D^+}(v_{2i})$ edges between
		$v_{2i-1}$ and $v_{2i}$ to $N$.
		
		We call $D\cup L\cup N$ and $D^+\cup N$ a  \emph{canonical completion} of
		$D$ and $D^+$, respectively. By construction, the canonical completion of $D^+$ is
		$\Delta_0$-regular.
	\end{DEF}

	The lemma below will be used in Section~\ref{Section:5} when we extend a  partial edge coloring. 
	\begin{LEM}\label{lem:v1-v2-edges}
	Let $G=D\cup L\cup N$ be a canonical completion of $D$ described above.
	Suppose that $0<1/m\ll\eta\ll1$ and that, for every $i\in[3,m]$,
	$|N_G^s(v_i)|
	\ge \Delta_0-e_G(v_1,v_i)-6\eta m$ and  $e_D(v_1,v_i)\le2\eta^2m$. 
	Then  $	e_G(v_1,v_2)\ge e_G(v_1,v_i)-7\eta m$ for every $i\in [3,m]$. 
	\end{LEM}
	
 \begin{proof}
 	By Lemma~\ref{lem:graphical-biparite},  
 	we have 
 	\[
 	d_{D^+}(v_1)=d_{D^+}(v_2)=d_D(v_2).
 	\]
Suppose  first that $p\ge4$. By the construction of $N$, no edge of $N$ joins
  $v_1$ to a vertex $v_i$ with $i\ge3$. Moreover, the number of edges of $L$
  between $v_1$ and $v_i$ is at most $d_L(v_i)\le\Delta_0-d_D(v_i)$. Hence,
  for every $i\in[3,m]$,
  \[
  \begin{aligned}
  	e_G(v_1,v_i)
  	&\le \Delta_0-d_D(v_i)+e_D(v_1,v_i) \\
  	&\le \Delta_0-d_D(v_2)+2\eta^2m \\
  	&\le e_G(v_1,v_2)+2\eta^2m< e_G(v_1,v_2)+7\eta m.
  \end{aligned}
  \]	
Suppose now that $p=2$. 
Using the notation from
Definition~\ref{def:canonical-completion}, the construction of $N$ gives
\[
\begin{aligned}
	e_G(v_1,v_2)
	&=e_{D^+}(v_1,v_2)+\Delta_0-d_{D^+}(v_2)-q \\
	&= e_D(v_1,v_2)+ \Delta_0-d_D(v_2)-q.
\end{aligned}
\]
By the assumption $|N_G^s(v_3)|
 \ge  \Delta_0-e_G(v_1,v_3)- 6\eta m$, we get $q\le 6\eta m$. 
Consequently, for every $i\in[3,m]$,
\[
\begin{aligned}
	e_G(v_1,v_2)
	&\ge \Delta_0-d_D(v_2)-q \\
	&\ge \Delta_0-d_D(v_2)-6\eta m \\
	&\ge \Delta_0-d_D(v_i)-6\eta m \\
	&\ge e_G(v_1,v_i)-e_D(v_1,v_i)-6\eta m \\
	&\ge e_G(v_1,v_i)-7\eta m,
\end{aligned}
\]
where the last inequality follows from
$e_D(v_1,v_i)\le2\eta^2m\le\eta m$. Thus
the statement holds in both cases.
 \end{proof}

The following theorem is the regular-multigraph form that
will be proved in Section~\ref{Section:5}.

\begin{THM}\label{thm:robust-expander2}
	Let $n_0$ be a positive integer, and let $\eta,\nu,\tau,\alpha$ be positive
	constants such that $0<1/n_0\ll \eta\ll \nu\le \tau\ll \alpha<1$. Let $G$
	be a regular multigraph on $n\ge n_0$ vertices, where $n$ is even, and suppose
	that $n>\Delta(G)\ge \alpha n$. Suppose that $G$ satisfies the following
	conditions.
		\begin{enumerate}[(1)]
		\item $G$ is the edge-disjoint union of three multigraphs $D,L$, and $N$,
		all on the vertex set $\{v_1,\ldots,v_n\}$, such that
		$\Delta(D)=\Delta(G)$ and $G=D\cup L\cup N$ is a  canonical completion
		of $D$  defined in  Definition~\ref{def:canonical-completion}. In
		addition, the following hold. \label{item:thm4.1}
		\begin{enumerate}[(a)]

			\item $e_D(v_1,v_i) \le 2\eta^2 n$ for each $i\in [3,n]$. \label{item:thm4.1.a-new}
			\item $d_G^s(v_i)\ge \frac{i-2}{i-1}\Delta(G)-6\eta n$ for every
			$i\in [3,n]$. \label{item:thm4.1.b}
			
			\item $|N_G^s(v_2)|\ge \Delta(G)-e_G(v_1,v_2)-13\eta n$, and
			$|N_G^s(v_i)|\ge \Delta(G)-e_G(v_1,v_i)-6\eta n$ for every
			$i\in [3,n]$. \label{item:thm4.1.c}

		\end{enumerate}
		
		\item The underlying simple graph of $G$ is a robust $(\nu,\tau)$-expander.
		\label{item:thm4.2}
	\end{enumerate}
	Then $\chi'(G)=\Delta(G)$ if and only if $G$ contains no
	$\Delta(G)$-overfull subgraph.
\end{THM}

\begin{proof}[Proof of Theorem~\ref{thm:robust-expander} assuming Theorem~\ref{thm:robust-expander2}]
	We choose an additional parameter $\nu'$   such that
	\[
	0<  1/n_0\ll \eta\ll \nu'\ll\nu\le \tau\ll \alpha<1.
	\]
	
	Let $G$ be a graph on $n\ge n_0$ vertices satisfying the following conditions:
	\begin{enumerate}[(i)]
		\item $\Delta(G)\ge \alpha n$ and $G$ satisfies weak VAL;
		\item $|U_\eta(G)|<\eta^2 n$;
		\item $G$ is a robust $(\nu,\tau)$-expander.
	\end{enumerate}
	We prove that $\chi'(G)=\Delta(G)$ if and only if $G$ contains no
	$\Delta(G)$-overfull subgraph. Recall that
	$U_\eta(G)=\{v\in V(G):\Delta(G)-d^s_G(v)> \eta n\}$, and define
	\[
	W(G)=\{w\in U_\eta(G):d_G(w)<\Delta(G)/3-\eta n\}.
	\]
	
	If $G$ contains a $\Delta(G)$-overfull subgraph, then
	$\chi'(G)=\Delta(G)+1$. Hence we may assume that $G$ contains no
	$\Delta(G)$-overfull subgraph, and suppose for a contradiction that
	$\chi'(G)=\Delta(G)+1$. Let $\Delta=\Delta(G)$.

We reduce and modify $G$
into a multigraph satisfying the conditions of Theorem~\ref{thm:robust-expander2} through two stages. In both stages, whenever edges are deleted, any vertices that become isolated are immediately discarded, and we keep the same notation for the resulting graph and for the intersections of the relevant vertex sets with its vertex set. Thus every vertex at the beginning of a subsequent round has positive degree.
	
	\begin{center}
		{\bf \noindent Stage I:  Trimming}
	\end{center}
	
	We will iteratively remove edge-disjoint sets of edges. Each round is designed
	to remove two edges incident with every vertex of maximum degree in the current
	graph, while also controlling the degrees of the vertices in $U_\eta(G)$. The
	goal is to reduce $G$, after identifying some vertices in certain cases, to a
	graph that serves as the underlying simple graph of the multigraph $D$ in
	Theorem~\ref{thm:robust-expander2}.
	
	Since $G$ satisfies weak VAL, every edge $uv\in E(G)$ satisfies
	$d_G(u)+d_G(v)\ge \Delta+2$. On the other hand, if $u,v\in W(G)$, then
	$d_G(u),d_G(v)<\Delta/3-\eta n$, and hence
	$d_G(u)+d_G(v)<\Delta+2$. It follows that $W(G)$ is independent in $G$.
	
	We first aim to remove prescribed edge sets that, whenever possible, cover the
	current vertices in $W(G)$ while also reducing the degree of every vertex of
	maximum degree by two. Ideally, this process gradually reduces the degrees of
	the vertices in $W(G)$ relative to the current maximum degree, eventually
	eliminating all such vertices and leaving a graph whose maximum degree is at
	least $\Delta/3+2\eta n$.
	
	There are two possible obstructions to this ideal outcome. First, $W(G)$ may
	be empty initially, or all vertices of $W(G)$ may become isolated while the
	current maximum degree is still larger than $\Delta/3+2\eta n$ and the degrees
	of many other vertices are not within $O(\eta n)$ of the maximum degree. In
	this case, we switch to a different choice of vertices to be covered. Second,
	an overfull-type obstruction may appear in the current graph. To monitor this
	second possibility, we use the following notion.
	
	A subset $X\subseteq V(G)$ with $|X|$ odd is called \emph{almost-overfull} in
	$G$ if $\df_G(X)\le \Delta+7\eta^2 n$. The process terminates as soon as an
	almost-overfull set appears. We 
	initialize
	\[
	W_0=W(G), \quad U_0=U_\eta(G), \quad  E_0=\emptyset, \quad
	G_0=G, \quad  \text{and}\quad  G'_0=G_0-U_0.
	\]
	Also set
	\[
	d_0(v)=d_G(v) \quad \text{for every $v\in V(G)$}.
	\]
	
	We will apply Lemma~\ref{lem:removing-linear-forests} to remove cycles
	containing prescribed edges. The edges removed in each round will be
	edge-$2$-colorable, or equivalently, coverable by two edge-disjoint matchings.
	This is the main reason that the edges are selected carefully in the subsequent
	Operation~\ref{ope:edge-removing}.
	Since the graph changes from round to round, the degrees of the vertices and
	the set of vertices that must be covered also change. To keep track of this,
	degree changes will be recorded relative to the fixed reference graph $G'_0$.
	When a vertex of $U_0$ must later be represented inside $G'_0$, we project it
	onto suitable vertices of $G'_0$.

To formulate a notion of almost-overfullness relative to the auxiliary degree
records, for each $i\ge1$ and each current graph $G_{i-1}$, we call a set
$X\subseteq V(G_{i-1})$ with $|X|$ odd \emph{almost-overfull} if
\[
\sum_{v\in X}\bigl(\Delta-2(i-1)-d_{i-1}(v)\bigr)
+\sum_{u\in V(G_{i-1})\setminus X}d_{i-1}(u)
\le \Delta-2(i-1)+7\eta^2n.
\]
	Here the quantities $d_i(v)$ are the auxiliary degree records defined in Step~3
of Operation~\ref{ope:edge-removing}. They record the degree changes that will
occur after deleting the prescribed edges in $E_i$ and, later, after deleting a
suitable spanning cycle of $G'_0$ or a subgraph of $G'_0$ supplied by
Lemma~\ref{lem:removing-linear-forests}. At this stage, the actual edges of
that spanning cycle inside $G'_0$ have not yet been fixed.
Thus $d_{i-1}(v)$ should be thought of as the intended degree of $v$ after the first $i-1$ rounds have happened, rather than its actual degree in the current graph $G_{i-1}$.

To explain this definition, suppose first that the auxiliary degrees
$d_{i-1}(v)$ are the actual degrees in $G_{i-1}$. Writing
$k=\Delta-2(i-1)$, the left-hand side is
\[
k|X|-2e_{G_{i-1}}(X)
+2e_{G_{i-1}}\bigl(V(G_{i-1})\setminus X\bigr).
\]
In particular, if $V(G_{i-1})\setminus X$ is independent, then this is exactly
\[
\sum_{v\in X}\bigl(k-d_{G_{i-1}[X]}(v)\bigr),
\]
the $k$-deficiency of $G_{i-1}[X]$. Thus, apart from the error term
$7\eta^2n$, the definition recovers the usual notion of a $k$-overfull set.

When $V(G_{i-1})\setminus X$ is not independent, the additional term
$2e_{G_{i-1}}(V(G_{i-1})\setminus X)$ appears. Nevertheless, the definition
still gives an upper bound on the $k$-deficiency of $G_{i-1}[X]$, when the
auxiliary degrees are treated as actual degrees. When we apply this notion
later, we will show that $|V(G_{i-1})\setminus X|$ is small and that this set
spans only $O(\eta^2n)$ edges.

	We next define the exceptional situation in which the set $W_{i-1}$ is
	modified. Suppose that $i\ge 1$ and that $G_{i-1}$ has no almost-overfull set.
	We say that $G_{i-1}$ is in the \emph{critical case} if the following three
	conditions hold:
	\begin{itemize}
		\item $|W_{i-1}|=0$ and $|V(G'_{i-1})|$ is odd;
		\item $G'_{i-1}$ has at most one vertex $u$ with $d_{i-1}(u)< \Delta-2(i-1)$ and 
		$d_{i-1}(u) =\Delta-2(i-1)-1$ if it exists; 
		\item every vertex $z\in U_{i-1}$ satisfies
		$d_{i-1}(z)<\Delta/3-\eta n+2$.
	\end{itemize}
	
	If $G_{i-1}$ is in the critical case, we define an auxiliary set
	$W'_{i-1}\subseteq U_{i-1}$ as follows. If $|U_{i-1}|$ is odd, then
	$|U_{i-1}|\ge 3$, since $G_{i-1}$ has no almost-overfull set. In this case,
	let $W'_{i-1}=\{u,v\}$, where $u,v\in U_{i-1}$ are distinct vertices of minimum
	and second-minimum auxiliary degree; that is,
	$d_{i-1}(u)\le d_{i-1}(v)\le d_{i-1}(w)$ for every
	$w\in U_{i-1}\setminus\{u,v\}$. If $|U_{i-1}|$ is even, then
	$|U_{i-1}|\ge 2$, since $G_{i-1}$ has no almost-overfull set. In this case,
	let $W'_{i-1}=\{u\}$, where $u\in U_{i-1}$ has minimum auxiliary degree.
	
	Now define
	\[
	W^*_{i-1}=
	\begin{cases}
		W_{i-1}, & \text{if $G_{i-1}$ is not in the critical case},\\
		W'_{i-1}, & \text{if $G_{i-1}$ is in the critical case}.
	\end{cases}
	\]
	
	We will also use the following two deletion procedures. The second procedure
	uses a weaker condition when $W$ contains no vertex of maximum degree. Indeed,
	if $wu\in E_{H'}(W',V(H')\setminus W')$ with $w\in W'$ and
	$e_{H'}(u,W')\ge d_{H'}(w)$, then $u$ is adjacent in $H'$ to at most
	$\Delta(H')-e_{H'}(u,W')\le \Delta(H')-d_{H'}(w)$ vertices of maximum degree.
	However, we still need the first procedure in Operation~\ref{ope:edge-removing},
	where vertices are tracked by the auxiliary degree records in $G_{i-1}$ rather
	than by their actual degrees.
	
	\begin{algorithm}
		\SetAlgoLined
		\SetKwInOut{Input}{Input}
		\SetKwInOut{Output}{Output}
		\Input{A graph $H$ and an independent set $W\subseteq V(H)$}
		\Output{An updated graph $H'$ and an updated set $W'$}
		\BlankLine
		Let $H'=H$ and $W'=W$\;
		\Repeat{there is no edge $wu\in E_{H'}(W',V(H')\setminus W')$ with
			$w\in W'$ and $e_{H'}(u,W')\ge d_{H'}(w)$}{
			Choose an edge $wu\in E_{H'}(W',V(H')\setminus W')$ with $w\in W'$
			such that $e_{H'}(u,W')\ge d_{H'}(w)$\;
			Delete $wu$ from $H'$ and discard any isolated vertices of $H'$\;
			Update $W'=\{w\in W'\cap V(H'):d_{H'}(w)\ge 1\}$\;
		}
		\KwRet{$H'$ and $W'$}\;
		\caption{Iterative edge deletion with respect to an independent set}
		\label{alg:edge-deletion-wrt-W}
	\end{algorithm}
	
	\begin{algorithm}
		\SetAlgoLined
		\SetKwInOut{Input}{Input}
		\SetKwInOut{Output}{Output}
		\Input{A graph $H$ and an independent set $W\subseteq V(H)$ containing no vertex of maximum degree}
		\Output{An updated graph $H'$ and an updated set $W'$}
		\BlankLine
		Let $H'=H$ and $W'=W$\;
		\Repeat{there is no edge $wu\in E_{H'}(W',V(H')\setminus W')$ with
			$w\in W'$ such that $u$ is adjacent in $H'$ to fewer than
			$\Delta(H')-d_{H'}(w)+1$ vertices of maximum degree}{
			Choose an edge $wu\in E_{H'}(W',V(H')\setminus W')$ with $w\in W'$
			such that $u$ is adjacent in $H'$ to fewer than
			$\Delta(H')-d_{H'}(w)+1$ vertices of maximum degree\;
			Delete $wu$ from $H'$ and discard any isolated vertices of $H'$\;
			Update $W'=\{w\in W'\cap V(H'):d_{H'}(w)\ge 1\}$\;
		}
		\KwRet{$H'$ and $W'$}\;
		\caption{Iterative maximum-degree-neighbor deletion with respect to an independent set}
		\label{alg:edge-deletion-wrt-W2}
	\end{algorithm}
	
	Operation~\ref{ope:edge-removing} is a bookkeeping procedure for the trimming stage. In round $i$, the target maximum degree is lowered from $\Delta-2(i-1)$ to $\Delta-2i$, to be realized later by deleting the prescribed edges $E_i$ together with a sparse spanning configuration in $G'_0$. The auxiliary degrees $d_i(v)$ record these intended degree changes before the corresponding configurations have been fixed. The process stops either after the allowed number of rounds or when these degree records reveal an almost-overfull obstruction, at which point the proof moves to the regularization stage.
	
	\begin{Ope}[The Edge-Removal Process]\label{ope:edge-removing}
		For $i\ge 1$, proceed as follows. At the beginning of round $i$, if
		$G_{i-1}$ contains an almost-overfull set, then terminate the process.
		Otherwise define $W^*_{i-1}$ as above. If
		$W^*_{i-1}\subseteq W^*_{i-2}\subseteq \cdots \subseteq W^*_1\subseteq W_0$,
		then the process is performed only while $i\le \Delta/3-\eta n$. If this
		containment chain fails, then the process is performed only while
		$i\le \Delta/3-2\eta n$.
		
		\begin{enumerate}[Step 1:]
			\item \textbf{Preprocessing around $W^*_{i-1}$.}
			
			Initialize $R_i=\emptyset$ and set $G^*_{i-1}=G_{i-1}$. If
			$W^*_{i-1}$ is an independent set in $G_{i-1}$, apply
			Algorithm~\ref{alg:edge-deletion-wrt-W} to $G_{i-1}$ with respect to
			$W^*_{i-1}$. Let $R_i$ be the set of edges deleted by the algorithm, let
			$G^*_{i-1}$ be obtained from $G_{i-1}-R_i$ by discarding any isolated
			vertices, and replace $W^*_{i-1}$ by the subset of vertices that remain
			non-isolated in $G^*_{i-1}$.
			
			By the stopping rule of Algorithm~\ref{alg:edge-deletion-wrt-W} and
			Lemma~\ref{lem:matching-in-critical-graph}, there exist two edge-disjoint
			matchings of
			$G^*_{i-1}[W^*_{i-1},V(G^*_{i-1})\setminus W^*_{i-1}]$
			that each saturate $W^*_{i-1}$. If $W^*_{i-1}$ is not an independent set,
			the existence of these two matchings will be verified in
			Claim~\ref{claim:Gi-property} right after Operation~\ref{ope:edge-removing}.
			
			\item \textbf{Selecting the prescribed edges.}

			We now specify the special edges that will be required in the application of
			Lemma~\ref{lem:removing-linear-forests}. At the beginning of this step, set
			$M_i=N_i=N_i^*=\emptyset$, $I_i=\emptyset$, $h_i=0$, and
			$L_i^*=J_i^*=\emptyset$. The set $L_i^*$ records one or two vertices of
			positive deficiency that will be used as endvertices of a path, while
			$J_i^*$ and $I_i$ record vertices, if any, that will be excluded in the
			current round.

			\medskip
			\noindent\textbf{Edges related to vertices of $W^*_{i-1}$.}
			
			Let $M^1_{i-1}$ and $M^2_{i-1}$ be two edge-disjoint matchings of
			$G^*_{i-1}[W^*_{i-1},V(G^*_{i-1})\setminus W^*_{i-1}]$
			that each saturate $W^*_{i-1}$. If
			$W^*_{i-1}\subseteq W^*_{i-2}\subseteq \cdots \subseteq W_0$, then their
			existence is guaranteed by Step~1 and Lemma~\ref{lem:matching-in-critical-graph}
			since $W_0$ is an independent set. Otherwise, their existence will be shown
			in Claim~\ref{claim:Gi-property}.
			
			\CaseLabel{S2.1}
			Suppose either that $|W^*_{i-1}|$ is even and at least two, or that
			$|W^*_{i-1}|$ is odd and at least three. Also include the case in which
			$|W^*_{i-1}|=1$ and
			$\big|V(G'_{i-1})\cup V(M^1_{i-1}\cup M^2_{i-1})\big|$ is even.
			
			If $|W^*_{i-1}|$ is even, set $M_i=M^1_{i-1}$. If $|W^*_{i-1}|$ is odd,
			obtain $M_i$ from $M^1_{i-1}$ by adding one edge of $M^2_{i-1}$. See
			Figure~\ref{fig:S21-Mi-Ni} for an illustration of $M_i$ when
			$|W^*_{i-1}|$ is odd.
			
			Then the number of degree-one vertices of $G^*_{i-1}[M_i]$ lying in
			$V(M_i)\setminus W^*_{i-1}$ is even. Indeed, if $|W^*_{i-1}|$ is even,
			this number is exactly $|W^*_{i-1}|$. If $|W^*_{i-1}|$ is odd, then adding
			one edge of $M^2_{i-1}$ either introduces one new degree-one vertex outside
			$W^*_{i-1}$ or changes one old such vertex from degree one to degree two.
			In either case, the resulting number is even.
			
			Denote these degree-one vertices by
			$x^*_{i,1},y^*_{i,1},\ldots,x^*_{i,h_i},y^*_{i,h_i}$, where $h_i\ge 1$ by the assumption on the size of $W^*_{i-1}$, and set
			$A_i^*=\{x^*_{i,1},y^*_{i,1},\ldots,x^*_{i,h_i},y^*_{i,h_i}\}$.
			
			Let $N_i$ be a matching of
			$G^*_{i-1}[A_i^*\cap U_0,V(G'_0)\setminus (A_i^*\cap V(G'_0))]$
			that saturates $A_i^*\cap U_0$. For each $a\in A_i^*$, define its
			representative $\rho_i(a)\in V(G'_0)$ as follows. If $a\in V(G'_0)$, set
			$\rho_i(a)=a$. If $a\in U_0$, let $\rho_i(a)$ be the vertex matched to $a$
			by $N_i$. Then, for each $1\le r\le h_i$, set
			$x_{i,r}=\rho_i(x^*_{i,r})$ and $y_{i,r}=\rho_i(y^*_{i,r})$. The existence
			of $N_i$ is guaranteed by Claim~\ref{claim:Gi-property}\eqref{claim-Gi-property-5}. See
			Figure~\ref{fig:S21-Mi-Ni} for an illustration of $N_i$.
			
			\begin{figure}[ht]
				\centering
				\begin{tikzpicture}[
					scale=0.95,
					every node/.style={font=\small},
					Wnode/.style={circle,draw,fill=gray!15,minimum size=7mm,inner sep=0pt},
					Unode/.style={circle,draw,fill=orange!25,minimum size=7mm,inner sep=0pt},
					Gnode/.style={circle,draw,fill=green!25,minimum size=7mm,inner sep=0pt},
					Mone/.style={blue!75!black,very thick},
					Mtwo/.style={red!75!black,very thick},
					Nedge/.style={orange,very thick,dashed},
					Wbox/.style={draw=blue!45,fill=blue!8,rounded corners,inner sep=8pt},
					Ubox/.style={draw=orange!60!black,fill=orange!10,rounded corners,inner sep=8pt},
					Gbox/.style={draw=green!55!black,fill=green!10,rounded corners,inner sep=8pt}
					]
					
					\node[Wnode] (w1) at (0,2) {$w_1$};
					\node[Wnode] (w2) at (0,0) {$w_2$};
					\node[Wnode] (w3) at (0,-2) {$w_3$};
					
					\node[Gnode] (a1) at (3,2) {$a_1$};
					\node[Unode] (a2) at (3,0) {$a_2$};
					\node[Gnode] (a3) at (3,-2) {$a_3$};
					\node[Unode] (b3) at (3,-3.3) {$b_3$};
					
					\node[Gnode] (r1) at (5.8,0.35) {$c_1$};
					\node[Gnode] (r2) at (5.8,-3.0) {$c_2$};
					
					\begin{scope}[on background layer]
						\node[Wbox,fit=(w1)(w2)(w3)] {};
						\node[Ubox,fit=(a1)(a2)(a3)(b3)] {};
						\node[Gbox,fit=(r1)(r2)] {};
					\end{scope}
					
					\node at (0,2.95) {$W^*_{i-1}$};
					\node at (3,2.95) {$V(M_i)\setminus W^*_{i-1}$};
					\node at (5.8,1.3) {$V(G'_0)$};
					
					\draw[Mone] (w1) -- node[above,sloped] {$M^1_{i-1}$} (a1);
					\draw[Mone] (w2) -- (a2);
					\draw[Mone] (w3) -- (a3);
					
					\draw[Mtwo] (w3) -- node[below,sloped] {$M^2_{i-1}$} (b3);
					
					\draw[Nedge] (a2) -- node[above,sloped] {$N_i$} (r1);
					\draw[Nedge] (b3) -- node[below,sloped] {$N_i$} (r2);
					
					\node[anchor=west] at (1.75,-4.35)
					{$\{x^*_{i,1},y^*_{i,1},x^*_{i,2},y^*_{i,2}\}=\{a_1,a_2,a_3,b_3\}$};
					
				\end{tikzpicture}
				
				\caption{An illustration of Case~\eqref{case:S2.1} when $|W^*_{i-1}|$ is odd. The blue solid edges are the edges of $M^1_{i-1}$ used in $M_i$, the red solid edge is the additional edge from $M^2_{i-1}$, and the orange dashed edges form the matching $N_i$. Here $a_1,a_3\in V(G'_0)$, while $a_2,b_3\in U_0$. Thus $N_i$ projects $a_2,b_3$ to $c_1,c_2\in V(G'_0)$, and $\{x_{i,1},y_{i,1},x_{i,2},y_{i,2}\}=\{a_1,c_1,a_3,c_2\}$.}
				\label{fig:S21-Mi-Ni}
			\end{figure}
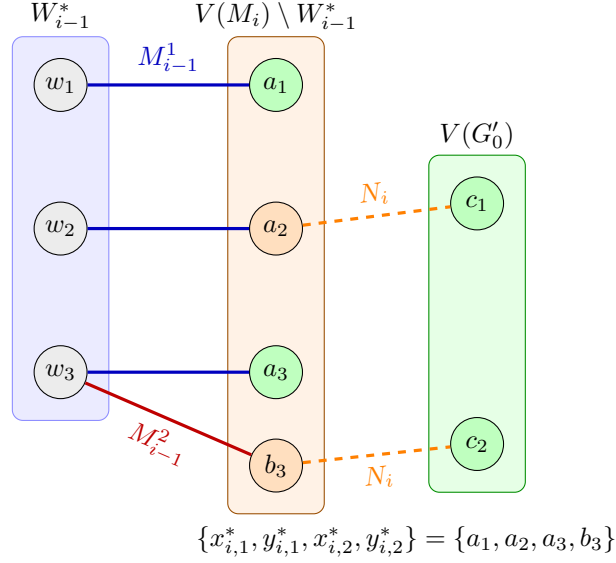
			
			\CaseLabel{S2.2}
			Suppose that $|W^*_{i-1}|=1$ and
			$\big|V(G'_{i-1})\cup V(M^1_{i-1}\cup M^2_{i-1})\big|$ is odd.
			
			\CaseLabel{S2.2.1}
			If there exists
			$x^*\in V(G'_{i-1})\cup V(M^1_{i-1}\cup M^2_{i-1})$ such that
			$d_{i-1}(x^*)\le \Delta-2(i-1)-1$, then choose
			$M_i\in\{M^1_{i-1},M^2_{i-1}\}$ so that $M_i$ does not cover $x^*$. Let
			$x^*_{i,1}$ be the unique vertex in $V(M_i)\setminus W^*_{i-1}$, set
			$y^*_{i,1}=x^*$, $h_i=1$, and $L_i^*=\{y^*_{i,1}\}$.
			
			Set $A_i^*=\{x^*_{i,1},y^*_{i,1}\}$. Let $N_i$ be a matching of
			$G^*_{i-1}[A_i^*\cap U_0,V(G'_0)\setminus (A_i^*\cap V(G'_0))]$
			that saturates $A_i^*\cap U_0$.  For each $a\in A_i^*$, define its representative
			$\rho_i(a)\in V(G'_0)$ as follows. If $a\in V(G'_0)$, set
			$\rho_i(a)=a$. If $a\in U_0$, let $\rho_i(a)$ be the vertex matched to $a$
			by $N_i$. Then set $x_{i,1}=\rho_i(x^*_{i,1})$ and
			$y_{i,1}=\rho_i(y^*_{i,1})$. The existence of the required projection edges
			is guaranteed by Claim~\ref{claim:Gi-property}\eqref{claim-Gi-property-5}. See
			Figure~\ref{fig:S221-Mi-Ni} for an illustration of $M_i$ and $N_i$.
			
			\begin{figure}[ht]
				\centering
				\begin{tikzpicture}[
					scale=0.95,
					every node/.style={font=\small},
					Wnode/.style={circle,draw,fill=gray!15,minimum size=7mm,inner sep=0pt},
					Unode/.style={circle,draw,fill=orange!25,minimum size=7mm,inner sep=0pt},
					Gnode/.style={circle,draw,fill=green!25,minimum size=7mm,inner sep=0pt},
					Mchosen/.style={blue!75!black,very thick},
					Mnotchosen/.style={red!75!black,very thick,dashed},
					Nedge/.style={orange!85!black,very thick,dashed},
					Wbox/.style={draw=blue!45,fill=blue!8,rounded corners,inner sep=8pt},
					Ubox/.style={draw=orange!60!black,fill=orange!10,rounded corners,inner sep=8pt},
					Gbox/.style={draw=green!55!black,fill=green!10,rounded corners,inner sep=8pt}
					]
					
					\node[Wnode] (w) at (0,0) {$w$};
					\node[Unode] (a) at (4,0) {$a$};
					\node[Gnode] (r) at (8,0) {$b$};
					\node[Gnode] (xstar) at (8,-2) {$x^*$};
					
					\begin{scope}[on background layer]
						\node[Wbox,fit=(w)] {};
						\node[Ubox,fit=(a)] {};
						\node[Gbox,fit=(r)(xstar)] {};
					\end{scope}
					
					\node at (0,0.95) {$W^*_{i-1}$};
					\node at (4,0.95) {$V(M_i)\setminus W^*_{i-1}$};
					\node at (8,0.95) {$V(G'_0)$};
					
					\draw[Mchosen] (w) -- node[above,sloped] {$M_i=M^1_{i-1}$} (a);
					\draw[Mnotchosen] (w) -- node[below,sloped] {$M^2_{i-1}$ not chosen} (xstar);
					\draw[Nedge] (a) -- node[above,sloped] {$N_i$} (r);
					
					\node[anchor=west] at (1.4,-3.15)
					{$x^*_{i,1}=a,\qquad y^*_{i,1}=x^*,\qquad h_i=1$};
					
					\node[anchor=west] at (1.4,-3.7)
					{$x_{i,1}=\rho_i(a)=b,\qquad y_{i,1}=\rho_i(x^*)=x^*$};
					
				\end{tikzpicture}
				
				\caption{An illustration of Case~\eqref{case:S2.2.1}. Here $W^*_{i-1}=\{w\}$. We choose $M_i=M^1_{i-1}$ so that $M_i$ does not cover the vertex $x^*$; the dashed red edge indicates the unused edge of $M^2_{i-1}$ covering $x^*$. The unique vertex in $V(M_i)\setminus W^*_{i-1}$ is $x^*_{i,1}=a$, and we set $y^*_{i,1}=x^*$. In this example, $a\in U_0$, so the matching $N_i$ projects $a$ to $b=\rho_i(a)\in V(G'_0)$, while $x^*\in V(G'_0)$ already represents itself. Thus $x_{i,1}=b$ and $y_{i,1}=x^*$.}
				\label{fig:S221-Mi-Ni}
			\end{figure}
			
			\CaseLabel{S2.2.2}
			Suppose that every vertex
			$u\in \big(V(G'_{i-1})\cup V(M^1_{i-1}\cup M^2_{i-1})\big)\setminus
			W^*_{i-1}$ satisfies $d_{i-1}(u)=\Delta-2(i-1)$. Then
			$V(G'_{i-1})\cup (V(M^1_{i-1}\cup M^2_{i-1})\setminus W^*_{i-1}) =V(G'_{i-1})$. In this case,
			keep $M_i=N_i=\emptyset$ and $h_i=0$, and select no edge related to
			$W^*_{i-1}$.
			
			\CaseLabel{S2.3}
			If $|W^*_{i-1}|=0$ and $|V(G'_{i-1})|$ is even, then keep
			$M_i=N_i=\emptyset$ and $h_i=0$, and select no edge related to
			$W^*_{i-1}$.
			
			\CaseLabel{S2.4}
			Suppose that $|W^*_{i-1}|=0$ and $|V(G'_{i-1})|$ is odd.
			
			\CaseLabel{S2.4.1}
			If $G'_{i-1}$ has two vertices $u,v$ with
			$d_{i-1}(u),d_{i-1}(v)\le \Delta-2(i-1)-1$, then set
			$L_i^*=\{u,v\}$. If $u,v\in V(G'_0)$, set $x_{i,1}=u$, $y_{i,1}=v$, and
			$h_i=1$. Otherwise, choose a matching $N_i$ of
			$G^*_{i-1}[\{u,v\}\cap U_0,V(G'_0)\setminus (\{u,v\}\cap V(G'_0))]$
			that saturates $\{u,v\}\cap U_0$.  For each $a\in\{u,v\}$, define $\rho_i(a)\in V(G'_0)$ by
			setting $\rho_i(a)=a$ if $a\in V(G'_0)$, and otherwise letting
			$\rho_i(a)$ be the vertex matched to $a$ by $N_i$. Then set
			$x_{i,1}=\rho_i(u)$, $y_{i,1}=\rho_i(v)$, and $h_i=1$. The existence of the
			required projection edges is guaranteed by Claim~\ref{claim:Gi-property}\eqref{claim-Gi-property-5}.
			
			\CaseLabel{S2.4.2}
			Suppose Case~\eqref{case:S2.4.1} does not occur. If $G'_{i-1}$ has a
			vertex $s_i^*$ with $d_{i-1}(s_i^*)\le \Delta-2(i-1)-2$, then this vertex
			is unique. If $s_i^*\in V(G'_0)$, set $I_i=\{s_i^*\}$. If
			$s_i^*\notin V(G'_0)$, set $I_i=\emptyset$ and $J_i^*=\{s_i^*\}$.
			
			\CaseLabel{S2.4.3}
			Suppose Cases~\eqref{case:S2.4.1} and~\eqref{case:S2.4.2} do not occur.
			Then $G'_{i-1}$ has at most one vertex $u$ with
			$d_{i-1}(u)<\Delta-2(i-1)$, and, if such a vertex exists, then
			$d_{i-1}(u)=\Delta-2(i-1)-1$.
			
			\CaseLabel{S2.4.3.1}
			If there is a vertex $z\in U_{i-1}$ such that
			$d_{i-1}(z)\ge \Delta/3-\eta n+2$, then choose two distinct vertices
			$x_{i,1},y_{i,1}\in N_{G^*_{i-1}}(z,V(G'_0))$.  Set
			$N_i=\{zx_{i,1},zy_{i,1}\}$ and $h_i=1$.
			
			\CaseLabel{S2.4.3.2}
			Suppose instead that every vertex $z\in U_{i-1}$ satisfies
			$d_{i-1}(z)<\Delta/3-\eta n+2$. Then $G_{i-1}$ is in the critical case.
			Hence $|W^*_{i-1}|=|W'_{i-1}|\in\{1,2\}$, contradicting the assumption that
			$|W^*_{i-1}|=0$.
			
			\medskip
			\noindent\textbf{Edges incident with vertices of
				$U_0\setminus (U_{i-1}\cup V(M_i)\cup V(N_i)\cup J_i^*)$.}
			
			Set
			\[
			Z_i=\{x_{i,1},y_{i,1},\ldots,x_{i,h_i},y_{i,h_i}\}\cup I_i
			\subseteq V(G'_0)
			\]
			and
			\[
			B_i=U_0\setminus (U_{i-1}\cup V(M_i)\cup V(N_i)\cup J_i^*).
			\]
			
			\CaseLabel{S2.5}
			If $B_i\ne \emptyset$, then for each vertex $z\in B_i$, choose two
			distinct vertices
			\[
			z_{i,1},z_{i,2}\in V(G'_0)\setminus Z_i
			\]
			such that $z$ is adjacent in $G^*_{i-1}$ to both $z_{i,1}$ and $z_{i,2}$,
			and such that $z_{i,1}$ and $z_{i,2}$ have not been assigned to any other
			such vertex in the current round. Let $N_i^*$ be the set of all edges
			selected in this way. For each $z\in B_i$, let
			\[
			e_z=z_{i,1}z_{i,2}
			\]
			be the auxiliary edge on $V(G'_0)$ corresponding to the projected two-edge
			path $z_{i,1}zz_{i,2}$. The existence of $N_i^*$ is guaranteed by
			Claim~\ref{claim:Gi-property}\eqref{claim-Gi-property-5}. See
			Figure~\ref{fig:S25-Nistar} for an illustration of $N_i^*$ and the
			auxiliary edges $e_z$.
			
			\begin{figure}[ht]
				\centering
				\begin{tikzpicture}[
					scale=0.95,
					every node/.style={font=\small},
					Unode/.style={circle,draw,fill=orange!25,minimum size=7mm,inner sep=0pt},
					Gnode/.style={circle,draw,fill=green!25,minimum size=7mm,inner sep=0pt},
					Resnode/.style={circle,draw,fill=gray!20,minimum size=7mm,inner sep=0pt},
					Nstar/.style={orange!85!black,very thick,dashed},
					Pedge/.style={green!45!black,very thick,densely dotted},
					Forbidden/.style={gray!70,thick,densely dotted},
					Ubox/.style={draw=orange!60!black,fill=orange!10,rounded corners,inner sep=8pt},
					Gbox/.style={draw=green!55!black,fill=green!10,rounded corners,inner sep=8pt},
					Rbox/.style={draw=gray!60!black,fill=gray!10,rounded corners,inner sep=8pt}
					]
					
					\node[Unode] (z1) at (0,1) {$z$};
					\node[Unode] (z2) at (0,-1) {$z'$};
					
					\node[Resnode] (x) at (3.9,2.85) {$x_{i,1}$};
					\node[Resnode] (y) at (5.25,2.85) {$y_{i,1}$};
					\node[Resnode] (I) at (6.6,2.85) {$I_i$};

					\node[Gnode] (a1) at (3.9,1.25) {$z_{i,1}$};
					\node[Gnode] (a2) at (3.9,0.25) {$z_{i,2}$};
					\node[Gnode] (b1) at (3.9,-0.55) {$z'_{i,1}$};
					\node[Gnode] (b2) at (3.9,-1.55) {$z'_{i,2}$};
					
					\begin{scope}[on background layer]
						\node[Ubox,fit=(z1)(z2)] {};
						\node[Rbox,fit=(x)(y)(I)] {};
						\node[Gbox,fit=(a1)(a2)(b1)(b2)] {};
					\end{scope}
					
					\node[align=center] at (0,2.05)
					{$B_i=U_0\setminus (U_{i-1}\cup V(M_i)\cup V(N_i)\cup J_i^*)$};
					\node[align=center] at (5.0,3.95)
					{Reserved vertices};
					\node[align=center] at (4.75,-2.55)
					{Available vertices in $V(G'_0)$};
					
					\foreach \v in {x,y,I}{
						\draw[Forbidden] ($(\v)+(-0.22,-0.22)$) -- ($(\v)+(0.22,0.22)$);
						\draw[Forbidden] ($(\v)+(-0.22,0.22)$) -- ($(\v)+(0.22,-0.22)$);
					}
					
					\draw[Nstar] (z1) -- node[above,sloped,pos=0.42] {$N_i^*$} (a1);
					\draw[Nstar] (z1) -- (a2);
					
					\draw[Nstar] (z2) -- node[below,sloped,pos=0.42] {$N_i^*$} (b2);
					\draw[Nstar] (z2) -- (b1);
					
					\draw[Pedge] (a1) -- node[right,pos=0.48] {$e_z$} (a2);
					\draw[Pedge] (b1) -- node[right,pos=0.48] {$e_{z'}$} (b2);
					
					\node[anchor=west] at (-0.35,-3.15)
					{$N_i^*=\{zz_{i,1},zz_{i,2},z'z'_{i,1},z'z'_{i,2}\}$};
					
					\node[anchor=west] at (-0.35,-3.7)
					{$e_z=z_{i,1}z_{i,2},\qquad e_{z'}=z'_{i,1}z'_{i,2}$};
					
				\end{tikzpicture}
				
				\caption{An illustration of Case~\eqref{case:S2.5}. For each vertex $z\in B_i$, we choose two distinct neighbors $z_{i,1},z_{i,2}$ in $V(G'_0)\setminus (Z_i\cup V(N_i))$. The orange dashed edges are added to $N_i^*$, while the green dotted edge $e_z=z_{i,1}z_{i,2}$ is the auxiliary edge used in the projected cycle on $V(G'_0)$. When the cycle is projected back, $e_z$ is replaced by the two-edge path $z_{i,1}zz_{i,2}$. The figure shows two such vertices, $z$ and $z'$, and the four chosen representatives $z_{i,1},z_{i,2},z'_{i,1},z'_{i,2}$ are distinct in the current round.}
				\label{fig:S25-Nistar}
			\end{figure}
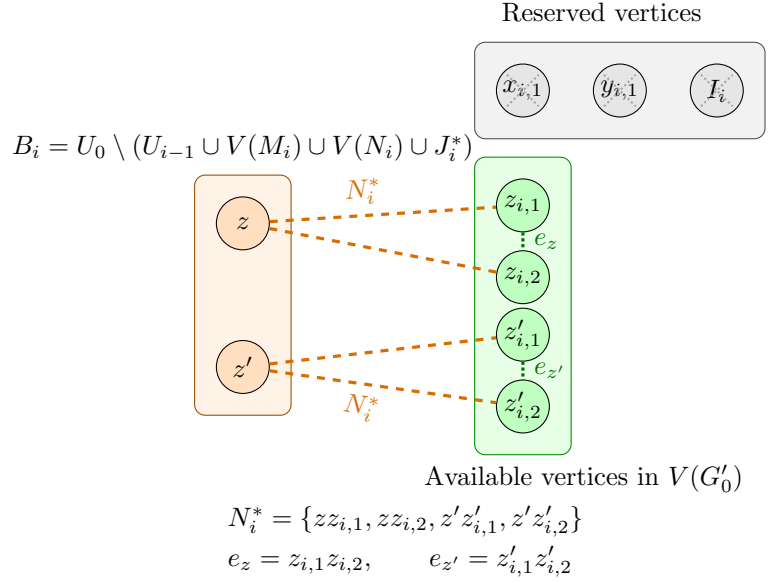
			
			After Case~\eqref{case:S2.5}, define
			\[
			F_i=\{x_{i,1}y_{i,1},\ldots,x_{i,h_i}y_{i,h_i}\}\cup \{e_z:z\in B_i\}.
			\]
			Thus the edges in $F_i$ are auxiliary edges on $V(G'_0)$. When the future
			cycle is projected back to the original vertex set, each auxiliary edge
			$e_z=z_{i,1}z_{i,2}$ with $z\in B_i$ is replaced by the two-edge path
			$z_{i,1}zz_{i,2}$.
			
			An illustration of the edges $M_i$, $N_i$, and $N_i^*$ all together is
			provided in Figure~\ref{fig:projection-mechanism-vertical}.
			
			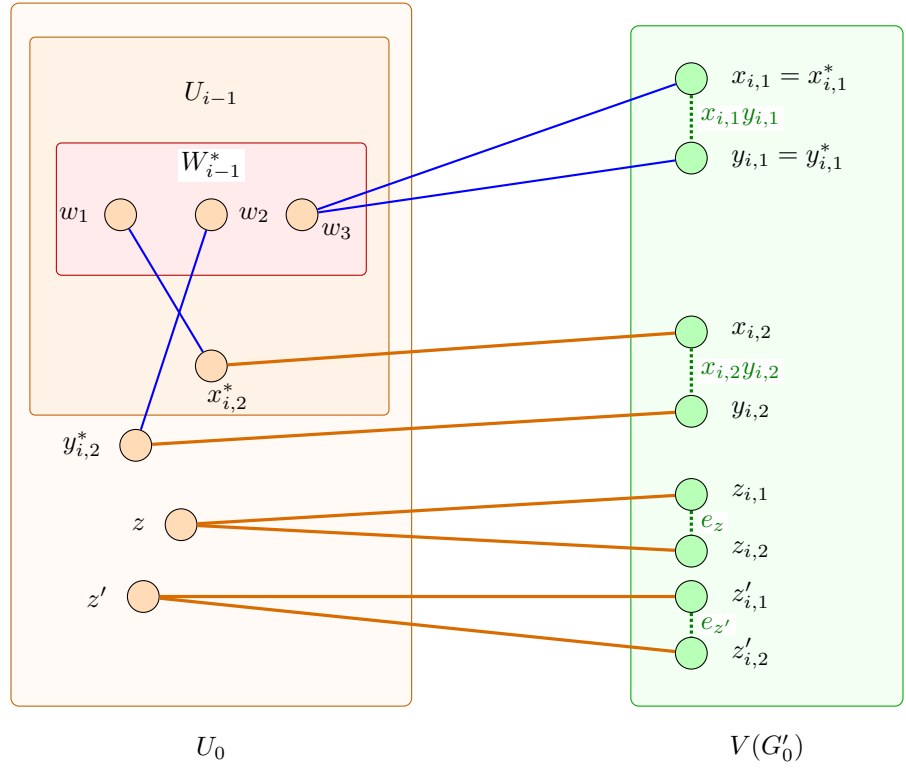
\begin{figure}[ht]
				\centering
				\begin{tikzpicture}[
					x=1cm,y=1cm,
					every node/.style={font=\small},
					leftv/.style={circle,draw,fill=orange!28,minimum size=4.2mm,inner sep=0pt},
					rightv/.style={circle,draw,fill=green!28,minimum size=4.2mm,inner sep=0pt},
					exc/.style={circle,draw,fill=gray!20,minimum size=4.2mm,inner sep=0pt},
					proj/.style={orange!85!black,very thick},
					auxedge/.style={green!50!black,very thick,densely dotted},
					pathhint/.style={blue!75!black,thick,dashed},
					mainboxL/.style={draw=orange!75!black,rounded corners=3pt,fill=orange!4},
					mainboxR/.style={draw=green!65!black,rounded corners=3pt,fill=green!5},
					uibox/.style={draw=orange!75!black,rounded corners=2pt,fill=orange!10},
					wbox/.style={draw=red!70!black,rounded corners=2pt,fill=red!8}
					]
					
					\draw[mainboxL] (0,-0.5) rectangle (5.3,8.8);
					\draw[mainboxR] (8.2,-0.5) rectangle (11.8,8.5);
					
					\draw[uibox] (0.25,3.35) rectangle (5.0,8.35);
					\draw[wbox] (0.60,5.20) rectangle (4.70,6.95);
					
					\node at (2.65,-1.05) {$U_0$};
					\node at (10.0,-1.05) {$V(G'_0)$};
					
					\node at (2.65,7.60) {$U_{i-1}$};
					\node[fill=white,inner sep=1pt] at (2.65,6.65) {$W^*_{i-1}$};
					
					\node[leftv] (w1) at (1.45,6.00) {};
					\node[left=0.05cm] at (w1.west) {$w_1$};
					
					\node[leftv] (w2) at (2.65,6.0) {};
					\node[right=0.01cm] at (w2.east) {$w_2$};
					
					\node[leftv] (w3) at (3.85,6.0) {};
					\node[right=0.12cm] at (w3.south) {$w_3$};
					
					\node[leftv] (xs2) at (2.65,4.0) {};
					\node[below=0.10cm] at (xs2.east) {$x^*_{i,2}$};
					
					\node[leftv] (ys2) at (1.65,2.95) {};
					\node[left=0.10cm] at (ys2.west) {$y^*_{i,2}$};
					
					\node[leftv] (z) at (2.25,1.90) {};
					\node[left=0.12cm] at (z.west) {$z$};
					
					\node[leftv] (zp) at (1.75,0.95) {};
					\node[left=0.12cm] at (zp.west) {$z'$};
					
					\draw[blue,thick] (w1) -- (xs2);
					\draw[blue,thick] (w2) -- (ys2);
					
					\node[rightv] (x1) at (9.00,7.80) {};
					\node[right=0.18cm] at (x1.east) {$x_{i,1}=x^*_{i,1}$};
					
					\node[rightv] (y1) at (9.00,6.75) {};
					\node[right=0.18cm] at (y1.east) {$y_{i,1}=y^*_{i,1}$};
					
					\node[rightv] (x2) at (9.00,4.45) {};
					\node[right=0.18cm] at (x2.east) {$x_{i,2}$};
					
					\node[rightv] (y2) at (9.00,3.40) {};
					\node[right=0.18cm] at (y2.east) {$y_{i,2}$};
					
					\node[rightv] (z11) at (9.00,2.30) {};
					\node[right=0.18cm] at (z11.east) {$z_{i,1}$};
					
					\node[rightv] (z12) at (9.00,1.55) {};
					\node[right=0.18cm] at (z12.east) {$z_{i,2}$};
					
					\node[rightv] (zp11) at (9.00,0.95) {};
					\node[right=0.18cm] at (zp11.east) {$z'_{i,1}$};
					
					\node[rightv] (zp12) at (9.00,0.2) {};
					\node[right=0.18cm] at (zp12.east) {$z'_{i,2}$};
					
					\draw[blue,thick] (w3) -- (x1);
					\draw[blue,thick] (w3) -- (y1);
					
					\draw[proj] (xs2) -- (x2);
					\draw[proj] (ys2) -- (y2);
					
					\draw[proj] (z) -- (z11);
					\draw[proj] (z) -- (z12);
					
					\draw[proj] (zp) -- (zp11);
					\draw[proj] (zp) -- (zp12);
					
					\draw[auxedge] (x1) -- node[right=2pt,pos=0.5,fill=white,inner sep=1pt] {$x_{i,1}y_{i,1}$} (y1);
					\draw[auxedge] (x2) -- node[right=2pt,pos=0.5,fill=white,inner sep=1pt] {$x_{i,2}y_{i,2}$} (y2);
					\draw[auxedge] (z11) -- node[right=2pt,pos=0.5,fill=white,inner sep=1pt] {$e_z$} (z12);
					\draw[auxedge] (zp11) -- node[right=2pt,pos=0.5,fill=white,inner sep=1pt] {$e_{z'}$} (zp12);
					
				\end{tikzpicture}
				
				\caption{An illustration of the projection structure. The left panel represents $U_0$, with the shaded inner region indicating $U_{i-1}$ and the darker shaded subregion indicating $W^*_{i-1}\subseteq U_{i-1}$. The vertices $w_1,w_2,w_3$ illustrate vertices of $W^*_{i-1}$. The orange straight edges indicate projections to representatives in $V(G'_0)$. The blue edges represent edges of $M_i$, and the green dotted edges are auxiliary edges in $F_i$. When the future cycle on $G'_0-I_i+F_i$ is projected back, each edge $e_z$ is replaced by a two-edge path through the corresponding vertex $z$, while each edge $x_{i,k}y_{i,k}$ is replaced by the corresponding prescribed path on the left.}
				\label{fig:projection-mechanism-vertical}
			\end{figure}
			
			\item \textbf{Updating the degree records.}
			
			We update the auxiliary degree records by assuming the removal of all edges
			in $E_i:=R_i\cup M_i\cup N_i\cup N_i^*$, and by also accounting for a
			future spanning cycle of
			\[
			G'_0-I_i+F_i
			\]
			that contains every edge of $F_i$. When this cycle is projected back, each
			auxiliary edge $e_z=z_{i,1}z_{i,2}$ with $z\in B_i$ is replaced by the
			two-edge path $z_{i,1}zz_{i,2}$. If $B_i=\emptyset$, then no such auxiliary
			edge $e_z$ is present. Let $Q_i=G_{i-1}[E_i]$.
			
			The values $d_i(v)$ are defined as follows.
			\begin{enumerate}[]
				\item \CaseLabel{S3.1} For every $w\in W^*_{i-1}$, set
				$d_i(w)=d_{i-1}(w)-d_{Q_i}(w)$.
				
				\item \CaseLabel{S3.2} For every $u\in U_{i-1}\setminus W^*_{i-1}$,
				set $d_i(u)=d_{i-1}(u)-d_{Q_i}(u)$.
				
				\item \CaseLabel{S3.3} For every $u\in U_0\setminus U_{i-1}$, set
				$d_i(u)=d_{i-1}(u)-d_{Q_i}(u)$. Thus a vertex handled by $N_i^*$ is
				charged two, a vertex in $L_i^*\setminus V(G'_0)$ is charged one, and
				a vertex in $J_i^*$ is not charged in this round.
				
				\item \CaseLabel{S3.4} For every $v\in V(G'_0)$, set
				\[
				d_i(v)=
				\begin{cases}
					d_{i-1}(v), & \text{if } v\in I_i,\\
					d_{i-1}(v)-1, & \text{if } v\in L_i^*\cap V(G'_0),\\
					d_{i-1}(v)-2, & \text{otherwise.}
				\end{cases}
				\]
			\end{enumerate}
			
			\item \textbf{Updating the graph and the exceptional sets.}
			
			Let $G_i$ be obtained from $G^*_{i-1}-(E_i\setminus R_i)=G_{i-1}-E_i$ by
			discarding any isolated vertices, and define
			\[
			U_i=\{v\in U_{i-1}\cap V(G_i):\Delta-2i-d_i(v)\ge \eta n\}.
			\]
			Then let
			\[
			W_i=U_i\cap W_{i-1} \quad \text{and} \quad G'_i=G_i-U_i.
			\]
			Increase $i$ by one and repeat the process.
		\end{enumerate}
	\end{Ope}
	
	In the claim below, we verify the auxiliary properties needed in Step~2 of
	Operation~\ref{ope:edge-removing}. In particular, we prove the existence of the
	matchings $M^1_{i-1}$ and $M^2_{i-1}$ when $W^*_{i-1}$ is not an independent set in the current graph, the
	existence of the matching $N_i$, and the availability of the vertices used in
	the construction of $N_i^*$. Throughout the claim, the statements are
	understood at the beginning of Step~2 of round $i$, after Step~1 has been
	performed. Thus $G^*_{i-1}=G_{i-1}-R_i$ after discarding isolated vertices.
	
	\begin{CLA}\label{claim:Gi-property}
		For each round $i\ge 1$ for which Step~2 of
		Operation~\ref{ope:edge-removing} is reached, the following statements hold.
		\begin{enumerate}[(1)]
			\item $U_{i-1}\subseteq \ldots \subseteq U_0$.
			\label{claim-Gi-property-1}
			
			\item For any vertex $u\in U_0\cap V(G_{i-1})$, we have
			$d_{i-1}(u)=d_{G_{i-1}}(u)$.
			\label{claim-Gi-property-1.3}
			
			\item If it exists, let $j\ge 0$ be the smallest index for which
			$W_j^*=W'_j$. Then $W^*_{i-1}=W'_j$ for every $i\ge j+1$ for which
			Operation~\ref{ope:edge-removing} reaches Step~2, and
			\[
			d_{G^*_{i-1}}(u)\ge
			\Delta/3-\eta n-2\eta^2n-(i-1-j)
			\]
			for any $u\in W^*_{i-1}$.
			\label{claim-Gi-property-1.5}

	\item \label{claim-Gi-property-3}
		For every $u\in V(G^*_{i-1})\setminus W^*_{i-1}$, we have
		\[
		d_{G^*_{i-1}}(u)\ge
		\begin{cases}
			\Delta-\eta n-2(i-1)-\eta^2n
			\ge \Delta/3-\eta n-2\eta^2n-4,
			& \text{if $u\in V(G^*_{i-1})\setminus U_{i-1}$},\\
			d_0(u)-2\eta^2n-4-2\alpha_u
			\ge \Delta/3-\eta n-2\eta^2n-4,
			& \text{if $u\in U_{i-1}\setminus W^*_{i-1}$},
		\end{cases}
		\]
		where $\alpha_u$ denotes the number of rounds before round $i$ in which
		$u$ is chosen as the special vertex $z$ in
		Case~\eqref{case:S2.4.3.1}.

			\item The graph
			$G^*_{i-1}[W^*_{i-1},V(G^*_{i-1})\setminus W^*_{i-1}]$ has two
			edge-disjoint matchings $M^1_{i-1}$ and $M^2_{i-1}$, each of which
			saturates $W^*_{i-1}$.
			\label{claim-Gi-property-4}
			
			\item Let $U^*_{i-1}\subseteq U_{i-1}\setminus W^*_{i-1}$ with
			$|U^*_{i-1}|<2\eta^2 n$, and let $Z^*_{i-1}\subseteq V(G'_0)$ with
			$|Z^*_{i-1}|<2\eta^2 n$. Then
			$G^*_{i-1}[U^*_{i-1},V(G'_0)\setminus Z^*_{i-1}]$ has a matching
			saturating $U^*_{i-1}$.
			\label{claim-Gi-property-5}
			
			\item In Operation~\ref{ope:edge-removing}, for each $u\in V(G'_0)$, we
			can ensure that
			\[
			u\in \{x_{i,1},y_{i,1},\ldots,x_{i,h_i},y_{i,h_i}\}\cup I_i
			\cup (V(N_i^*)\cap V(G'_0))
			\]
			for at most $3\eta n$ distinct indices $i$.
			\label{claim-Gi-property-6}
		\end{enumerate}
	\end{CLA}
	
	\begin{proof}
		For Claim~\ref{claim:Gi-property}\eqref{claim-Gi-property-1}, the assertion
		follows immediately from the definition of $U_i$ in Step~4 of
		Operation~\ref{ope:edge-removing}.
		
		For Claim~\ref{claim:Gi-property}\eqref{claim-Gi-property-1.3}, for $i=1$ the
		assertion follows since $d_0(u)=d_{G_0}(u)$. Now suppose that $i\ge 2$ and the
		assertion holds before round $i$. Let $u\in U_0\cap V(G_{i-1})$. Since
		$G_{i-1}=G_{i-2}-E_{i-1}$ after discarding isolated vertices, we have
		$d_{G_{i-1}}(u)=d_{G_{i-2}}(u)-d_{Q_{i-1}}(u)$, 
		where $Q_{i-1}=G_{i-2}[E_{i-1}]$. On the other hand, Step~3 of
		Operation~\ref{ope:edge-removing} gives
		$d_{i-1}(u)=d_{i-2}(u)-d_{Q_{i-1}}(u)$. Since
		$d_{i-2}(u)=d_{G_{i-2}}(u)$ by the induction hypothesis, we get
		$d_{i-1}(u)=d_{G_{i-1}}(u)$.
		
		For Claim~\ref{claim:Gi-property}\eqref{claim-Gi-property-1.5}, let $j\ge 0$
		be the first index for which $W_j^*=W'_j$. Then $W_j=\emptyset$, and
		$W'_j\subseteq U_j$ with $|W'_j|\le 2$ and $|W'_j\cup V(G'_j)|$ even. It
		suffices to show that no vertex of $W'_j$ becomes isolated in any later round
		for which Step~2 of Operation~\ref{ope:edge-removing} is reached.

	If $j>\Delta/3-2\eta n$, then the process stops immediately, so it remains
	only to prove the asserted degree bound. This follows by the same argument as
	in the case $j\le\Delta/3-2\eta n$. We therefore assume below that
	$j\le\Delta/3-2\eta n$.

		Let $u\in W'_j$. Since $u\in U_j\subseteq U_0$, we have
		$d_0(u)=d_{G_0}(u)\ge \Delta/3-\eta n$. By the minimality of $j$, we have
		$W^*_{s-1}=W_{s-1}\subseteq W_0$ for every $s\in[j]$. Hence, before the
		critical case is entered, if the degree record of $u$ decreases in round $s$,
		then $u$ is incident in $G_{s-1}$ with an edge of
		$R_s\cup M_s\cup N_s$ arising from an edge between $u$ and
		$W^*_{s-1}$.
		Moreover, $R_s\cup M_s\subseteq
		E_{G_0}\bigl(W_0,V(G_0)\setminus W_0\bigr)$, 
		and each edge of $N_s$ either corresponds to the projection of an edge of
		$M_s$, or arises from a vertex $z\in U_{s-1}$ satisfying
		$d_{s-1}(z)\ge \Delta/3-\eta n+2$, as in
		Case~\eqref{case:S2.4.3.1}. In the latter case, after the two edges of $N_s$
		incident with $z$ are deleted, its degree record remains at least
		$\Delta/3-\eta n$. Thus such a use does not reduce the degree record below
		the lower bound used here. Consequently, when estimating the remaining degree
		loss of $u$, we only need to count its uses through $R_s$, $M_s$, and the
		projection edges of $N_s$ corresponding to edges of $M_s$.
		
		Each edge from $u$ to $W_0$ can be charged at most twice: once
		through an edge of $R_s\cup M_s$, and once through the corresponding
		projection edge in $N_s$. Therefore
		\[
		d_{G_j^*}(u)\ge d_{G_0}(u)-2e_{G_0}(u,W_0)
		\ge \Delta/3-\eta n-2\eta^2n.
		\]

		We claim that no later edge from an $R_s$ is deleted with respect to $W'_j$.
		Suppose not, and let $s\ge j+1$ be the first round after the critical case in
		which Algorithm~\ref{alg:edge-deletion-wrt-W} deletes an edge with respect to
		$W^*_{s-1}=W'_j$. Let $uv$ be the first such edge deleted in that round, with
		$u\in W'_j$. At the moment just before $uv$ is deleted, we have
		$e(v,W'_j)\ge d(u)$. Since $|W'_j|\le 2$, this gives $d(u)\le 2$. On the other
		hand, by the choice of $s$, no preprocessing edge has been deleted with
		respect to $W'_j$ in the rounds $j+1,\ldots,s-1$, and in each such round only
		one prescribed edge incident with $u$ is deleted. Hence
		\[
		d(u)\ge \Delta/3-\eta n-2\eta^2n-(s-1-j).
		\]
		Thus $s-1-j\ge \Delta/3-\eta n-2\eta^2n-2$, contradicting
		$s-1\le \Delta/3-2\eta n$ from the stopping condition of
		Operation~\ref{ope:edge-removing}. Therefore no
		such $R_s$-edge is deleted.
		
		Thus, for any $i\ge j+1$ for which Step~2 is reached,
		\[
		d_{G^*_{i-1}}(u)\ge d_{G^*_j}(u)-(i-1-j)
		\ge \Delta/3-\eta n-2\eta^2n-(i-1-j)>0.
		\]
		It follows that $W^*_{i-1}=W'_j$ for every such $i$.

		For Claim~\ref{claim:Gi-property}\eqref{claim-Gi-property-3}, 
		we first estimate the degree loss of a vertex $u\in V(G^*_{i-1})\setminus W^*_{i-1}$ 
		with respect to the edges deleted in the preprocessing step of Operation~\ref{ope:edge-removing}. 
		Every edge in
		$R_s$ is deleted in the preprocessing step with respect to $W^*_{s-1}$. Before
		the critical case, $W^*_{s-1}\subseteq W_0$ for every $s$, so the charging
		property of Algorithm~\ref{alg:edge-deletion-wrt-W} gives
		$d_{G_0[\bigcup_{s=1}^{i}R_s]}(u)\le e_{G_0}(u,W_0)$ for any fixed vertex
		$u$. If the critical case occurs, then after the critical case,
		Claim~\ref{claim:Gi-property}\eqref{claim-Gi-property-1.5} shows that no
		$R_s$-edge is deleted with respect to the auxiliary active set, so there is no
		further contribution. Hence
		$d_{G_0[\bigcup_{s=1}^{i}R_s]}(u)\le e_{G_0}(u,W_0)\le \eta^2n$. 
		
		For
		$u\in V(G^*_{i-1})\setminus U_{i-1}$, the definition of $U_{i-1}$ and the
		bound above give
		$d_{G^*_{i-1}}(u)\ge \Delta-\eta n-2(i-1)-\eta^2n> \Delta/3-\eta n-2\eta^2n-4$, 
		since the operation is performed only while
		$i\le \Delta/3-\eta n$ or $i\le \Delta/3-2\eta n$.

		Now let $u\in U_{i-1}\setminus W^*_{i-1}$, and let $\alpha_u$ be the
		number of rounds $s\in[i-1]$ in which $u$ is chosen as the special vertex
		$z$ in Case~\eqref{case:S2.4.3.1}. Let
	$i_1=d_{G_0[\bigcup_{s=1}^{i}R_s]}(u)$, 
		and let $i_2$ be the total number of edges incident with $u$ in
		$\bigcup_{s=1}^{i-1}(M_s\cup N_s)$, excluding, in each round in which $u$ is
		chosen as the special vertex $z$, the two prescribed edges of $N_s$ incident
		with $u$. Each time that $u$ is chosen as the special vertex $z$, the two prescribed
		edges of $N_s$ decrease its degree by two, while all other losses before
		Step~2 of round $i$ are counted by $i_1+i_2$. Hence
		\[d_{G^*_{i-1}}(u)
		\ge d_0(u)-2\alpha_u-i_1-i_2.
		\]

	It remains to bound $i_1+i_2$. Before the critical case, every active set is
	contained in $W_0\subseteq U_0$. Thus every edge of $R_s$ or $M_s$ incident
	with $u$ has its other endpoint in $W_0$. Moreover, except in
	Case~\eqref{case:S2.4.3.1}, every edge of $N_s$ incident with $u$ is the
	projection of an edge of $M_s$ incident with a vertex of $W_0$.
	Consequently, each edge from $u$ to $W_0$ is charged at most twice: once as
	an edge of $R_s\cup M_s$, and once through the corresponding projection edge
	in $N_s$.

		After the critical case, there is a unique inherited auxiliary active
		set $W'_t$. By
		Claim~\ref{claim:Gi-property}\eqref{claim-Gi-property-1.5}, no
		preprocessing edge is deleted with respect to $W'_t$. Since
		$|W'_t|\le 2$, the prescribed edges involving the vertices of $W'_t$,
		together with their projections, contribute at most four additional
		losses at $u$. Therefore, $i_1+i_2
		\le 2e_{G_0}(u,W_0)+4
		\le 2\eta^2n+4$. 
		It follows that $d_{G^*_{i-1}}(u)\ge d_0(u)-2\alpha_u-2\eta^2n-4$.

		We finally show that $d_0(u)-2\alpha_u\ge \Delta/3-\eta n$. 
		If $\alpha_u=0$, this follows from $u\in U_{i-1}\subseteq U_0$ and the
		definition of $U_0$. Suppose that $\alpha_u>0$, and let $r$ be the last
		round in which $u$ is chosen as the special vertex $z$. Before round $r$,
		the vertex $u$ has already been used as $z$ exactly $\alpha_u-1$ times.
		Since each such use deletes two edges incident with $u$, while all other
		operations can only decrease its degree record, we have $d_{r-1}(u)\le d_0(u)-2(\alpha_u-1)$. 
		On the other hand, the defining condition of
		Case~\eqref{case:S2.4.3.1} gives $d_{r-1}(u)\ge \Delta/3-\eta n+2$. 
		Thus $d_0(u)-2\alpha_u\ge \Delta/3-\eta n$.  
		Combining this with the  bound $d_{G^*_{i-1}}(u)\ge d_0(u)-2\alpha_u-i_1-i_2$ yields
		\[
		d_{G^*_{i-1}}(u)
		\ge d_0(u)-2\alpha_u-2\eta^2n-4
		\ge \Delta/3-\eta n-2\eta^2n-4,
		\]
		as required.

		For Claim~\ref{claim:Gi-property}\eqref{claim-Gi-property-4}, if
		$W^*_{i-1}$ is independent in $G_{i-1}$, then after Step~1 the existence of two
		edge-disjoint matchings saturating $W^*_{i-1}$ follows from the stopping rule
		of Algorithm~\ref{alg:edge-deletion-wrt-W} and
		Lemma~\ref{lem:matching-in-critical-graph}.
		
		Now suppose $W^*_{i-1}$ is not independent. Since $W_0$ is independent in
		$G_0$, this cannot occur if $W^*_{i-1}\subseteq W_0$, so we are after the
		unique critical case. Hence $W^*_{i-1}=W'_j$ is the auxiliary active set chosen
		in that case for some $j\le i$, and since $W^*_{i-1}$ is not independent and
		$|W^*_{i-1}|\le 2$, we get $|W^*_{i-1}|=2$. By  
		Claim~\ref{claim:Gi-property}\eqref{claim-Gi-property-1.5},
		\[
	d_{G^*_{i-1}}(u) \ge \Delta/3-\eta n-2\eta^2n-(i-1-j)>5
		\]
		for each $u\in W^*_{i-1}$, where the last inequality uses
		$i\le\Delta/3-2\eta n$. Hence each vertex of $W^*_{i-1}$ has at least four
		neighbors in $V(G^*_{i-1})\setminus W^*_{i-1}$, and so
		$G^*_{i-1}[W^*_{i-1},V(G^*_{i-1})\setminus W^*_{i-1}]$ has two edge-disjoint
		matchings saturating $W^*_{i-1}$.
		
		For Claim~\ref{claim:Gi-property}\eqref{claim-Gi-property-5}, let
		$u\in U^*_{i-1}$. By Claim~\ref{claim:Gi-property}\eqref{claim-Gi-property-3},
		we have $d_{G^*_{i-1}}(u)\ge \Delta/3-\eta n-2\eta^2n-4$. 
	Excluding neighbors in $U_{0}$ and in
		$Z^*_{i-1}$, in $G^*_{i-1}$, the number of available neighbors of $u$ in
		$V(G'_0)\setminus Z^*_{i-1}$ is at least
		\[
		\Delta/3-\eta n-2\eta^2n-4-|U_{0}|-|Z^*_{i-1}|
		\ge \Delta/3-\eta n-5\eta^2n-4
		>2\eta^2n>|U^*_{i-1}|.
		\]
	 Hence
		a matching saturating $U^*_{i-1}$ exists greedily in
		$G^*_{i-1}[U^*_{i-1},V(G'_0)\setminus Z^*_{i-1}]$.
		
		For Claim~\ref{claim:Gi-property}\eqref{claim-Gi-property-6}, fix
		$u\in V(G'_0)$ and partition its appearances into freely chosen and forced.
		
		\medskip\noindent\textit{Freely chosen appearances.}
		In each round at most $\max\{2|U_0|,4\}\le 2\eta^2 n$ vertices of $V(G'_0)$
		are needed as freely chosen projection vertices, so over fewer than
		$\Delta/3$ rounds the total number of such usages is less than
		$\frac{2}{3}\eta^2n\Delta$. Hence at most $\frac{2}{3}\eta^2n\Delta/(\eta n)\le\eta\Delta$ vertices
		are used in at least $\lfloor\eta n\rfloor$ such rounds. By
		Claim~\ref{claim:Gi-property}\eqref{claim-Gi-property-3} and
		Claim~\ref{claim:Gi-property}\eqref{claim-Gi-property-1.3}, every vertex
		requiring projection has at least $\Delta/3-2\eta n$ neighbors in $V(G'_0)$.
		After excluding those already used in at least $\lfloor\eta n\rfloor$ rounds
		and the fewer than $2\eta^2n$ vertices already fixed in the current round, at
		least $\Delta/3-2\eta n-\eta n-2\eta^2n>3\eta^2n$ 
		neighbors remain. Hence one may choose all freely selected representatives so
		that each $u\in V(G'_0)$ appears in fewer than $\eta n$ such rounds.
		
		\medskip\noindent\textit{Forced $M_i$-endpoint appearances.}
		These occur in Cases~\eqref{case:S2.1} and~\eqref{case:S2.2} when
		$a\in A_i^*$ already lies in $V(G'_0)$, forcing $\rho_i(a)=a$. Apart from the
		exceptional vertex placed in $L_i^*$ in Case~\eqref{case:S2.2.1}, each such
		appearance corresponds to a distinct edge of $M_i$ incident with $a$ and a
		vertex of $W^*_{i-1}$. Before the critical case the count is at most
		$e_{G_0}(a,W_0)$, and after it the auxiliary active set has size at most two,
		contributing at most two further appearances. Hence each $a\in V(G'_0)$
		appears in this way at most $e_{G_0}(a,W_0)+2\le\eta^2n+2$ times.
		
		\medskip\noindent\textit{Exceptional forced appearances.}
		These include $y^*_{i,1}$ in Case~\eqref{case:S2.2.1} and vertices  selected in
		Cases~\eqref{case:S2.4.1} and~\eqref{case:S2.4.2}, when they lie in
		$V(G'_0)$. For such $u\in V(G'_{i-1})$,
		Claim~\ref{claim:Gi-property}\eqref{claim-Gi-property-3} gives
		$d_{i-1}(u)\ge \Delta-\eta n-2(i-1)-\eta^2n$. 
		Since each such occurrence causes $u$ to lose one less degree than the other
		vertices of $V(G'_0)$, while a vertex in $I_i$ loses exactly two less, and the
		maximum auxiliary degree is $\Delta-2(i-1)$, it follows that $u$ can appear in
		this way in at most $\eta n+\eta^2n$ rounds.
		
		\medskip
		Summing the three bounds, each $u\in V(G'_0)$ appears in
		\[
		\{x_{i,1},y_{i,1},\ldots,x_{i,h_i},y_{i,h_i}\}\cup I_i
		\cup(V(N_i^*)\cap V(G'_0))
		\]
		for at most
		\[
		(\eta n)+(\eta^2n+2)+(\eta n+\eta^2n)<3\eta n
		\]
		indices $i$. 
	\end{proof}

	Let $\ell$ be the number of completed rounds of the edge-removal process. For
	each $i\in[\ell]$, recall that
	\[
	B_i=U_0\setminus (U_{i-1}\cup V(M_i)\cup V(N_i)\cup J_i^*)
	\]
	and
	\[
	F_i=\{x_{i,1}y_{i,1},\ldots,x_{i,h_i}y_{i,h_i}\}\cup \{e_z:z\in B_i\},
	\]
	where we let $F_i=\emptyset$ if $h_i=0$ and $B_i=\emptyset$. These auxiliary
	edges are regarded as labelled edges distinct from the edges of $G'_0$, even if
	they have the same endpoints as some edge of $G'_0$.
	
	By construction, $|F_i|\le 2\eta^2 n$ for every $i\in[\ell]$. Furthermore, by
	Claim~\ref{claim:Gi-property}\eqref{claim-Gi-property-6}, the choices can be
	made so that each vertex of $V(G'_0)$ belongs to $V(F_i)$ for at most
	$3\eta n$ distinct indices $i\in[\ell]$.
	
	For each $i\in[\ell]$, choose distinct vertices
	$v_{i,1},v_{i,2}\in V(G'_0)\setminus (V(F_i)\cup I_i)$ so that, for every
	$v\in V(G'_0)$, there are at most two pairs $(i,j)\in[\ell]\times[2]$ with
	$v=v_{i,j}$. This can be done greedily, since in each round the forbidden set
	has size $O(\eta^2n)$ and $\ell\le \Delta/3$.
	Let $P_{i,1},\ldots,P_{i,f_i}$ be the edges of $F_i$, each regarded as a
	path of length one. For each $j\in[f_i]$, let $u_{i,j}$ and $w_{i,j}$ be the
	endvertices of $P_{i,j}$.
	If $f_i\ge 1$, let $L_i$ be the multiset consisting
	of the path
	\[
	v_{i,1}u_{i,1}P_{i,1}w_{i,1}u_{i,2}P_{i,2}w_{i,2}\cdots
	u_{i,f_i}P_{i,f_i}w_{i,f_i}v_{i,2},
	\]
	the one-edge path $v_{i,2}v_{i,1}$, and all vertices of $I_i$ as isolated
	vertices. If $f_i=0$, let $L_i$ consist of two copies of the one-edge path
	$v_{i,1}v_{i,2}$, together with all vertices of $I_i$ as isolated vertices.
	
	Then $(L_i,F_i)$ is a layout for each $i\in[\ell]$: indeed,
	$F_i\subseteq E(L_i)$, while the connector edges, in particular the edge
	$v_{i,2}v_{i,1}$, ensure that $E(L_i)\setminus F_i\neq\emptyset$. Denote
	$L=\bigcup_{i\in[\ell]}L_i$. For each $i\in[\ell]$, we have
	$V(L_i)\subseteq V(G'_0)$, $|V(L_i)|\le 2e(F_i)+4<6\eta^2n$, and
	$|E(L_i)|\le 2e(F_i)+2<6\eta^2 n$. Hence, with $\varepsilon=\eta^{\frac 14}$, we
	have $|V(L_i)|\le \varepsilon^2| V(G'_0)|$ and $|E(L_i)|\le \varepsilon^4 | V(G'_0)|$ for $n$
	sufficiently large.
	
	Moreover, for each $v\in V(G'_0)$,
	Claim~\ref{claim:Gi-property}\eqref{claim-Gi-property-6} and the construction
	of the connector vertices imply that $v\in V(L_i)$ for at most
	$3\eta n+2$ distinct indices $i$. Each such occurrence contributes at most two
	to $d_L(v)$, and the isolated occurrences from $I_i$ contribute zero. Thus
	$d_L(v)\le 6\eta n+4$. Since $\eta\ll\varepsilon$, we have
	$d_L(v)\le \varepsilon^3| V(G'_0)|$ and $v\in V(L_i)$ for at most $\varepsilon^2| V(G'_0)|$
	indices $i$. Therefore the assumptions on the layouts in
	Lemma~\ref{lem:removing-linear-forests} are satisfied.
	
	Let $\alpha^*=\Delta/n$. Since $d_{G'_0}(v)\ge (\alpha^*-\eta)n-\eta n $ for every
	$v\in V(G'_0)$,   the maximum degree of $G'_0$ is at most
	$\Delta=\alpha^*n$,  and $|V(G_0')| \ge (1-\eta^2) n$, the graph $G'_0$ is $(\alpha^*,3\eta)$-almost regular.
	Furthermore, $G'_0$ is a robust $(\nu-\eta^2,2\tau)$-expander by
	Lemma~\ref{lem:stability-expansion}. Since the algorithm stops before $\Delta/3$ rounds,
	we have $\ell\le \Delta/3$.
	
	We apply Lemma~\ref{lem:removing-linear-forests} to $G'_0$, with $\alpha^*$,
	$\nu -\eta^2$, $2\tau$, and $\alpha^*/3$ playing the roles of
	$\alpha,\nu,\tau$, and $\eta^*$, respectively, and with $3\eta$ and
	$\eta^{\frac 14}$ playing the roles of $\varepsilon^*$ and $\varepsilon$. We obtain
	edge-disjoint spanning configurations
	$\mathcal{C}_1,\ldots,\mathcal{C}_\ell\subseteq G'_0+\bigcup_{i=1}^{\ell}F_i$
	such that $\mathcal{C}_i$ has shape $(L_i,F_i)$ for each $i\in[\ell]$. Since
	the vertices of $I_i$ are isolated in $L_i$, they remain isolated in
	$\mathcal{C}_i$. Thus the non-isolated part of $\mathcal{C}_i$, which we
	denote by $C_i$, is a cycle on $V(G'_0)\setminus I_i$ containing every edge of
	$F_i$. Moreover,
	\[
	G'_0-\bigcup_{i=1}^{\ell}(E(C_i)\setminus E(F_i))
	\]
	is a robust $(\nu'/2,8\tau)$-expander.

	We now project these cycles back to the original graph. Set $H_0=G_0$. For
	each $i\in[\ell]$, define
	\[
	E_i^*=(E(C_i)\setminus E(F_i))\cup R_i\cup M_i\cup N_i\cup N_i^*.
	\]
	Let $H_i$ be obtained from $H_{i-1}-E_i^*$ by discarding all isolated
	vertices, and let $H_i^*$ be obtained from $H_{i}-R_{i+1}$ by discarding all
	isolated vertices.
	
	Finally, suppose that $W_\ell^*\ne\emptyset$. Starting with $H_\ell^*$, apply
	Algorithm~\ref{alg:edge-deletion-wrt-W2} successively with respect to the
	singletons $\{u\}$, where $u\in W_\ell^*$ and
	$d_{H_\ell^*}(u)<\Delta(H_\ell^*)$. By Lemma~\ref{lemma:reduce-to-weak-VAL}, these applications do not
	change the maximum degree. Thus, throughout the entire deletion process, the
	maximum degree remains equal to $\Delta(H_\ell^*)$.
	
	Let $R_{\ell+1}$ be the union of all edges deleted in these applications.
	After all applications have been completed, discard all isolated vertices and
	retain the notation $H_\ell^*$ for the resulting graph. If
	$W_\ell^*=\emptyset$, set $R_{\ell+1}=\emptyset$ and leave $H_\ell^*$
	unchanged.
	
	We next verify that an inequality obtained when processing one vertex
	$u\in W_\ell^*$ is preserved during all later applications. At the end of the
	application with respect to $\{u\}$, every neighbor $v$ of $u$ in the current
	graph is adjacent to at least $\Delta(H_\ell^*)-d(u)+1$ 
	vertices of maximum degree, where $d(u)$ denotes the degree of $u$ at that
	time. In any later application, if an edge $wv$ is deleted, then $w$ is not a
	vertex of maximum degree. Hence none of the edges from $v$ to the
	maximum-degree vertices counted above is deleted. Consequently, as long as
	$uv$ remains an edge, we have $	d(v)\ge \Delta(H_\ell^*)-d(u)+1$. 
	Since $uv$ itself contributes one additional edge incident with $v$, this
	gives $d(u)+d(v)\ge \Delta(H_\ell^*)+2$. 
	Thus the inequalities obtained in the successive applications are preserved,
	and hence
	\begin{equation}
		d_{H_\ell^*}(u)+d_{H_\ell^*}(v)
		\ge \Delta(H_\ell^*)+2 
	\label{eqn:weak-VAL}
	\end{equation}
	for every edge $uv\in E(H_\ell^*)$ such that $u\in W_\ell^*$ and
	$d_{H_\ell^*}(u)<\Delta(H_\ell^*)$.
	
	Here the graphs $H_i$ correspond to the graphs $G_i$ in
	Operation~\ref{ope:edge-removing}, while the graphs $H_i^*$ correspond to the
	graphs $G_i^*$ in that operation. We now prove that these graphs satisfy the
	following properties.

	\begin{CLA}\label{claim:property-Gi}
		For each $0\le i\le \ell$, $H^*_i$ satisfies the following properties.
		\begin{enumerate}[(1)]
			\item\label{claim:property-Gi-expander}
			$H^*_i$ is a robust $(\nu'/2-2\eta^2,16\tau)$-expander.
			
			\item\label{claim:property-Gi-class2}
			$H^*_i$ is class~$2$ with $\Delta(H^*_i)=\Delta-2i$.
			
			\item\label{claim:property-Gi-weak-val}
			for every $uv\in E(H^*_i)$, we have
			\[
			d_{H^*_i}(u)+d_{H^*_i}(v)\ge \Delta-2i-2\eta n.
			\]
			
			\item\label{claim:property-Gi-no-overfull}
			$H^*_i$ contains no $\Delta(H^*_i)$-overfull subgraph.

		\end{enumerate}
	\end{CLA}
	
	\begin{proof}
		Write $\Delta_i=\Delta-2i$.
		
		For Claim~\ref{claim:property-Gi}\eqref{claim:property-Gi-expander}, by
		Lemma~\ref{lem:removing-linear-forests}, $H^*_\ell[V(G'_0)]$ is a robust
		$(\nu'/2,8\tau)$-expander. Since
		$H^*_\ell[V(G'_0)]\subseteq H^*_i[V(G'_0)]$ for every
		$0\le i\le\ell$, each $H^*_i[V(G'_0)]$ is also a robust
		$(\nu'/2,8\tau)$-expander. Moreover,
		$|V(H^*_i)\setminus V(G'_0)|\le |U_0|\le\eta^2n <2 \eta^2 |V(G'_0)|$, so
		Lemma~\ref{lem:stability-expansion} gives that $H^*_i$ is a robust
		$(\nu'/2-2\eta^2,16\tau)$-expander.
		
		For Claim~\ref{claim:property-Gi}\eqref{claim:property-Gi-class2}, by
		Lemma~\ref{lemma:reduce-to-weak-VAL}, $H^*_0$ is class~$2$ with
		$\Delta(H^*_0)=\Delta$. Let $i\ge1$ and suppose that $H^*_{i-1}$ is
		class~$2$ with $\Delta(H^*_{i-1})=\Delta_{i-1}$. 
		
		Let $\mathcal{Q}_i=H_{i-1}[E_i^*\setminus R_i]$. The graph $\mathcal{Q}_i$ is
		obtained from $C_i$ by replacing each auxiliary edge
		$x_{i,k}y_{i,k}\in F_i$ by the corresponding path in $M_i\cup N_i$, and
		each auxiliary edge $z_{i,1}z_{i,2}\in F_i$ by the path
		$z_{i,1}zz_{i,2}$ in $N_i^*$. In Case~\eqref{case:S2.4.2}, no projection is made: if
		the unique deficient vertex $s_i^*$ lies in $V(G'_0)$, it is placed in
		$I_i$, while if $s_i^*\notin V(G'_0)$, it is recorded in $J_i^*$.
		
		By the construction of $E_i^*$ and Step~3 of Operation~\ref{ope:edge-removing}, every vertex of degree
		$\Delta_{i-1}$ in $H^*_{i-1}$ loses exactly two incident edges when passing
		to $H_i$. Indeed, every uncharged vertex of $V(G'_0)$ has degree~$2$ in
		$\mathcal{Q}_i$, while the charged vertices in $I_i\cup L_i^*$ already
		carry enough deficit: vertices in $L_i^*$ lose only one incident edge but
		have degree at most $\Delta_{i-1}-1$ in $H^*_{i-1}$, and vertices in
		$I_i$ lose no incident edge but have degree at most $\Delta_{i-1}-2$ in
		$H^*_{i-1}$.
		
		For vertices outside $V(G'_0)$, all of which lie in $U_0$, we argue as
		follows. If such a vertex lies in $U_{i-1}$, then by the definition of
		$U_{i-1}$ it already has degree at most $\Delta_{i-1}-\eta n$ in
		$H^*_{i-1}$, and hence cannot have degree larger than $\Delta_i$ in
		$H_i$. If it lies in $U_0\setminus U_{i-1}$ and is handled by $N_i^*$,
		then it loses two incident edges in this round. If it lies in
		$L_i^*\setminus V(G'_0)$, then it loses one incident edge and already has
		degree at most $\Delta_{i-1}-1$. Finally, if it lies in $J_i^*$, then it
		is the unprojected deficient vertex from Case~\eqref{case:S2.4.2} and already has
		degree at most $\Delta_i$ in $H_i$. Thus no vertex has degree larger than
		$\Delta_i$ in $H_i$, and some uncharged vertex of $V(G'_0)$ has degree
		exactly $\Delta_i$. Therefore $\Delta(H_i)=\Delta_i$. 
		Since $H_i^*=H_i-R_{i+1}$, and the edges in $R_{i+1}$ are deleted through Algorithm~\ref{alg:edge-deletion-wrt-W} or 
		Algorithm~\ref{alg:edge-deletion-wrt-W2},  Lemma~\ref{lemma:reduce-to-weak-VAL} implies 
		$\Delta(H_i^*) =\Delta_i$ and $H_i^*$ is class 2 provided $H_i$ is class 2. 
		 Therefore, it  remains to show that $H_i$ is class~$2$. 
		
		Since $H^*_{i-1}$ is class~$2$,
		it suffices to prove that $\mathcal{Q}_i$ is edge-$2$-colorable. Indeed, if
		$H_i$ admitted an edge-coloring with $\Delta_i$ colors, then adding two new
		colors for an edge-$2$-coloring of $\mathcal{Q}_i$ would yield an
		edge-coloring of $H^*_{i-1}$ with
		$\Delta_i+2=\Delta_{i-1}$ colors, contradicting that $H^*_{i-1}$ is
		class~$2$.
		
		Suppose first that $|W^*_{i-1}|\ge2$. Then $G[M_i\cup N_i]$ is a linear
		forest in which at least two vertices of $W^*_{i-1}$ have degree one.
		Replacing the corresponding auxiliary edges of $C_i$ by paths in
		$M_i\cup N_i$ breaks the cycle, so $\mathcal{Q}_i$ is a linear forest and
		hence edge-$2$-colorable.
		
		Suppose that $|W^*_{i-1}|=1$. When
		$|V(G'_{i-1})\cup V(M^1_{i-1}\cup M^2_{i-1})|$ is even, the projection
		gives an even cycle. When it is odd, either the deficient vertex from
		Case~\eqref{case:S2.2.1} breaks the cycle into a linear forest, or Case~\eqref{case:S2.2.2} applies
		and the projection gives an even cycle. In either case, $\mathcal{Q}_i$ is
		edge-$2$-colorable.
		
		Finally, suppose that $|W^*_{i-1}|=0$. In
		Cases~\eqref{case:S2.3}, \eqref{case:S2.4.2},
		and~\eqref{case:S2.4.3.1}, the parity choice yields an even cycle after
		projection, while in Case~\eqref{case:S2.4.1}, the two charged vertices in
		$L_i^*$ break the cycle into a linear forest. Thus $\mathcal{Q}_i$ is
		edge-$2$-colorable.

		For Claim~\ref{claim:property-Gi}\eqref{claim:property-Gi-weak-val}, the case
		$i=0$ follows from the weak VAL property of $H_0^*$. Let $i\ge1$ and let
		$uv\in E(H_i^*)$. Recall that $\Delta_i=\Delta-2i$.
		
		Suppose first that one endpoint, say $u$, lies in
		$V(H_i^*)\setminus U_i$. By the definition of $U_i$ and
		Claim~\ref{claim:Gi-property}\eqref{claim-Gi-property-3},
		we have $d_{H_i^*}(u)\ge\Delta_i-\eta n-2\eta^2n-4$. 
		Thus  $d_{H_i^*}(u)+d_{H_i^*}(v)
		\ge\Delta_i-\eta n-2\eta^2n-3
		\ge\Delta_i-2\eta n$. 
		We may therefore assume that $u,v\in U_i$.
		
		Suppose next that $u,v\in U_i\setminus W_i^*$. By
		Claim~\ref{claim:Gi-property}\eqref{claim-Gi-property-3}, for each
		$x\in\{u,v\}$, we have $d_{H_i^*}(x)\ge d_0(x)-2\eta^2n-4-2\alpha_x$, 
		where $\alpha_x$ is the number of rounds before round $i+1$ in which $x$ is
		chosen as the special vertex $z$ in
		Case~\eqref{case:S2.4.3.1}. Since at most one vertex is chosen as the special
		vertex in each round, we have $\alpha_u+\alpha_v\le i$. Moreover,
		$uv\in E(H_i^*)\subseteq E(H_0^*)$, and the weak VAL property of $H_0^*$
		gives $d_0(u)+d_0(v)\ge\Delta+2$. Therefore
		\[
		\begin{aligned}
			d_{H_i^*}(u)+d_{H_i^*}(v)
			&\ge d_0(u)+d_0(v)-4\eta^2n-8
			-2(\alpha_u+\alpha_v) \\
			&\ge \Delta+2-2i-4\eta^2n-8 \\
			&\ge \Delta_i-\eta n.
		\end{aligned}
		\]
		
		We now consider the case in which
		$u\in U_i\setminus W_i^*$ and $v\in W_i^*$. Suppose first that no
		preprocessing edge is deleted with respect to $v$ in the first $i$ rounds.
		Since the active set containing $v$ is nonempty, the special case
		Case~\eqref{case:S2.4.3.1} does not occur. Thus $v$ loses at most two incident
		edges in each round. On the other hand, by the charging estimate in
		Claim~\ref{claim:Gi-property}\eqref{claim-Gi-property-3}, the total degree
		loss of $u$ is at most $2\eta^2n+4$. Hence
		\[
		\begin{aligned}
			d_{H_i^*}(u)+d_{H_i^*}(v)
			&\ge d_0(u)+d_0(v)-2i-2\eta^2n-4 \\
			&\ge \Delta+2-2i-2\eta^2n-4 \\
			&\ge \Delta_i-\eta n.
		\end{aligned}
		\]
		
		Suppose now that a preprocessing edge is deleted with respect to $v$. Let
		$s\le i$ be the first round in which this occurs, and let $vx$ be the first
		edge deleted with respect to $v$ in that round. Immediately before $vx$ is
		deleted, Algorithm~\ref{alg:edge-deletion-wrt-W} gives $d(v)\le e(x,W_{s-1}^*)\le |W_{s-1}^*|\le\eta^2n$, 
		where we use that the relevant graph is simple.
		
		Up to this moment, no preprocessing edge has been deleted with respect to
		$v$. Hence the estimate from the preceding subcase, applied up to round $s$,
		gives $d(u)+d(v)\ge\Delta_s-\eta n$. 
		It follows that, at this moment, $d(u)\ge\Delta_s-\eta n-\eta^2n$.

		From this point until $H_i^*$ is reached, $u$ loses at most two incident edges
		in each of the remaining rounds. By the charging estimate in
		Claim~\ref{claim:Gi-property}\eqref{claim-Gi-property-3}, all additional
		losses at $u$ total at most $2\eta^2n+4$. Since the edge $uv$ remains in
		$H_i^*$, we also have $d_{H_i^*}(v)\ge1$. Consequently,
		\[
		\begin{aligned}
			d_{H_i^*}(u)+d_{H_i^*}(v)
			&\ge
			\Delta_s-\eta n-\eta^2n
			-2(i-s)-2\eta^2n-4+1 \\
			&=
			\Delta_i-\eta n-3\eta^2n-3 \\
			&\ge \Delta_i-2\eta n.
		\end{aligned}
		\]
		
		It remains to consider the case $u,v\in W_i^*$. Since $W_0$ is independent
		in $H_0^*$ and $uv\in E(H_i^*)$, the set $W_i^*$ must be inherited from an
		auxiliary set $W'_j$ for some unique $j\le i$ with $|W'_j|=2$, where the
		uniqueness of $j$ follows from
		Claim~\ref{claim:Gi-property}\eqref{claim-Gi-property-1.5}. At the time when
		$W'_j=\{u,v\}$ is chosen, both $u$ and $v$ lie in $U_j\setminus W_j$.
		Therefore, by the second case above,  $d_{H_j^*}(u)+d_{H_j^*}(v)\ge\Delta_j-2\eta n$.

		By Claim~\ref{claim:Gi-property}\eqref{claim-Gi-property-1.5},
		$\{u,v\}$ remains the inherited auxiliary active set in every later round.
		Moreover, since $\{u,v\}$ is not independent,
		Algorithm~\ref{alg:edge-deletion-wrt-W} is not applied to this set in any
		later round. Thus, in each round from $j+1$ through $i$, Step~3 removes
		exactly one prescribed edge incident with each of $u$ and $v$. Therefore
		\[
		\begin{aligned}
			d_{H_i^*}(u)+d_{H_i^*}(v)
			&\ge d_{H_j^*}(u)+d_{H_j^*}(v)-2(i-j) \\
			&\ge \Delta_j-2\eta n-2(i-j) \\
			&=\Delta_i-2\eta n.
		\end{aligned}
		\]
		This proves Claim~\ref{claim:property-Gi}\eqref{claim:property-Gi-weak-val}.

		For Claim~\ref{claim:property-Gi}\eqref{claim:property-Gi-no-overfull}, 
	since  $H^*_\ell[V(G'_0)]\subseteq H_i[V(G'_0)]$
		and $|V(H_i)\setminus V(G'_0)|\le |U_0|\le\eta^2n<2\eta^2 |V(G_0')|$, the same argument as in
		Claim~\ref{claim:property-Gi}\eqref{claim:property-Gi-expander}, together with
		Lemma~\ref{lem:stability-expansion}, shows that $H_i$ is also a robust
		$(\nu'/2-2\eta^2,16\tau)$-expander.
		Since $H_i^*\subseteq H_i$ and both graphs have maximum degree
		$\Delta_i$, it suffices to show that $H_i$ contains no
		$\Delta_i$-overfull subgraph.  
		The case $i=0$ follows since
		$H_0=G_0$ contains no $\Delta$-overfull subgraph. 
			Let $i\ge1$, and suppose for a contradiction that there is an odd set
		$X\subseteq V(H_i)$ such that $H_i[X]$ is $\Delta_i$-overfull.  Let $Y=V(H_{i-1})\setminus X$. 
		 
		We first show that $|Y|$ is small. Since $H_i$ is simple and $H_i[X]$ is
		$\Delta_i$-overfull, we have $|X|>\Delta_i\ge\Delta/3$. This implies 
		$e_{H_i}\bigl(X,V(H_i)\setminus X\bigr)
		\le \df_{H_i}(X)<\Delta_i\le n$. 
		On the other hand, if both $X$ and $V(H_i)\setminus X$ had size at least
		$16\tau n$, then the robust expansion of $H_i$  would imply 
		$e_{H_i}\bigl(X,V(H_i)\setminus X\bigr)\ge
		(\nu'/2-2\eta^2)^2n^2>n$, 
		a contradiction. Since $|X|>\Delta/3>16\tau n$, it follows that
		$|V(H_i)\setminus X|<16\tau n$. The vertices of
		$V(H_{i-1}^*)\setminus V(H_i)$ lie in $U_0$, and hence 
		$|Y|<17\tau n$.

		We next show that all edges of $H_{i-1}[Y]$ are incident with a common
		vertex. Suppose first that $H_{i-1}[Y]$ contains two independent edges
		$u_1v_1$ and $u_2v_2$. By
		Claim~\ref{claim:property-Gi}\eqref{claim:property-Gi-weak-val} for
		$H_{i-1}$ ($H^*_{i-1}$ is a subgraph of $H_{i-1}$ with the same maximum degree),
		\[
		d_{H_{i-1}}(u_j)+d_{H_{i-1}}(v_j)
		\ge\Delta_{i-1}-2\eta n
		\quad\text{for }j\in[2].
		\]
		Since $H_{i-1}$ is simple, the total number of edges from these four
		vertices to $Y$ is at most $4(|Y|-1)$. Moreover, each vertex loses at most
		two incident edges when passing from $H_{i-1}$ to $H_i$. Therefore
		\[
		\begin{aligned}
			e_{H_i}\bigl(X,V(H_i)\setminus X\bigr)
			&\ge 2(\Delta_{i-1}-2\eta n)-4(|Y|-1)-8 >\Delta_i,
		\end{aligned}
		\]
		where the last inequality follows since  $|Y|<17\tau n$ and
		$\eta,\tau\ll\alpha$. This contradicts
		$e_{H_i}(X,V(H_i)\setminus X)<\Delta_i$.
		
		Similarly, $H_{i-1}[Y]$ cannot contain a triangle with vertices
		$u_1,u_2,u_3$. Indeed, applying the weak VAL inequality to the three edges of
		the triangle gives
		\[
		d_{H_{i-1}}(u_1)+d_{H_{i-1}}(u_2)
		+d_{H_{i-1}}(u_3)
		\ge \frac{3}{2}(\Delta_{i-1}-2\eta n).
		\]
		Thus
		\[
		\begin{aligned}
			e_{H_i}\bigl(X,V(H_i)\setminus X\bigr)
			&\ge \frac{3}{2}(\Delta_{i-1}-2\eta n)
			-3(|Y|-1)-6 >\Delta_i,
		\end{aligned}
		\]
		again a contradiction. A simple graph with no two independent edges is
		either a star or a triangle; since we have ruled out a triangle, it follows that all edges
		of $H_{i-1}[Y]$ are incident with a common vertex.
		
		We now show that $Y$ contains at most one vertex of $V(G_0)\setminus U_0$.
		Suppose that $y_1,y_2\in Y\setminus U_0$ are distinct. By the degree bound for
		vertices outside $U_0$ from Claim~\ref{claim:Gi-property}\eqref{claim-Gi-property-3}, 
		\[
		d_{H_{i-1}}(y_j)
		\ge \Delta_{i-1}-\eta n-\eta^2n
		\quad\text{for }j\in[2].
		\]
		Since $H_{i-1}^*$ is simple and each vertex loses at most two incident edges
		in round $i$, we obtain
		\[
		\begin{aligned}
			e_{H_i}\bigl(X,V(H_i)\setminus X\bigr)
			&\ge
			2(\Delta_{i-1}-\eta n-\eta^2n)
			-2(|Y|-1)-4 >\Delta_i,
		\end{aligned}
		\]
		contradicting the overfullness of $H_i[X]$. Hence
		$|Y|\le |U_0|+1\le\eta^2n+1$. 
		Since all edges of $H_{i-1}^*[Y]$ are incident with a common vertex, it follows
		that $e_{H_{i-1}}(Y)\le |Y|-1\le\eta^2n$.

Let
\[
B^*_i=U_0\cup I_i\cup L_i^*\cup J_i^*.
\]
Then $|B^*_i|\le\eta^2n+4$. Let $Q_i$ be the subgraph consisting of the edges
deleted when passing from $H^*_{i-1}$ to $H_i$. By construction,
$\Delta(Q_i)\le2$, and every vertex outside $B^*_i$ has degree exactly two in
$Q_i$. Since $\Delta_i=\Delta_{i-1}-2$, we have
\[
\begin{aligned}
	\df_{H_i}(X)
	&=\df_{H^*_{i-1}}(X)-2|X|+2e_{Q_i}(X) \\
	&=\df_{H^*_{i-1}}(X)
	-\sum_{x\in X}\bigl(2-d_{Q_i}(x)\bigr)
	-e_{Q_i}\bigl(X,V(H^*_{i-1})\setminus X\bigr).
\end{aligned}
\]
The first error term is at most $2|B^*_i|$, while
\[
e_{Q_i}\bigl(X,V(H^*_{i-1})\setminus X\bigr)
\le 2|V(H^*_{i-1})\setminus X|
\le 2|Y|,
\]
since $\Delta(Q_i)\le2$ and
$V(H^*_{i-1})\setminus X\subseteq Y$. Therefore
\[
\begin{aligned}
	\df_{H_i}(X)
	&\ge \df_{H^*_{i-1}}(X)-2|B^*_i|
	-2|V(H^*_{i-1})\setminus X| \\
	&\ge \df_{H^*_{i-1}}(X)-2|B^*_i|-2|Y| \\
	&\ge \df_{H^*_{i-1}}(X)-4\eta^2n-10 \\
	&> \df_{H^*_{i-1}}(X)-5\eta^2n.
\end{aligned}
\]

Since $H_i[X]$ is $\Delta_i$-overfull, we have $\df_{H_i}(X)<\Delta_i$, and hence $\df_{H^*_{i-1}}(X)<\Delta_i+5\eta^2n$. Moreover, $H^*_{i-1}=H_{i-1}-R_i$ and $\Delta(H^*_{i-1})=\Delta(H_{i-1})=\Delta_{i-1}$. Since $H^*_{i-1}\subseteq H_{i-1}$, we have $ e_{H^*_{i-1}}(X)\le e_{H_{i-1}}(X)$, and therefore 
\[ \df_{H_{i-1}}(X)\le\df_{H^*_{i-1}}(X) <\Delta_i+5\eta^2n. \]

	By the definition of the auxiliary degree records,
	$d_{i-1}(v)=d_{H_{i-1}}(v)$ for every $v\in V(H_{i-1})$. Therefore
	\[
	\begin{aligned}
		&\sum_{v\in X}\bigl(\Delta_{i-1}-d_{i-1}(v)\bigr)
		+\sum_{u\in Y}d_{i-1}(u) \\
		&\quad=
		\df_{H_{i-1}}(X)+2e_{H_{i-1}}(Y) \\
		&\quad<
		\Delta_i+5\eta^2n+2\eta^2n \\
		&\quad<
		\Delta_{i-1}+7\eta^2n.
	\end{aligned}
	\]
	Thus $X$ is almost-overfull at the beginning of round $i$, contradicting the
	condition under which Operation~\ref{ope:edge-removing} continues.
	\end{proof}

	\begin{center}
		{\bf \noindent Stage II: Regularization.}
	\end{center}
	
	In the remaining part of the proof, we modify $H^*_\ell$ into an auxiliary
	multigraph $D$, and then complete $D$ to a regular multigraph to which
	Theorem~\ref{thm:robust-expander2} can be applied. Write
	\[
	\Delta_\ell=\Delta(H^*_\ell)=\Delta-2\ell.
	\]
	
	In each case, $D$ will be obtained from $H^*_\ell$, or from a graph obtained from
	$H^*_\ell$ after a further controlled deletion, by a small modification so that
	$|V(D)|$ is even. Write
	\[
	\Delta_0:=\Delta(D).
	\]
	The value of $\Delta_0$ will be verified in each case. In several cases
	$\Delta_0=\Delta_\ell$, while after the Matching-Removal Operation (Operation~\ref{ope:stage1-matching-removal-case2}) we may have
	$\Delta_0<\Delta_\ell$.  The case analysis after Claim~\ref{claim:regular-completion} distinguishes whether or not 
	Operation~\ref{ope:edge-removing} stopped because an almost-overfull set 
	appeared.  
	The completion step is the same in all cases, but the
	verification that $D$ satisfies the required hypotheses is case-dependent.
Thus we treat the completion step in the following claim as a generic step.
	
	\begin{CLA}\label{claim:regular-completion}
		Let $0<1/m\ll \eta\ll \nu'\le \tau\ll \alpha<1$, and let $D$ be a multigraph
		on $\{u_1,\ldots,u_m\}$ with
		$d_D(u_1)\le d_D(u_2)\le\cdots\le d_D(u_m)$. Put
		$\Delta_0=\Delta(D)$. Suppose that the following conditions hold.
		\begin{enumerate}[(1)]
			\item\label{claim:regular-completion-order}
			$m$ is even and $m>\Delta_0\ge \alpha m/3$.
			
			\item\label{claim:regular-completion-simple}
			$D-\{u_1,u_2\}$ is simple and $e_D(u_1,u_i), e_D(u_2,u_i) \le 2\eta^2 m$
			for all $i\ge 3$.
			
			\item\label{claim:regular-completion-no-overfull}
			$D$ contains no $\Delta_0$-overfull subgraph.
			
			\item\label{claim:regular-completion-degree}
			Either $d_D^s(u_i)\ge \Delta_0-6\eta m$ for every $i\ge 3$, or
			$\df_D(D-u_1)\le \Delta_0+5\eta m$.
			
			\item\label{claim:regular-completion-local}
			If $|N_D(u_1)\setminus\{u_2\}|\le 1$, then either $D$ has exactly three
			vertices of degree less than $\Delta_0$, or
			\[
			e_D(\{u_1,u_2,u\},V(D)\setminus\{u_1,u_2,u\})\ge \Delta_0
			\]
			for every $u\in V(D)\setminus\{u_1,u_2\}$.
			
			\item\label{claim:regular-completion-u2}
			If $u_2$ is not simple in $D$, then $u_1u_2\notin E(D)$; for each
			$i\in[2]$ and every $w\in V(D)\setminus\{u_1,u_2\}$ with
			$u_iw\in E(D)$, we have
			\[
			e_D(u_i,w)\le e_D(u_i,V(D)\setminus\{w\})
			\quad\text{and}\quad
			d_D(u_i)+d_D(w)\ge \Delta_0+2;
			\]
			and also $d_D(u_2)\le10\eta m$, while
			$d_D^s(u_i)\ge \Delta_0-5\eta m$ for every $i\ge 3$.
			
			\item\label{claim:regular-completion-expander}
			$D^s$ is a robust $(\nu'/3, 32\tau)$-expander.
		\end{enumerate}
		Then $D$ can be completed to a $\Delta_0$-regular multigraph
		$\widehat D$ such that $\chi'(\widehat D)=\Delta_0$.
	\end{CLA}
	
	\begin{proof}
		For each $i\in[m]$, let $d_i=\Delta_0-d_D(u_i)$, so
		$(d_1,\ldots,d_m)$ is a non-increasing sequence of nonnegative integers with
		\[
		\sum_{i=1}^m d_i=m\Delta_0-2e(D)
		\]
		even since $m$ is even. For admissibility it suffices to check
		$d_1-2\le\sum_{i=2}^m d_i$. Let $S=V(D)\setminus\{u_1\}$; then $|S|=m-1$ is
		odd, and Condition~\eqref{claim:regular-completion-no-overfull} gives
		$e_D(S)\le\Delta_0(|S|-1)/2$, so
		\[
		\sum_{i=2}^m d_i-(d_1-2)=(m-2)\Delta_0-2e_D(S)+2\ge 2.
		\]
		
		Applying Lemma~\ref{lem:graphical-biparite} to $(d_1,\ldots,d_m)$ yields a
		bipartite multigraph $L$ on $\{u_1,\ldots,u_m\}$ and an even index
		$p\in[2,m]$ satisfying the conclusion of that lemma, with
		$d_{D^+}(u_1)=d_{D^+}(u_2)=d_D(u_2)$ where $D^+=D\cup L$.
		We now let $\widehat D=D\cup L\cup N$ be the $\Delta_0$-regular 
		multigraph that is the  canonical  completion of $D^+$ defined in Definition~\ref{def:canonical-completion}.

		We verify the hypotheses of Theorem~\ref{thm:robust-expander2} for
		$\widehat D$, with $v_i=u_i$ for every $i\in[m]$.
		
		Condition~\eqref{item:thm4.1}: $\widehat D=D\cup L\cup N$ is the edge-disjoint
		union of multigraphs on $\{u_1,\ldots,u_m\}$, with
		$\Delta(\widehat D)=\Delta_0=\Delta(D)$, and $L$ was constructed from the
		deficiency sequence $d_i=\Delta(D)-d_D(u_i)$ with vertices ordered
		$d_D(u_1)\le\cdots\le d_D(u_m)$.
		
		Condition~\eqref{item:thm4.1.a-new}: this is the second part of 
		Condition~\eqref{claim:regular-completion-simple}.
		
		Condition~\eqref{item:thm4.1.b}: if the first alternative of
		Condition~\eqref{claim:regular-completion-degree} holds, then
		$d_D^s(u_i)\ge\Delta_0-6\eta m$ for every $i\ge 3$, and since adding edges
		from $L\cup N$ cannot decrease simple degree,
		\[
		d_{\widehat D}^s(u_i)\ge\Delta_0-6\eta m
		\ge \frac{i-2}{i-1}\Delta_0-6\eta m
		\]
		for every $i\ge 3$.
		
		If instead the second alternative holds, so
		\[
		\sum_{j=2}^m(\Delta_0-d_D(u_j)) \le\sum_{j=2}^m(\Delta_0-d_{D-u_1}(u_j))=\df_D(D-u_1)\le \Delta_0+5\eta m,
		\]
		then for every $i\ge 3$,
		\[
		(i-1)(\Delta_0-d_D(u_i))
		\le \sum_{j=2}^i(\Delta_0-d_D(u_j))
		\le \Delta_0+5\eta m.
		\]
		Hence
		\[
		d_D(u_i)\ge \frac{i-2}{i-1}\Delta_0-\frac{6\eta m}{i-1}.
		\]
		If $u_2$ is not simple in $D$, Condition~\eqref{claim:regular-completion-u2}
		gives $d_D^s(u_i)\ge\Delta_0-6\eta m$ for every $i\ge 3$, and adding edges
		from $L\cup N$ gives the required bound. So assume $u_2$ is simple in $D$.
		By Condition~\eqref{claim:regular-completion-simple}, $D-\{u_1,u_2\}$ is
		simple and $e_D(u_i,u_1)\le 2\eta^2m$ for every $i\ge 3$, so
		\[
		d_D^s(u_i)\ge d_D(u_i)-e_D(u_i,u_1)
		\ge\frac{i-2}{i-1}\Delta_0-\frac{6\eta m}{i-1}-2\eta^2m
		\ge\frac{i-2}{i-1}\Delta_0-6\eta m.
		\]
		Since adding edges from $L\cup N$ cannot decrease simple degree,
		$d_{\widehat D}^s(u_i)\ge\frac{i-2}{i-1}\Delta_0-6\eta m$ for every $i\ge 3$.
		
		Condition~\eqref{item:thm4.1.c}: by
		Condition~\eqref{claim:regular-completion-degree}, either
		$d_D^s(u_i)\ge\Delta_0-6\eta m$ for every $i\ge 3$, or
		$\df_D(D-u_1)\le\Delta_0+5\eta m$. In the first case, since adding edges from
		$L\cup N$ cannot decrease simple degree,
		\[
		d_{\widehat D}^s(u_i)\ge\Delta_0-6\eta m
		\ge \Delta_0-e_{\widehat D}(u_1,u_i)-6\eta m
		\]
		for every $i\ge 3$.
		
		So assume $\df_D(D-u_1)\le\Delta_0+5\eta m$. Since
		$d_D(u_1)+d_L(u_1)=d_{D^+}(u_1)=d_{D^+}(u_2)$ by Lemma~\ref{lem:graphical-biparite}(a),
		\[
		\begin{aligned}
			\df_D(D-u_1)
			&\ge d_D(u_1)+d_L(u_1)+\Delta_0-d_{D^+}(u_2)
		+\sum_{i=3}^m(\Delta_0-d_{D^+}(u_i))\\
			&=\Delta_0+\sum_{i=3}^m(\Delta_0-d_{D^+}(u_i)),
		\end{aligned}
		\]
		so $\sum_{i=3}^m(\Delta_0-d_{D^+}(u_i))\le 6\eta m$. When $p<m$, the residual
		deficiency of $u_{p+1}$ in $D^+$ is $2q$, hence $q\le 3\eta m$, and the
		edges of $N$ create at most $6\eta m$ parallel edges at any vertex $u_i$ with
		$i\ge 3$, apart from the bundle to $u_1$. Hence
		\[
		d_{\widehat D}^s(u_i)\ge
		\Delta_0-e_{\widehat D}(u_1,u_i)-6\eta m
		\]
		for every $i\ge 3$.
		
		For $i=2$: if $u_2$ is simple in $D$,
		\[
		d_{\widehat D}^s(u_2)\ge\Delta_0-e_{\widehat D}(u_1,u_2)-e_N(u_2,u_{p+1})
		\ge\Delta_0-e_{\widehat D}(u_1,u_2)-3\eta m,
		\]
		using $e_N(u_2,u_{p+1})\le q\le 3\eta m$, or $0$ if the edge is not created
		in the construction of $N$. If $u_2$ is not simple in $D$,
		Condition~\eqref{claim:regular-completion-u2} gives $d_D(u_2)\le 10\eta m$,
		so
		\[
		d_{\widehat D}^s(u_2)\ge
		\Delta_0-e_{\widehat D}(u_1,u_2)-d_D(u_2)-e_N(u_2,u_{p+1})
		\ge \Delta_0-e_{\widehat D}(u_1,u_2)-13\eta m.
		\]

		Condition~\eqref{item:thm4.2}: $D^s$ is a robust $(\nu'/3,32\tau)$-expander
		by Condition~\eqref{claim:regular-completion-expander}, and adding the edges of
		$L\cup N$ cannot destroy robust expansion, so the underlying simple graph of
		$\widehat D$ is also a robust $(\nu'/3, 32\tau)$-expander. Thus $\widehat D$
		satisfies all hypotheses of Theorem~\ref{thm:robust-expander2}, with
		$\nu'/3$, $32\tau$, and $\alpha/3$ playing the roles of $\nu$, $\tau$, and
		$\alpha$.
		
		It remains to show $\widehat D$ contains no $\Delta_0$-overfull subgraph.
		Suppose $X\subseteq V(\widehat D)$ is odd with $\widehat D[X]$
		$\Delta_0$-overfull. Since $\widehat D$ is $\Delta_0$-regular,
		$e_{\widehat D}(X,V(\widehat D)\setminus X)<\Delta_0$. Replacing $X$ by its
		complement if necessary, assume $|X|\le m/2$.
		
		If $32\tau m\le|X|\le m/2$, robust expansion of $D^s$ gives more than
		$\nu'm/3$ vertices in $V(D)\setminus X$ each with at least $\nu'm/3$ neighbors
		in $X$, so $e_D(X,V(D)\setminus X)\ge(\nu'm/3)^2>\Delta_0$, a contradiction.
		Hence $|X|<32\tau m$.
		
		By the bound established for Condition~\eqref{item:thm4.1.b},
		$d_D^s(u_i)\ge\frac{i-2}{i-1}\Delta_0-6\eta m$ for every $i\ge 3$. At most
		one vertex of $X$ lies in $\{u_3,\ldots,u_m\}$: otherwise choose
		$u_a,u_b\in X$ with $3\le a<b$, arranged so $b\ge 4$, and then
		\[
		d_D^s(u_a)+d_D^s(u_b)
		\ge\left(\frac{1}{2}+\frac{2}{3}\right)\Delta_0-12\eta m.
		\]
		Since each of $u_a,u_b$ has at most $|X|-1$ simple neighbors in $X$ and
		$|X|<32\tau m$,
		\[
		e_D(X,V(D)\setminus X)\ge d_D^s(u_a)+d_D^s(u_b)-2(|X|-1)>\Delta_0,
		\]
		by the hierarchy $\eta\ll\tau\ll\alpha$ and $\Delta_0\ge\alpha m/3$,
		contradicting $e_{\widehat D}(X,V(\widehat D)\setminus X)<\Delta_0$.
		
		So $X=\{u_1,u_2,u\}$ for some $u\in V(D)\setminus\{u_1,u_2\}$.
		
		If $u_2$ is not simple in $D$, Condition~\eqref{claim:regular-completion-u2}
		gives $u_1u_2\notin E(D)$ and, for every edge from $u_i$ to $u$ with
		$i\in[2]$,
		\[
		e_D(u_i,u)\le e_D(u_i,V(D)\setminus\{u\})
		\quad\text{and}\quad
		d_D(u_i)+d_D(u)\ge\Delta_0+2,
		\]
		which gives at least $\Delta_0$ edges leaving $\{u_1,u_2,u\}$, contradicting
		that $\widehat D[X]$ is $\Delta_0$-overfull.
		
		So assume $u_2$ is simple in $D$. If $d^s_{\widehat D}(u_1)\ge 3$, then
		$u_1$ has a neighbor outside $\{u_2,u\}$, and simplicity of $u_2$ gives at
		least $\Delta_0$ edges leaving $X$, a contradiction. Hence
		$|N_{\widehat D}(u_1)\setminus\{u_2\}|\le 1$. By
		Condition~\eqref{claim:regular-completion-local}, either $D$ has exactly three
		vertices of degree less than $\Delta_0$, or
		\[
		e_D(\{u_1,u_2,u\},V(D)\setminus\{u_1,u_2,u\})\ge\Delta_0.
		\]
		The second alternative gives at least $\Delta_0$ edges leaving $X$, a
		contradiction. In the first alternative, the construction of $L\cup N$ forces
		the exceptional three-set to be precisely the deficient vertices; its complement
		is $\Delta_0$-overfull in $D$, contradicting
		Condition~\eqref{claim:regular-completion-no-overfull}.
		
		Thus $\widehat D$ contains no $\Delta_0$-overfull subgraph, and
		Theorem~\ref{thm:robust-expander2} gives
		$\chi'(\widehat D)=\Delta(\widehat D)=\Delta_0$.
	\end{proof}
	
	Since Operation~\ref{ope:edge-removing} stopped after the $\ell$ completed
	rounds, either
	\[
	\begin{cases}
		\ell>\Delta/3-\eta n,
		& \text{if } W^*_\ell\subseteq W^*_{\ell-1}\subseteq\cdots
		\subseteq W^*_1\subseteq W_0,\\
		\ell>\Delta/3-2\eta n,
		& \text{otherwise,}
	\end{cases}
	\]
	or $G_\ell$ has an almost-overfull set. By
	Claim~\ref{claim:regular-completion}, in each case it suffices to construct a
	multigraph $D$ satisfying
	Conditions~\eqref{claim:regular-completion-order}--\eqref{claim:regular-completion-expander}
	with $\chi'(D)>\Delta(D)$.
	
	\smallskip
	\noindent\textbf{Case 1.} $G_\ell$ has no almost-overfull set.
	
	\smallskip
	\noindent\textbf{Case 1.1.}
	$W^*_\ell\subseteq W^*_{\ell-1}\subseteq\cdots\subseteq W^*_1\subseteq W_0$.
	
	Then $\ell>\Delta/3-\eta n$. Since the operation was performed for $\ell$
	rounds and did not stop earlier, $\ell-1\le\Delta/3-\eta n$, and hence
	\[
	\Delta/3+2\eta n-2\le\Delta_\ell<\Delta/3+2\eta n.
	\]
	Since $W^*_\ell\subseteq\cdots\subseteq W_0$, the process never entered the
	critical case before round $\ell$, so $W_i^*=W_i$ for every $0\le i\le\ell$.
	If $|W_\ell|\ge 2$, then by Step~1 of Operation~\ref{ope:edge-removing}, every
	$v\in W_\ell$ satisfies
	\[
	d_{H^*_\ell}(v)\le d_{G_0}(v)-\ell
	<\Delta/3-\eta n-(\Delta/3-\eta n)=0,
	\]
	a contradiction. Hence $|W_\ell|\le 1$, and every vertex in
	$V(H^*_\ell)\setminus W_\ell$ has degree at least
	$\Delta/3-\eta n-2\eta^2n-4\ge\Delta_\ell-3\eta n$ in $H^*_\ell$ by Claim~\ref{claim:Gi-property}\eqref{claim-Gi-property-3}.
	
	Suppose first that $|V(H^*_\ell)|$ is even. Let $D=H^*_\ell$, ordered as
	$u_1,\ldots,u_m$ with $d_D(u_1)\le\cdots\le d_D(u_m)$, and
	$\Delta_0=\Delta_\ell$. Condition~\eqref{claim:regular-completion-order}
	follows from the parity assumption and the lower bound on $\Delta_\ell$.
	Condition~\eqref{claim:regular-completion-simple} holds since $D=H^*_\ell$ is
	simple: $D-\{u_1,u_2\}$ is simple and
	$e_D(u_1,u_i),e_D(u_2,u_i)\le 1\le 2\eta^2 m$ for every $i\ge 3$.
	Conditions~\eqref{claim:regular-completion-no-overfull} and
	\eqref{claim:regular-completion-expander} follow from
	Claim~\ref{claim:property-Gi}\eqref{claim:property-Gi-no-overfull},
	Claim~\ref{claim:property-Gi}\eqref{claim:property-Gi-expander}, and
	Lemma~\ref{lem:stability-expansion}. Since at most one vertex lies in
	$W_\ell$, every $u_i$ with $i\ge 3$ lies outside $W_\ell$, so
	$d_D^s(u_i)=d_D(u_i)\ge\Delta_0-3\eta n$ for every $i\ge 3$, giving
	Condition~\eqref{claim:regular-completion-degree}.
	
	For Condition~\eqref{claim:regular-completion-local}, suppose
	$|N_D(u_1)\setminus\{u_2\}|\le 1$. Since $H^*_\ell$ has no isolated vertices,
	$d_D(u_1)\ge 1$. If $d_D(u_1)=1$, then $u_1\in W_\ell^*$, and the cleanup
	condition~\eqref{eqn:weak-VAL} gives
	$d_D(u_1)+d_D(u_2)\ge\Delta_0+2$, a contradiction. Hence $d_D(u_1)=2$ and
	$N_D(u_1)=\{u_2,u\}$ for some $u\ne u_1,u_2$. Applying~\eqref{eqn:weak-VAL} to both edges incident with $u_1$ gives
	$d_D(u_2)=d_D(u)=\Delta_0$, so
	\[
	e_D(\{u_1,u_2,u\},V(D)\setminus\{u_1,u_2,u\})
	=d_D(u_1)+d_D(u_2)+d_D(u)-2e_D(\{u_1,u_2,u\})\ge 2\Delta_0-4>\Delta_0.
	\]
	Condition~\eqref{claim:regular-completion-u2} is vacuous since $D$ is simple.
	Since $\chi'(H^*_\ell)>\Delta_\ell$ by
	Claim~\ref{claim:property-Gi}\eqref{claim:property-Gi-class2}, applying
	Claim~\ref{claim:regular-completion} gives a $\Delta_0$-regular multigraph
	$\widehat D$ with $\chi'(\widehat D)=\Delta_0$, so restricting to $E(D)$
	yields a $\Delta_0$-coloring of $D$, contradicting $\chi'(D)>\Delta(D)$.
	
	Now suppose $|V(H^*_\ell)|$ is odd. If $H^*_\ell$ has at least three vertices of
	degree less than $\Delta_\ell$, let $D$ be obtained from $H^*_\ell$ by adding a
	new vertex $x$ joined to two vertices of largest degree among those with degree
	less than $\Delta_\ell$. Otherwise $H^*_\ell$ has exactly two such vertices
	(since $H^*_\ell$ has odd order and contains no $\Delta_\ell$-overfull subgraph),
	and let $D$ be the disjoint union of $H^*_\ell$ and an isolated vertex $x$. In
	both cases $\Delta_0=\Delta(D)=\Delta_\ell$ and $m=|V(D)|$ is even.
	
	Order the vertices of $D$ as $u_1,\ldots,u_m$ by nondecreasing degree.
	Condition~\eqref{claim:regular-completion-order} is immediate.
	Condition~\eqref{claim:regular-completion-simple} holds since adding a new
	vertex creates no parallel edges in $H^*_\ell$: $D-\{u_1,u_2\}$ is simple and
	$e_D(u_1,u_i),e_D(u_2,u_i)\le 1\le 2\eta^2 m$ for every $i\ge 3$.
	Condition~\eqref{claim:regular-completion-expander} follows from
	Lemma~\ref{lem:stability-expansion}. For
	Condition~\eqref{claim:regular-completion-no-overfull}: if $x$ is isolated, an
	odd set containing $x$ corresponds to an even set in $H^*_\ell$, while an odd set
	avoiding $x$ is handled by
	Claim~\ref{claim:property-Gi}\eqref{claim:property-Gi-no-overfull}; if $x$ has
	degree two, every vertex of $W_\ell^*$ has degree at least $2$ in $H^*_\ell$, so
	$x$ has minimum degree, and Lemma~\ref{lem:overfull-delete-one-vertex} rules
	out a $\Delta_0$-overfull set containing $x$, while one avoiding $x$ is already overfull
	in $H^*_\ell$. The same degree estimates give
	Condition~\eqref{claim:regular-completion-degree}; the verification of
	Condition~\eqref{claim:regular-completion-local} is identical to the
	even-order case; and Condition~\eqref{claim:regular-completion-u2} is vacuous.
	Applying Claim~\ref{claim:regular-completion} gives a contradiction as before.
	
	\smallskip
	\noindent\textbf{Case 1.2.}
	It is not the case that
	$W^*_\ell\subseteq W^*_{\ell-1}\subseteq\cdots\subseteq W^*_1\subseteq W_0$.
	
	Let $t\in[\ell]$ be the unique index for which $W_t^*=W'_t$.

	Suppose first that $t>\Delta/3-2\eta n$. Then the operation stops immediately
	after round $t$, and hence $t=\ell$. Since
	$\Delta/3-2\eta n<t\le\Delta/3-\eta n$, we have
	\[
	\Delta/3+2\eta n\le\Delta_\ell<\Delta/3+4\eta n.
	\]
	Since $W_t=\emptyset$, Claim~\ref{claim:Gi-property}\eqref{claim-Gi-property-3} gives
	$d_{H^*_t}(v)\ge\Delta/3-\eta n-2\eta^2n-4$ for every
	$v\in W'_t$.  Consequently,
	\[
	\Delta_\ell-d_{H_\ell^*}(v)
	<5\eta n+2\eta^2n+4<6\eta n
	\]
	for every $v\in W'_t$.
	
	By the parity choice in the critical case, $|V(H_\ell^*)|$ is odd. Construct
	$D$ as in the odd-order part of Case~1.1. The order, simplicity,
	no-overfullness, and robust-expansion conditions follow exactly as there.
	Moreover, Claim~\ref{claim:Gi-property}\eqref{claim-Gi-property-3} gives
	$d_{H_\ell^*}(v)>\Delta_\ell-6\eta n$ for every vertex
	$v\in V(H_\ell^*)$, so
	Condition~\eqref{claim:regular-completion-degree} also holds. The local
	condition follows from~\eqref{eqn:weak-VAL}  and the two-neighbor argument
	from Case~1.1, while Condition~\eqref{claim:regular-completion-u2} is
	vacuous. Applying Claim~\ref{claim:regular-completion} gives a contradiction.
	
	We may therefore assume that $t\le\Delta/3-2\eta n$. Then
	$\ell>\Delta/3-2\eta n$. Since the operation was performed for $\ell$ rounds
	and did not stop earlier, $\ell-1\le\Delta/3-2\eta n$, and hence
	\[
	\Delta/3+4\eta n-2\le\Delta_\ell<\Delta/3+4\eta n.
	\]
	Then
	by Claim~\ref{claim:Gi-property}\eqref{claim-Gi-property-3}, 
	$d_{H^*_t}(v)\ge\Delta/3-\eta n-2\eta^2 n-4$ for every $v\in W'_t$. 
	Since
	$\ell-1\le\Delta/3-2\eta n$, Step~1 of Operation~\ref{ope:edge-removing}
	gives $d_{H^*_\ell}(v)\ge\eta n$ for every $v\in W'_t$, so $W^*_\ell=W'_t$.
	
	If $|W^*_\ell|=1$, then $|V(H^*_\ell)|$ is odd and we construct $D$ as in the
	odd-order part of Case~1.1. If $|W^*_\ell|=2$, then $|V(H^*_\ell)|$ is even and
	$D=H^*_\ell$. In either case $\Delta_0=\Delta(D)=\Delta_\ell$.
	Conditions~\eqref{claim:regular-completion-order},
	\eqref{claim:regular-completion-simple},
	\eqref{claim:regular-completion-no-overfull}, and
	\eqref{claim:regular-completion-expander} follow as in Case~1.1. The only
	possible vertices with degree below $\Delta/3-\eta n$ lie in $W^*_\ell$, hence
	among $u_1,u_2$, so
	\[
	d_D^s(u_i)\ge\Delta/3-\eta n-2\eta^2 n-4\ge\Delta_\ell-2\eta n>\Delta_0-6\eta n
	\]
	for every $i\ge 3$, giving Condition~\eqref{claim:regular-completion-degree}.
	If Condition~\eqref{claim:regular-completion-local} needs to be checked, then,
	by the lower degree bound, we must be in the odd-order case. In this case,
	$u_1$ is the newly added vertex, and the verification is identical to that in
	the odd-order part of Case~1.1.
	Condition~\eqref{claim:regular-completion-u2} is vacuous since $D$ is simple.
	Applying Claim~\ref{claim:regular-completion} gives the desired contradiction.

	\smallskip
	\noindent\textbf{Case 2.} $G_\ell$ has an almost-overfull set $X$.
	
	Put $\Delta_\ell=\Delta(H_\ell)=\Delta-2\ell$ and
	$Y=V(H_\ell)\setminus X$. Since $X$ is almost-overfull in $G_\ell$, we have
	\[
	\sum_{v\in X}\bigl(\Delta_\ell-d_\ell(v)\bigr)
	+\sum_{u\in V(G_\ell)\setminus X}d_\ell(u)
	\le \Delta_\ell+7\eta^2n.
	\]
	By the construction of the projected graphs,
	$d_\ell(v)=d_{H_\ell}(v)$ for every $v\in V(H_\ell)$. Therefore
	\[
	\begin{aligned}
		&\sum_{v\in X}\bigl(\Delta_\ell-d_\ell(v)\bigr)
		+\sum_{u\in V(G_\ell)\setminus X}d_\ell(u)\\
		&=
		\sum_{v\in X}\bigl(\Delta_\ell-d_{H_\ell}(v)\bigr)
		+\sum_{u\in Y}d_{H_\ell}(u)\\
		&=
		\df_{H_\ell}(X)+2e_{H_\ell}(Y).
	\end{aligned}
	\]
	Consequently, $\df_{H_\ell}(X)+2e_{H_\ell}(Y)
	\le \Delta_\ell+7\eta^2n$, 
	and, in particular, $\df_{H_\ell}(X)\le\Delta_\ell+7\eta^2n$. 
	Since $H_\ell^*$ is obtained from $H_\ell$  by deleting only edges from $E_{H_\ell}(X,Y)$, every vertex of $X$ has the same degree in $H_\ell^*[X]$ as in $H_\ell[X]$. Also $\Delta(H_\ell^*)=\Delta(H_\ell)=\Delta_\ell$. Hence
	$$
	\df_{H^*_\ell}(X)=\df_{H_\ell}(X)\le\Delta_\ell+7\eta^2n. 
	$$

	\begin{CLA}\label{claim:Y-size}
		$|Y|\le |U_\ell|+1\le\eta^2 n+1$.
	\end{CLA}
	
	\begin{proof}
		Since $e_{H^*_\ell}(X,Y)\le\df_{H^*_\ell}(X)\le\Delta_\ell+7\eta^2 n$ and
		$H^*_\ell$ is a robust $(\nu'/2-2\eta^2,16\tau)$-expander,
		$\min\{|X|,|Y|\}<16\tau n$. Since $H^*_\ell$ is simple, the bound
		$\df_{H^*_\ell}(X)\le\Delta_\ell+7\eta^2 n$ implies $H^*_\ell[X]$ contains a
		vertex of degree at least $\Delta_\ell-7\eta^2 n$, so
		$|X|>\Delta_\ell-7\eta^2 n>16\tau n$ and hence $|Y|<16\tau n$.
		
		If $|Y\setminus U_\ell|\ge 2$, then by
		Claim~\ref{claim:Gi-property}\eqref{claim-Gi-property-3},
		\[
		e_{H^*_\ell}(X,Y)\ge 2(\Delta_\ell-(\eta n+\eta^2 n+2)-(|Y|-1))
		>\Delta_\ell+7\eta^2 n,
		\]
		contradicting $e_{H^*_\ell}(X,Y)\le\Delta_\ell+7\eta^2 n$. Hence
		$|Y\setminus U_\ell|\le 1$ and $|Y|\le |U_\ell|+1\le\eta^2 n+1$.
	\end{proof}
	
	Modify $X$ and $Y$ so that the new $Y$-set is independent. If every two
	distinct vertices $y_1,y_2\in Y$ satisfy
	\[
	d_{H^*_\ell}(y_1)+d_{H^*_\ell}(y_2)<\Delta_\ell-2\eta n,
	\]
	set $X_1=X$ and $Y_1=Y$. Otherwise, choose distinct vertices $y_1,y_2\in Y$
	with
	\[
	d_{H^*_\ell}(y_1)+d_{H^*_\ell}(y_2)\ge\Delta_\ell-2\eta n,
	\]
	and set $X_1=X\cup\{y_1,y_2\}$ and $Y_1=Y\setminus\{y_1,y_2\}$.
	
	If $X_1\ne X$, then
	\[
	e_{H^*_\ell}(X,Y)\ge d_{H^*_\ell}(y_1)+d_{H^*_\ell}(y_2)-2(|Y|-1)
	>\Delta_\ell-2.1\eta n,
	\]
	and since $e_{H^*_\ell}(X,Y)\le\Delta_\ell+7\eta^2 n$, any two distinct vertices
	of $Y\setminus\{y_1,y_2\}$ have degree sum at most $6\eta n$. If $X_1=X$, the
	same conclusion follows from the choice. Hence
	\begin{equation}\label{eqn:Y-vertex-degree-sum}
		d_{H^*_\ell}(y_1)+d_{H^*_\ell}(y_2)<\Delta_\ell-2\eta n
		\quad\text{for all distinct $y_1,y_2\in Y_1$.}
	\end{equation}
	By Claim~\ref{claim:property-Gi}\eqref{claim:property-Gi-weak-val}, $Y_1$ is
	independent in $H^*_\ell$.
	
	We show $\df_{H^*_\ell}(X_1)<\Delta_\ell+5\eta n$. If $X_1=X$ this follows from
	$\df_{H^*_\ell}(X_1)\le\Delta_\ell+7\eta^2 n$. Otherwise put $A=\{y_1,y_2\}$;
	since $H^*_\ell[Y]$ is simple and $|Y|\le\eta^2 n+1$,
	\[
	e_{H^*_\ell}(A,X)+e_{H^*_\ell}(A)\ge
	d_{H^*_\ell}(y_1)+d_{H^*_\ell}(y_2)-2(|Y|-1),
	\]
	so
	\[
	\begin{aligned}
		\df_{H^*_\ell}(X_1)
		&=\df_{H^*_\ell}(X)+2\Delta_\ell-2e_{H^*_\ell}(A,X)-2e_{H^*_\ell}(A)\\
		&\le\Delta_\ell+7\eta^2 n+2\Delta_\ell
		-2\bigl(d_{H^*_\ell}(y_1)+d_{H^*_\ell}(y_2)-2(|Y|-1)\bigr)\\
		&<\Delta_\ell+5\eta n,
	\end{aligned}
	\]
	using $d_{H^*_\ell}(y_1)+d_{H^*_\ell}(y_2)\ge\Delta_\ell-2\eta n$ and
	$|Y|\le\eta^2 n+1$.
	
	By repeatedly exchanging a lower-degree vertex $x$ of $X_1$ with a higher-degree
	vertex of $Y_1$, assume further that 
	\begin{equation}\label{eqn:Y-vertex-smaller-degree}
		d_{H^*_\ell}(y)\le d_{H^*_\ell}(x)
		\quad\text{for all $x\in X_1$ and $y\in Y_1$}.
	\end{equation}
	Since each exchanged vertex $x$ has no neighbor in $Y_1$ within $H_\ell^*$ by~\eqref{eqn:Y-vertex-degree-sum} and Claim~\ref{claim:property-Gi}\eqref{claim:property-Gi-weak-val},
	swapping it for a vertex $y\in Y_1$ with $d_{H^*_\ell}(y)>d_{H^*_\ell}(x)$
	decreases  $\df_{H^*_\ell}(X_1)$. Hence the bound 
	continues to hold:
	\begin{equation}\label{eqn:X1-almost-overfull}
		\df_{H^*_\ell}(X_1)<\Delta_\ell+5\eta n.
	\end{equation}

	Let $x_0\in X_1$ with
	$d_{H^*_\ell}(x_0)=\min\{d_{H^*_\ell}(x):x\in X_1\}$. If
	$d_{H^*_\ell}(x_0)\le e_{H^*_\ell}(X_1,Y_1)$, then by
	\eqref{eqn:X1-almost-overfull},
	\begin{equation}\label{eqn:degree-non-x0-vertex-of-X}
		d_{H^*_\ell}(x)\ge\Delta_\ell-5\eta n
		\quad\text{for every $x\in X_1\setminus\{x_0\}$}.
	\end{equation}
	Applying Algorithm~\ref{alg:edge-deletion-wrt-W2} to $H^*_\ell$ with respect to
	$Y_1$, and discarding isolated vertices, we keep the same notation for the
	resulting graph and for the intersections of $X_1$ and $Y_1$ with its vertex
	set. Thus, we have 
	\begin{equation}\label{eqn:Y1-weak-VAL}
		d_{H^*_\ell}(y)+d_{H^*_\ell}(v)\ge\Delta_\ell+2
		\quad\text{for every $yv\in E(H^*_\ell)$ with $y\in Y_1$}.
	\end{equation}
	By Lemma~\ref{lemma:reduce-to-weak-VAL} and Lemma~\ref{lem:stability-expansion},
	this operation preserves the class-$2$  property and maximum degree and $H^*_\ell$ remains a robust
	$(\nu'/2-4\eta^2, 32\tau)$-expander. Since Algorithm~\ref{alg:edge-deletion-wrt-W2}  only deletes 
	edges between $X_1$ and $Y_1$,  \eqref{eqn:X1-almost-overfull} still holds. 
	
	\smallskip
	\noindent\textbf{Subcase 2.1.} $|Y_1|=0$.
	
	If $H^*_\ell$ has at least three vertices of degree less than $\Delta_\ell$, let
	$D$ be obtained from $H^*_\ell$ by adding a new vertex $x$ joined to two vertices
	of largest degree among those with degree less than $\Delta_\ell$. Otherwise
	$H^*_\ell$ has exactly two such vertices (since $H^*_\ell$ has odd order and no
	$\Delta_\ell$-overfull subgraph), and let $D$ be the disjoint union of
	$H^*_\ell$ and an isolated vertex $x$. In both cases $\Delta_0=\Delta_\ell$ and
	$m=|V(D)|$ is even.
	
	Order the vertices of $D$ as $u_1,\ldots,u_m$ by nondecreasing degree. Since
	$Y_1=\emptyset$, \eqref{eqn:X1-almost-overfull} gives
	$\df_D(D-u_1)\le\Delta_\ell+5\eta n$, so
	Condition~\eqref{claim:regular-completion-degree} holds.
	Conditions~\eqref{claim:regular-completion-order},
	\eqref{claim:regular-completion-no-overfull}, and
	\eqref{claim:regular-completion-expander} are immediate. Since adding a new
	vertex creates no parallel edges in $H^*_\ell$,
	Condition~\eqref{claim:regular-completion-simple} holds:
	$D-\{u_1,u_2\}$ is simple and
	$e_D(u_1,u_i),e_D(u_2,u_i)\le 1\le 2\eta^2 m$ for every $i\ge 3$. For
	Condition~\eqref{claim:regular-completion-local}, by the construction of $D$,
	$|N_D(u_1)\setminus\{u_2\}|=0$ implies $D$ has exactly three vertices of degree
	less than $\Delta_0$. Condition~\eqref{claim:regular-completion-u2} is vacuous
	since $D$ is simple. Since $H^*_\ell\subseteq D$ and
	$\chi'(H^*_\ell)>\Delta_\ell$, applying Claim~\ref{claim:regular-completion}
	gives a $\Delta_0$-regular multigraph $\widehat D$ with
	$\chi'(\widehat D)=\Delta_0$, and restricting to $E(D)$ yields a
	$\Delta_0$-coloring of $D$, contradicting $\chi'(D)>\Delta(D)$.
	
	\smallskip
	\noindent\textbf{Subcase 2.2.} $|Y_1|=1$.
	
	Let $Y_1=\{u_1\}$ and $D=H^*_\ell$, with vertices ordered
	$u_1,u_2,\ldots,u_m$ by nondecreasing degree. Then
	$\Delta_0=\Delta_\ell$ and $\chi'(D)>\Delta(D)$ by
	Claim~\ref{claim:property-Gi}\eqref{claim:property-Gi-class2}.
	Conditions~\eqref{claim:regular-completion-order},
	\eqref{claim:regular-completion-no-overfull}, and
	\eqref{claim:regular-completion-expander} follow from
	Claim~\ref{claim:property-Gi}. Condition~\eqref{claim:regular-completion-simple}
	holds since $D=H^*_\ell$ is simple. By~\eqref{eqn:X1-almost-overfull} and
	\eqref{eqn:degree-non-x0-vertex-of-X}, either
	$\df_D(D-u_1)\le\Delta_0+5\eta n$ or all $u_i$ with $i\ge 3$ satisfy
	$d_D^s(u_i)\ge\Delta_0-5\eta n$, giving
	Condition~\eqref{claim:regular-completion-degree}. If
	Condition~\eqref{claim:regular-completion-local} or
	\eqref{claim:regular-completion-u2} needs to be checked,
	\eqref{eqn:Y1-weak-VAL} provides the required degree-sum bound on all edges
	incident with $u_1$, and the two-neighbor cut argument from Case~1.1 applies.
	Applying Claim~\ref{claim:regular-completion} gives a contradiction.
	
	\smallskip
	\noindent\textbf{Subcase 2.3.} $|Y_1|\ge 2$ and
	$d_{H^*_\ell}(x_0)\ge e_{H^*_\ell}(X_1,Y_1)$.
	
	Identify all vertices of $Y_1$ into a single vertex $u_1$, and let $D$ be the
	resulting multigraph. Since $Y_1$ is independent, every edge-coloring of $D$
	gives one of $H^*_\ell$ with the same colors, so $\chi'(D)>\Delta(D)$. Also
	$d_D(u_1)=e_{H^*_\ell}(X_1,Y_1)\le d_{H^*_\ell}(x_0)$ and
	$\Delta_0=\Delta_\ell$.
	
	We show $D$ contains no $\Delta_0$-overfull subgraph. Suppose
	$S\subseteq V(D)$ is odd with $D[S]$ $\Delta_0$-overfull. If $u_1\notin S$,
	then $D[S]$ is a subgraph of $H^*_\ell$, contradicting
	Claim~\ref{claim:property-Gi}\eqref{claim:property-Gi-no-overfull}. So
	$u_1\in S$; let $T=V(D)\setminus S$. Since $D[S]$ is overfull,
	$e_D(S,T)<\Delta_0$.  Since the original $H^*_\ell$ is a robust $(\nu'/2-2\eta^2,16\tau)$-expander by
	Claim~\ref{claim:property-Gi}\eqref{claim:property-Gi-expander}, and
	$|Y_1|\le\eta^2 n+1$, the cleanup operation with respect to $Y_1$ together with
	the identification performed here implies, by Lemma~\ref{lem:stability-expansion},
	that $D^s$ is a robust $(\nu'/2-4\eta^2,32\tau)$-expander. Robust expansion then
	implies $|S|<32\tau n$ or $|T|<32\tau n$. Since
	$D-u_1$ is simple and $e_D(u_1,z)\le |Y_1|\le \eta^2 n+1$ for every
	$z\ne u_1$, the case $|S|<32\tau n$ would give
	\[
	e_D(S)\le \binom{|S|}{2}+|S|(\eta^2 n+1)
	<\Delta_0(|S|-1)/2,
	\]
	by $\eta\ll\tau\ll\alpha$ and $\Delta_0\ge\alpha n/3$, contradicting
	overfullness. Hence $|T|<32\tau n$.
	
	If $|S|\le|V(D)|-3$, then $T$ contains two vertices
	$a,b\in X_1\setminus\{x_0\}$, and by
	\eqref{eqn:degree-non-x0-vertex-of-X},
	\[
	e_D(S,T)\ge d_{H^*_\ell}(a)+d_{H^*_\ell}(b)-2|T|>\Delta_0,
	\]
	contradicting that $D[S]$ is $\Delta_0$-overfull. Hence
	$|S|=|V(D)|-1$. But $\delta(D)=d_D(u_1)\le d_{H^*_\ell}(x_0)$, so
	Lemma~\ref{lem:overfull-delete-one-vertex} gives a contradiction.
	
	Conditions~\eqref{claim:regular-completion-order} and
	\eqref{claim:regular-completion-expander} follow from the construction and
	Lemma~\ref{lem:stability-expansion}. For
	Condition~\eqref{claim:regular-completion-simple}: all multiple edges are
	incident with $u_1$, so $D-\{u_1,u_2\}$ is simple; also
	$e_D(u_1,u_i)=e_{H^*_\ell}(Y_1,u_i)\le |Y_1|\le\eta^2 n+1\le 2\eta^2 m$ and
	$e_D(u_2,u_i)\le 1\le 2\eta^2 m$ for every $i\ge 3$.
	Condition~\eqref{claim:regular-completion-degree} follows from
	\eqref{eqn:X1-almost-overfull}. 
	Conditions~\eqref{claim:regular-completion-local} and
	\eqref{claim:regular-completion-u2} follow from \eqref{eqn:Y1-weak-VAL} and
	the two-neighbor argument above. Applying Claim~\ref{claim:regular-completion}
	gives a contradiction.
	
	\smallskip
	\noindent\textbf{Subcase 2.4.} $|Y_1|\ge 2$,
	$d_{H^*_\ell}(x_0)<e_{H^*_\ell}(X_1,Y_1)$, and there exists
	$x\in X_1\setminus\{x_0\}$ with
	$d_{H^*_\ell}(x)\le\Delta_\ell-d_{H^*_\ell}(x_0)$.
	
	Choose $x_1\in X_1\setminus\{x_0\}$ with minimum degree. By
	\eqref{eqn:degree-non-x0-vertex-of-X},
	$d_{H^*_\ell}(x_1)\ge\Delta_\ell-5\eta n$, so $d_{H^*_\ell}(x_0)\le 5\eta n$.
	Together with \eqref{eqn:Y-vertex-smaller-degree}, every vertex of
	$Y_1\cup\{x_0\}$ has degree less than $\Delta_\ell/3$, so
	$Y_1\cup\{x_0\}$ is independent by
	Claim~\ref{claim:property-Gi}\eqref{claim:property-Gi-weak-val}. Apply
	Algorithm~\ref{alg:edge-deletion-wrt-W2} to $H^*_\ell$ with respect to
	$Y_1\cup\{x_0\}$, discard isolated vertices, and keep the same notation.
	
	Partition $Y_1\cup\{x_0\}$ into two nonempty parts $A$ and $B$ with $x_0\in B$
	so that the degree sums differ by at most $d_{H^*_\ell}(x_0)$. Such a partition
	exists by minimizing the difference: if it exceeded $d_{H^*_\ell}(x_0)$, moving
	one vertex from the larger side would reduce it. Identify $A$ and $B$ into
	vertices $u_1$ and $u_2$ with $d_D(u_1)\le d_D(u_2)$, and let $D$ be the
	resulting multigraph. Since $Y_1\cup\{x_0\}$ is independent, every edge-coloring
	of $D$ gives one of $H^*_\ell$, so $\chi'(D)>\Delta(D)$ and
	$\Delta_0=\Delta_\ell$.
	
	Condition~\eqref{claim:regular-completion-simple}: all multiple edges are
	incident with $u_1$ or $u_2$, so $D-\{u_1,u_2\}$ is simple; and since
	$|A|,|B|\le |Y_1|+1\le\eta^2 n+2\le 2\eta^2 m$,
	$e_D(u_1,u_i),e_D(u_2,u_i)\le 2\eta^2 m$ for every $i\ge 3$.
	
	The weak VAL cleanup gives Condition~\eqref{claim:regular-completion-u2}:
	$u_1u_2\notin E(D)$, and for each $i\in[2]$ and every
	$w\in V(D)\setminus\{u_1,u_2\}$ with $u_iw\in E(D)$,
	\[
	e_D(u_i,w)\le e_D(u_i,V(D)\setminus\{w\})
	\quad\text{and}\quad
	d_D(u_i)+d_D(w)\ge\Delta_0+2.
	\]
	Since $e_{H^*_\ell}(X_1,Y_1)>d_{H^*_\ell}(x_0)$ and
	\eqref{eqn:X1-almost-overfull} gives
	$e_{H^*_\ell}(X_1,Y_1)\le d_{H^*_\ell}(x_0)+5\eta n$, the total degree of the two
	identified vertices is at most $15\eta n$, and the partition balancing gives
	$d_D(u_2)\le 10\eta n$. Vertices outside $\{u_1,u_2\}$ have simple degree at
	least $\Delta_0-5\eta n$ by \eqref{eqn:degree-non-x0-vertex-of-X}, so
	Condition~\eqref{claim:regular-completion-degree} holds.
	
	We show $D$ contains no $\Delta_0$-overfull subgraph. Suppose
	$S\subseteq V(D)$ is odd with $D[S]$ $\Delta_0$-overfull. The robust expansion
	and cut argument from Subcase~2.3 rules out $|S|\le |V(D)|-3$ unless
	$V(D)\setminus S=\{u_1,u_2,x_1\}$, but then $D[S]$ is a subgraph of $H^*_\ell$,
	contradicting that $H^*_\ell$ contains no $\Delta_0$-overfull subgraph. Hence
	$|S|=|V(D)|-1$. By Lemma~\ref{lem:overfull-delete-one-vertex}, the missing
	vertex is $u_1$ or $u_2$. If $S=V(D)\setminus\{u_1\}$, using
	$d_D(u_2)-d_D(u_1)\le d_{H^*_\ell}(x_0)$ and
	$d_D(x_1)=d_{H^*_\ell}(x_1)\le\Delta_\ell-d_{H^*_\ell}(x_0)$,
	\[
	\df_D(D-u_1)\ge d_D(u_1)+2\Delta_0-d_D(u_2)-d_D(x_1)\ge\Delta_0,
	\]
	contradicting that $D[S]$ is $\Delta_0$-overfull; the case
	$S=V(D)\setminus\{u_2\}$ is identical. Hence
	Condition~\eqref{claim:regular-completion-no-overfull} holds, and
	Claim~\ref{claim:regular-completion} gives a contradiction.
	
	\smallskip
	\noindent\textbf{Subcase 2.5.} $|Y_1|\ge 2$,
	$d_{H^*_\ell}(x_0)<e_{H^*_\ell}(X_1,Y_1)$, and
	$d_{H^*_\ell}(x)>\Delta_\ell-d_{H^*_\ell}(x_0)$ for every
	$x\in X_1\setminus\{x_0\}$.
	
	Put $s=e_{H^*_\ell}(X_1,Y_1)-d_{H^*_\ell}(x_0)>0$. Since
	$\df_{H^*_\ell}(X_1)<\Delta_\ell+5\eta n$, for every
	$x\in X_1\setminus\{x_0\}$,
	\begin{equation}\label{eqn:x-degree-subcase2.5}
		d_{H^*_\ell}(x)>\Delta_\ell-5\eta n+s.
	\end{equation}
	
	\begin{Ope}[Matching-Removal Operation]\label{ope:stage1-matching-removal-case2}
		Put $Q_0^*=H^*_\ell$. For $i\ge 1$, do the following.
		\begin{enumerate}[Step 1:]
			\item Apply Algorithm~\ref{alg:edge-deletion-wrt-W2} to $Q_{i-1}^*$ with
			$W=Y_1$. Let $R_i$ be the set of deleted edges, and let $Q_{i-1}$ be
			obtained from $Q_{i-1}^*-R_i$ by discarding isolated vertices. Keep the
			same notation for $X_1$ and $Y_1$ after intersecting with $V(Q_{i-1})$.
			
			\item Choose $x_{i-1}^0\in X_1$ with
			$d_{Q_{i-1}}(x_{i-1}^0)=\min\{d_{Q_{i-1}}(x):x\in X_1\}$. Stop and let
			$r=i-1$ if $|Y_1|\le 1$, or
			$e_{Q_{i-1}}(X_1,Y_1)\le d_{Q_{i-1}}(x_{i-1}^0)$, or there exists
			$x\in X_1\setminus\{x_{i-1}^0\}$ with
			$d_{Q_{i-1}}(x)\le\Delta(Q_{i-1})-d_{Q_{i-1}}(x_{i-1}^0)$.
			
			\item Otherwise, choose 
			$M_i=\{y_1x_1,y_2x_2\}\subseteq E_{Q_{i-1}}(Y_1,X_1)$ consisting of two independent edges, where
			$y_1,y_2\in Y_1$ have largest and second-largest degree in
			$Q_{i-1}[Y_1,X_1]$, and $x_1,x_2\in X_1\setminus\{x_{i-1}^0\}$. Such a
			matching exists by Lemma~\ref{lem:matching-in-critical-graph}, since
			$Q_{i-1}$ satisfies weak VAL with respect to $Y_1$.
			
			\item Let $N_i$ be a perfect matching of
			$Q_{i-1}[X_1]-\{x_{i-1}^0,x_1,x_2\}$. Such a matching exists because
			$Q_{i-1}[X_1]-\{x_{i-1}^0,x_1,x_2\}$ is a robust
			$(\nu'/2-6\eta,32\tau)$-expander by Lemma~\ref{lem:stability-expansion} ($s<5\eta n$ and $H_\ell^*$ is a robust $(\nu'/2-2\eta^2 n, 16\tau)$-expander),
			hence has a Hamilton cycle by Lemma~\ref{lem:hamiltonicity-of-expander};
			its order is even.
			
			\item Let $Q_i^*$ be obtained from $Q_{i-1}-(M_i\cup N_i)$ by discarding
			isolated vertices, and continue.
		\end{enumerate}
	\end{Ope}
	
	\begin{CLA}\label{claim:Qi-property}
		The operation stops with $r\le s$, and the following hold.
		\begin{enumerate}[(1)]
			\item\label{claim:Qi-stop}
			At least one of: $|Y_1|\le 1$,
			$e_{Q_r}(X_1,Y_1)\le d_{Q_r}(x_r^0)$, or there exists
			$x\in X_1\setminus\{x_r^0\}$ with
			$d_{Q_r}(x)\le\Delta(Q_r)-d_{Q_r}(x_r^0)$ holds.
			
			\item\label{claim:Qi-class2}
			For every $i\in[0,r]$, $\Delta(Q_i)=\Delta(Q_i^*)=\Delta_\ell-i$ and
			$Q_i$ is class~$2$.
			
			\item\label{claim:Qi-deficiency}
			If $\df_{H^*_\ell}(X_1)=\Delta_\ell+t$, then
			$\df_{Q_r}(X_1)=\Delta(Q_r)+t-2r<\Delta(Q_r)+5\eta n$.
			
			\item\label{claim:Qi-no-overfull}
			For every $i\in[0,r]$, $Q_i$ contains no $\Delta(Q_i)$-overfull subgraph.
			
			\item\label{claim:Qi-expander}
			$Q_r$ is a robust $(\nu'/2-6\eta,32\tau)$-expander.
		\end{enumerate}
	\end{CLA}
	
	\begin{proof}
		Let $s_i=e_{Q_i}(X_1,Y_1)-d_{Q_i}(x_i^0)$ whenever $Q_i$ is not a stopping
		graph. In a completed round $i+1$, $e_{Q_i}(X_1,Y_1)$ decreases by $2$; before
		the cleanup producing $Q_{i+1}$, every vertex of
		$X_1\setminus\{x_i^0\}$ loses one incident edge while $x_i^0$ loses none; and
		during cleanup every deleted edge from $X_1$ lies in $E(X_1,Y_1)$. Hence
		$s_{i+1}\le s_i-1$, and since $s_0\le s$ we have $r\le s$.
		
		Conclusion~\eqref{claim:Qi-stop} follows directly from Step~2.
		
		For~\eqref{claim:Qi-class2}: by Lemma~\ref{lemma:reduce-to-weak-VAL},
		Algorithm~\ref{alg:edge-deletion-wrt-W2} preserves the class-$2$ property and
		the current maximum degree, so $Q_0$ is class~$2$ and
		$\Delta(Q_0)=\Delta_\ell$. Suppose round $i$ is completed. Then
		$M_i\cup N_i$ is a matching of $Q_{i-1}$ saturating every maximum-degree
		vertex, so $\Delta(Q_i^*)=\Delta_\ell-i$. If
		$\chi'(Q_i^*)\le\Delta(Q_i^*)$, coloring $M_i\cup N_i$ with one new color gives
		a $\Delta(Q_{i-1})$-coloring of $Q_{i-1}$, contradicting class~$2$. Hence
		$Q_i^*$ is class~$2$, and the cleanup preserves this and the maximum degree.
		
		For~\eqref{claim:Qi-deficiency}: write
		$\df_{H^*_\ell}(X_1)=\Delta_\ell+t$ with $t<5\eta n$. 
		In each completed round
		the maximum degree decreases by $1$ and $N_i$ removes one edge from each vertex
		of $X_1\setminus\{x_{i-1}^0,x_1,x_2\}$, giving
		\[
		\df_{Q_r}(X_1)=\Delta(Q_0)+t-3r=\Delta(Q_r)+t-2r<\Delta(Q_r)+5\eta n.
		\]
		
		For~\eqref{claim:Qi-no-overfull}: the case $i=0$ holds since $Q_0$ is obtained
		from $H^*_\ell$ by edge deletion with the same maximum degree $\Delta_\ell$, and
		$H^*_\ell$ has no $\Delta_\ell$-overfull subgraph. Let $i\ge 1$ be the smallest
		index for which $Q_i$ contains a $\Delta(Q_i)$-overfull subgraph. Let
		$S\subseteq V(Q_i)$ be odd with
		$e_{Q_i}(S)>\Delta(Q_i)(|S|-1)/2$, and set $T=V(Q_i)\setminus S$. If
		$T\ne\emptyset$, overfullness gives $e_{Q_i}(S,T)<\Delta(Q_i)$ and robust
		expansion implies $|T|<32\tau n$; if $T=\emptyset$ this is trivial. Thus 
		$|T|<32\tau n$.

		Put $F_i=M_i\cup N_i$. Since $Q_i$ is obtained from
		$Q_i^*=Q_{i-1}-F_i$ by cleanup while preserving the maximum degree, we have
		$\df_{Q_i}(S)\ge\df_{Q_i^*}(S)$. Let $\theta_i(S)$ be the number of vertices
		of $S$ not saturated by $F_i\cap E(Q_i[S])$. Then
		\begin{equation}\label{eqn:deficiency-change-Fi}
			\df_{Q_i^*}(S)=\df_{Q_{i-1}}(S)-\theta_i(S).
		\end{equation}
		Indeed, passing from $Q_{i-1}$ to $Q_i^*$ decreases the maximum degree by one
		and removes exactly $|F_i\cap E(Q_i[S])|$ 
		edges with both endpoints in $S$. Since $F_i$ is a matching,
		\[
		2|F_i\cap E(Q_i[S])|=|S|-\theta_i(S),
		\]
		and hence
		\[
		\begin{aligned}
			\df_{Q_i^*}(S)
			&=\df_{Q_{i-1}}(S)-|S|
			+2|F_i\cap E(Q_i[S])|\\
			&=\df_{Q_{i-1}}(S)-\theta_i(S).
		\end{aligned}
		\]
		In particular, since $|S|$ is odd, $\theta_i(S)$ is odd.

		If $|T\cap X_1|\ge 3$, then
		$T\cap(X_1\setminus\{x_{i-1}^0\})$ contains two vertices $x_1,x_2$, and by
		\eqref{eqn:x-degree-subcase2.5},
		\[
		d_{Q_i}(x_1),d_{Q_i}(x_2)\ge\Delta_\ell-5\eta n+s-i =\Delta(Q_i)-5\eta n+s.
		\]
		With $|T| <32\tau n$ and $i\le r\le s$,
		\[
		e_{Q_i}(S,T)\ge d_{Q_i}(x_1)+d_{Q_i}(x_2)-2|T|>\Delta(Q_i),
		\]
		a contradiction. Hence  $|T\cap X_1|\le 2$.

		Suppose $|T\cap X_1|=2$. The case
		$x_{i-1}^0\notin T\cap X_1$ is handled by the above estimate, so assume
		$T\cap X_1=\{x_{i-1}^0,x\}$ for some
		$x\in X_1\setminus\{x_{i-1}^0\}$.
		
		If $e_{Q_i}(T\cap X_1,Y_1\setminus S)=0$, then since round $i$ was completed,
		$d_{Q_{i-1}}(x)>\Delta(Q_{i-1})-d_{Q_{i-1}}(x_{i-1}^0)$. Since $x$ loses one
		edge in $F_i$ while $x_{i-1}^0$ loses none and
		$\Delta(Q_i)=\Delta(Q_{i-1})-1$, we get
		$d_{Q_i}(x)>\Delta(Q_i)-d_{Q_i}(x_{i-1}^0)$, so
		\[
		\df_{Q_i}(S)\ge d_{Q_i}(x_{i-1}^0)
		+(\Delta(Q_i)-d_{Q_i}(x_{i-1}^0)+1)-2
		=\Delta(Q_i)-1.
		\]
		Since $|S|$ is odd, Lemma~\ref{lem:def-relation-with-Delta} gives
		$\df_{Q_i}(S)\ge\Delta(Q_i)$, a contradiction.

	Thus $e_{Q_i}(T\cap X_1,Y_1\setminus S)>0$. We first claim that 
	$e_{Q_i}(x^0_{i-1},Y_1\setminus S)=0$. 
	Suppose otherwise, and let $y\in Y_1\setminus S$ be such that
	$yx^0_{i-1}\in E(Q_i)$. Since $y\in Y_1$ and $Q_i$ satisfies weak VAL
	with respect to $Y_1$, we have $d_{Q_i}(y)+d_{Q_i}(x^0_{i-1})\ge \Delta(Q_i)+2$. 
	Moreover, $d_{Q_0}(y)\le d_{Q_0}(x^0_{i-1})$ for every $y\in Y_1$, and
	the difference between these two degrees can increase by at most $i$
	during the first $i$ rounds. Hence $d_{Q_i}(y)\le d_{Q_i}(x^0_{i-1})+i$. 
	It follows that $2d_{Q_i}(x^0_{i-1})+i\ge \Delta(Q_i)+2$, 
	and therefore
	\[
	d_{Q_i}(x^0_{i-1})
	\ge \frac{\Delta(Q_i)+2-i}{2}
	\ge \frac{\Delta(Q_i)}{2}-3\eta n.
	\]
	Since $d_{Q_i}(x)\ge d_{Q_i}(x^0_{i-1})$, we obtain
	\[
	\begin{aligned}
		e_{Q_i}(S,T)
		&\ge d_{Q_i}(x^0_{i-1})+d_{Q_i}(y)+d_{Q_i}(x)-3|T|\\
		&\ge d_{Q_i}(x^0_{i-1})+d_{Q_i}(y)
		+d_{Q_i}(x^0_{i-1})-3|T|\\
		&\ge \Delta(Q_i)+2+\frac{\Delta(Q_i)}2-3\eta n-3|T|\\
		&>\Delta(Q_i),
	\end{aligned}
	\]
	a contradiction. Thus $e_{Q_i}(x^0_{i-1},Y_1\setminus S)=0$. 

	Let $a=e_{Q_i}(x,Y_1\setminus S)$. 
	Since $Q_{i-1}$ is not a stopping graph, we have $d_{Q_i}(x)>\Delta(Q_i)-d_{Q_i}(x^0_{i-1})$. 
	For every edge $xy$ with $y\in Y_1\setminus S$, weak VAL with respect to
	$Y_1$ gives $d_{Q_i}(x)+d_{Q_i}(y)\ge \Delta(Q_i)+2$. 
	Summing this inequality over the $a$ edges from $x$ to
	$Y_1\setminus S$,   yields
	\[
	\sum_{y\in N_{Q_i}(x)\cap(Y_1\setminus S)}
	 d_{Q_i}(y)
	\ge a\bigl(\Delta(Q_i)+2-d_{Q_i}(x)\bigr).
	\]
	Since $Y_1$ is independent and
	$e_{Q_i}(x^0_{i-1},Y_1\setminus S)=0$, after subtracting the $a$ edges
	joining these vertices to $x$, the vertices in
	$N_{Q_i}(x)\cap(Y_1\setminus S)$ send at least 
	$a\bigl(\Delta(Q_i)+2-d_{Q_i}(x)\bigr)-a$ 
	edges to $S$. Also, $x$ sends at least $\Delta(Q_i)-d_{Q_i}(x^0_{i-1})-a$ 
	edges to $S$, while the edges incident with $x^0_{i-1}$ contribute at
	least $d_{Q_i}(x^0_{i-1})-1$ further edges to the cut. Therefore,
	\[
	\begin{aligned}
		e_{Q_i}(S,T)
		&\ge
		a\bigl(\Delta(Q_i)+2-d_{Q_i}(x)\bigr)-a 
	+\Delta(Q_i)-d_{Q_i}(x^0_{i-1})-a
		+d_{Q_i}(x^0_{i-1})-1\\
		&=\Delta(Q_i)
		+a\bigl(\Delta(Q_i)-d_{Q_i}(x)\bigr)-1\\
		&\ge \Delta(Q_i)-1.
	\end{aligned}
	\]
	Since $|S|$ is odd, Lemma~\ref{lem:def-relation-with-Delta} gives
	$\df_{Q_i}(S)\ge\Delta(Q_i)$, a contradiction.

		Since $|T\cap X_1| \le 2$, it remains to handle
		$|T\cap X_1|\le 1$. If $|S\cap Y_1|\ge 3$, then
		\[
		\sum_{y\in S\cap Y_1}d_{Q_i}(y)\le e_{Q_i}(X_1,Y_1) <\Delta(Q_i)+5\eta n,
		\]
		so $\df_{Q_i}(S) >3\Delta(Q_i)-(\Delta(Q_i)+5\eta n)>\Delta(Q_i)$, a contradiction. Hence
		$|S\cap Y_1|\le 2$, and since $|S|$ and $|X_1|$ are odd, the only
		possibilities are
		\[
		(|T\cap X_1|,|S\cap Y_1|)=(0,2)\quad\text{or}\quad(1,1).
		\]

		Suppose $(|T\cap X_1|,|S\cap Y_1|)=(0,2)$; write
		$S\cap Y_1=\{y_1,y_2\}$ and let $h=|\{y_1,y_2\}\cap V(M_i)|$.
		Since $T\cap X_1=\emptyset$, we have $X_1\subseteq S$. By the construction
		of $F_i$, the vertices of $S$ not saturated by $F_i\cap E(Q_i[S])$ are
		precisely $x_{i-1}^0$,  the vertices of $\{y_1,y_2\}\setminus V(M_i)$,
		and $2-h$ vertices of $X_1\setminus \{x_{i-1}^0\}$ covered by $M_i$
		through crossing edges. 
	Therefore $\theta_i(S)=1+|\{y_1,y_2\}\setminus V(M_i)|+2-h=5-2h$.

		Suppose first that $h=0$. Then $\theta_i(S)=5$. Let $y_1'$ and $y_2'$ be
		the two vertices of $Y_1$ saturated by $M_i$. By the choice of $M_i$,
		after relabeling $y_1'$ and $y_2'$ if necessary, we have
		$d_{Q_{i-1}}(y_1')\ge d_{Q_{i-1}}(y_1)$ and
		$d_{Q_{i-1}}(y_2')\ge d_{Q_{i-1}}(y_2)$. It follows that
		\[
		\begin{aligned}
			\df_{Q_i}(S)
			&\ge
			2\Delta(Q_{i-1})
			-d_{Q_{i-1}}(y_1)-d_{Q_{i-1}}(y_2)
			+d_{Q_{i-1}}(y_1)+d_{Q_{i-1}}(y_2)-2\\
			&=2\Delta(Q_{i-1})-2>\Delta(Q_i),
		\end{aligned}
		\]
		a contradiction.
		
		Suppose then  that $h=1$. Then $\theta_i(S)=3$. Let $y_1'$ and $y_2'$ be
		the two vertices of $Y_1$ saturated by $M_i$. Suppose, 
		without loss of generality, that $y_2'=y_2$. 
		By the choice of $M_i$,
		we have
		$d_{Q_{i-1}}(y_1')\ge d_{Q_{i-1}}(y_1)$.  It then follows that
		\[
		\begin{aligned}
			\df_{Q_i}(S)
			&\ge
			\Delta(Q_{i-1})
			-d_{Q_{i-1}}(y_1)
			+d_{Q_{i-1}}(y'_1)-1\\
			& \ge \Delta(Q_{i-1})-1.
		\end{aligned}
		\]
	Since $|S|$ is odd, Lemma~\ref{lem:def-relation-with-Delta} gives
	$\df_{Q_i}(S)\ge\Delta(Q_i)$, a contradiction.

		Suppose lastly that $h=2$. Then $\theta_i(S)=1$, and
		\eqref{eqn:deficiency-change-Fi} gives
		\[
		\df_{Q_i}(S)
		\ge \df_{Q_i^*}(S)
		=\df_{Q_{i-1}}(S)-1
		\ge \Delta(Q_{i-1})-1
		=\Delta(Q_i),
		\]
		where the second inequality follows from the fact that $Q_{i-1}$ has no
		$\Delta(Q_{i-1})$-overfull subgraph.  This is again a contradiction.

	Suppose next that $(|T\cap X_1|,|S\cap Y_1|)=(1,1)$; write
	$S\cap Y_1=\{y\}$ and $T\cap X_1=\{x\}$.
	
	We first consider the case where $e_{Q_i}(x,Y_1\setminus S)>0$.
	Let
	$a=e_{Q_i}(x,Y_1\setminus S)$. For every edge $xy'$ with
	$y'\in Y_1\setminus S$, weak VAL with respect to $Y_1$ gives
	$d_{Q_i}(y')+d_{Q_i}(x)\ge \Delta(Q_i)+2$. Summing over the $a$ edges
	from $x$ to $Y_1\setminus S$  gives
	\[
	\sum_{y'\in N_{Q_i}(x)\cap(Y_1\setminus S)}
d_{Q_i}(y')
	\ge a\bigl(\Delta(Q_i)+2-d_{Q_i}(x)\bigr).
	\]
	Since $Y_1$ is independent and $T\cap X_1=\{x\}$, after subtracting the
	$a$ edges joining these vertices to $x$, the vertices in
	$N_{Q_i}(x)\cap(Y_1\setminus S)$ send at least
	$a(\Delta(Q_i)+2-d_{Q_i}(x))-a$ edges to $S$. Also, the vertex $x$ sends
	$d_{Q_i}(x)-a$ edges to $S$. Therefore
	\[
	\begin{aligned}
		e_{Q_i}(S,T)
		&\ge a\bigl(\Delta(Q_i)+2-d_{Q_i}(x)\bigr)-a
		+d_{Q_i}(x)-a\\
		&=\Delta(Q_i)+(a-1)\bigl(\Delta(Q_i)-d_{Q_i}(x)\bigr) 
\ge \Delta(Q_i),
	\end{aligned}
	\]
	a contradiction.
	
	We may therefore assume that $e_{Q_i}(x,Y_1\setminus S)=0$. We
	consider two subcases according to whether $x^0_{i-1}$ belongs to $S$.
	
	First suppose that $x^0_{i-1}\notin S$. Then $x=x^0_{i-1}$. Since
	$x^0_{i-1}$ is not saturated by $F_i$, no edge of $M_i$ is incident with
	$x$. Hence both edges of $M_i$ have their $X_1$-endpoints in $S\cap X_1$.
	Let $r=e_{Q_i[M_i]}(y,S\cap X_1)$. Then the vertices of $S$ not saturated by
	$F_i\cap E(Q_i[S])$ are exactly the vertex $y$, if $r=0$, together with
	the $2-r$ vertices of $S\cap X_1$ covered by $M_i$ through crossing edges.
	Thus $\theta_i(S)=\mathbf 1_{\{r=0\}}+2-r$.
	
	Since $|S|$ is odd, $\theta_i(S)$ is odd. If $r=1$, then
	$\theta_i(S)=1$, and hence
	\[
	\df_{Q_i}(S)
	\ge \df_{Q_i^*}(S)
	=\df_{Q_{i-1}}(S)-1
	\ge \Delta(Q_{i-1})-1
	=\Delta(Q_i),
	\]
	a contradiction. Thus $r=0$, and so $\theta_i(S)=3$. In this case
	$y\notin V(M_i)$. Let $y_1'$ and $y_2'$ be the two vertices of $Y_1$
	saturated by $M_i$. By the choice of $M_i$, we have
	$d_{Q_{i-1}}(y_1')\ge d_{Q_{i-1}}(y)$ and
	$d_{Q_{i-1}}(y_2')\ge d_{Q_{i-1}}(y)$. Then
	\[
	\begin{aligned}
		\df_{Q_i}(S)
		&\ge
		\Delta(Q_{i-1})-d_{Q_{i-1}}(y)
		+d_{Q_{i-1}}(y_1')+d_{Q_{i-1}}(y_2')-2\\
		&\ge \Delta(Q_{i-1})+d_{Q_{i-1}}(y_2')-2 
\ge \Delta(Q_i),
	\end{aligned}
	\]
	again a contradiction, since $y_2'$ is saturated by $M_i$ and hence
	$d_{Q_{i-1}}(y_2')\ge 1$.
	
	It remains to consider the case $x^0_{i-1}\in S$. Then
	$x\ne x^0_{i-1}$. Let $S'=(S\setminus\{x^0_{i-1}\})\cup\{x\}$. Then
	$|S'|=|S|$, so $S'$ is odd. Since
	$e_{Q_i}(x,Y_1\setminus S)=0$, all edges of
	$Q_i$ incident with $x$ lie inside $S'\cup\{x^0_{i-1}\}$. Hence
	$d_{Q_i}(x,S')=d_{Q_i}(x)-e_{Q_i}(x,x^0_{i-1})$. On the other hand,
	$d_{Q_i}(x^0_{i-1},S\setminus\{x^0_{i-1}\})\le d_{Q_i}(x^0_{i-1})-e_{Q_i}(x,x^0_{i-1})$.
	Using $d_{Q_i}(x)\ge d_{Q_i}(x^0_{i-1})$, we get
	\[
	\begin{aligned}
		e_{Q_i}(S')
		&=e_{Q_i}(S)-d_{Q_i}(x^0_{i-1},S\setminus\{x^0_{i-1}\})
		+d_{Q_i}(x,S\setminus\{x^0_{i-1}\})\\
		&\ge e_{Q_i}(S)-d_{Q_i}(x^0_{i-1})+e_{Q_i}(x,x^0_{i-1})+d_{Q_i}(x)-e_{Q_i}(x,x^0_{i-1}) \ge e_{Q_i}(S).
	\end{aligned}
	\]
	Thus $Q_i[S']$ is also $\Delta(Q_i)$-overfull. But for $S'$, we have
	$S'\cap Y_1=\{y\}$, $x^0_{i-1}\notin S'$, and
	$|T'\cap X_1|=1$, where $T'=V(Q_i)\setminus S'$. Hence the already treated
	cases apply to $S'$, giving the final contradiction.

		Conclusion~\eqref{claim:Qi-expander} follows from
		Claim~\ref{claim:property-Gi}\eqref{claim:property-Gi-expander},
		Lemma~\ref{lem:stability-expansion}, and $r\le s<5\eta n$.
	\end{proof}
	
	By Conclusion~\eqref{claim:Qi-stop}, one of the three stopping alternatives
	holds. By Conclusions~\eqref{claim:Qi-class2}--\eqref{claim:Qi-expander},
	$Q_r$ satisfies the same structural hypotheses as $H^*_\ell$ with weaker
	expansion parameters (but $Q_\ell$ is still a robust $(\nu'/3, 32\tau)$-expander). With $Q_r$, $\Delta(Q_r)$, and $x_r^0$ in place of
	$H^*_\ell$, $\Delta_\ell$, and $x_0$, construct $D$ as in
	Subcases~2.1--2.4. The verification that $D$ satisfies
	Conditions~\eqref{claim:regular-completion-order}--\eqref{claim:regular-completion-expander}
	is identical to the above with $Q_r$ replacing $H^*_\ell$. Since an edge
	$\Delta(D)$-coloring of $D$ induces a $\Delta(Q_r)$-coloring of $Q_r$,
	Conclusion~\eqref{claim:Qi-class2} gives $\chi'(D)>\Delta(D)$. Applying
	Claim~\ref{claim:regular-completion} gives a $\Delta(D)$-regular multigraph
	$\widehat D$ with $\chi'(\widehat D)=\Delta(D)$, and restricting to $E(D)$
	yields a $\Delta(D)$-coloring of $D$, contradicting $\chi'(D)>\Delta(D)$.
	This contradiction completes Stage~II.
\end{proof}

	\section{Proof of Theorem~\ref{thm:robust-expander2}}\label{Section:5}

\subsection{The Regularity Lemma and A Balanced Bipartition}

The second case for the proof of Theorem~\ref{thm:robust-expander2}
 proceeds via an
alternating-path argument that decomposes the edges of $G$ into 1-factors by
exploiting a balanced bipartition $(A,B)$ of $V(G)$. For this to work, the
bipartite graph $G[A,B]$ must inherit a bipartite form of robust expansion from
$G$: this is what guarantees short alternating paths between any prescribed pair
of vertices. The key intermediary is Szemer\'edi's regularity lemma, which
decomposes $G$ into a bounded collection of cluster pairs each exhibiting
uniform edge-distribution. The reduced graph on these clusters inherits the
robust expansion of $G$ (Lemma~\ref{lem:reduced-robust}), and a random balanced
partition of the clusters, certified by a Chernoff bound, transfers this
expansion to $G[A,B]$ (Lemma~\ref{lem:bipartite-robust-slicing}).
Thus we introduce the regularity lemma in this subsection
and apply it  to obtain a desired bipartition $(A,B)$
of $G$.

Let $G[X,Y]$ be a bipartite graph. The \emph{density} of $G$ is the ratio
$d(X,Y):=\frac{e_G(X,Y)}{|X||Y|}$. Given $\ve>0$, we say $G$ is
\emph{$\ve$-regular} if for every $X'\subseteq X$ and $Y'\subseteq Y$ with
$|X'| \ge \ve|X|$ and $|Y'| \ge \ve|Y|$, $|d(X',Y')-d(X,Y)|<\ve$ holds.
We also say that the pair $(X,Y)$ is
$\ve$-regular. We call $G$ \emph{$(\ve, d)$-regular} if $G$ is $\ve$-regular
with density $d\pm \ve$.

\begin{LEM}[Degree form of the Regularity Lemma~{\cite{Szemeredi-regular-partitions}}]\label{lem:regularity-lemma}
	For every $\varepsilon>0$ and every positive integer $M'$ there exist integers
	$M$ and $n_0$ such that the following holds. Suppose that $G$ is a graph on
	$n\ge n_0$ vertices and that $d\in[0,1]$. Then there exists a partition of
	$V(G)$ into
	\[
	V_0,V_1,\dots,V_k
	\]
	and a spanning subgraph $G'\subseteq G$ such that the following conditions
	hold:
	\begin{enumerate}[(1)]
		\item $M'\le k\le M$;
		\item $|V_0|\le \varepsilon n$;
		\item $|V_1|=\cdots=|V_k|=:m$;
		\item $d_{G'}(x)\ge d_G(x)-(d+\varepsilon)n$ for all $x\in V(G)$;
		\item for all $i\in [k]$, the graph $G'[V_i]$ is empty;
		\item for all distinct $i,j\in [k]$, the graph $G'[V_i,V_j]$ is
		$\varepsilon$-regular and either has density $0$ or density at least $d$.
	\end{enumerate}
\end{LEM}

We refer to $V_0$ as the \emph{exceptional set} and to $V_1,\dots,V_k$ as
\emph{clusters}. The partition $V_0,V_1,\dots,V_k$ is also called a
\emph{regularity partition}, and $G'$ is called the \emph{pure graph}.

Given a graph $G$ on $n$ vertices, we form the \emph{reduced graph} $R$ of $G$
with parameters $\varepsilon,d$, and $M'$ by applying the regularity lemma with
these parameters to obtain $V_0,\dots,V_k$. Then $R$ is the graph whose
vertices are the clusters $V_1,\dots,V_k$, and whose edges are those pairs
$V_i, V_j$ of clusters for which $G'[V_i,V_j]$ is nonempty.

The version of Lemma~\ref{lem:reduced-robust} for robust outexpanders (the directed analogue of robust expanders)
 was proved in Lemma~14 of~\cite{MR2644240}. A stronger form of the
statement is Lemma~3.5 in~\cite{MR3090152}, and Lemma~6.2 from~\cite{KO2014}
stated that the stronger version holds also for robust expanders. Here, the
weak version suffices for our application, and we modify the short proof in
Lemma~14 of~\cite{MR2644240} for graphs for completeness.

\begin{LEM}\label{lem:reduced-robust}
	Let $M',n_0$ be positive integers and let $\varepsilon,d, \nu,\tau, \alpha$ be
	positive constants such that
	$\frac1{n_0}\ll \varepsilon \ll d \ll \nu,\tau,\alpha < 1$ and
	$M' \ll n_0$. Let $G$ be a graph on $n\ge n_0$ vertices. 
	If $G$ is robust $(\nu,\tau)$-expander, then $R$ is robust $(\nu/2,2\tau)$-expander.
	If additionally $\delta(G)\ge \alpha n$, then $\delta(R)\ge \alpha |R|/2$.
\end{LEM}

\begin{proof}
	Let $G'$ denote the pure graph, $k=|V(R)|$, let $V_1,\dots,V_k$ be the
	clusters of $G$, and let $V_0$ be the exceptional set. Let
	$m= |V_1| =\ldots =|V_k|$.

	We first prove the minimum-degree assertion, assuming in addition that
	$\delta(G)\ge \alpha n$.
	By Lemma~\ref{lem:regularity-lemma}(4), we have
	$\delta(R)\ge \frac{\delta(G')-|V_0|}{m}$. Moreover, every vertex loses at
	most $(d+\varepsilon)n$ neighbors when passing from $G$ to $G'$, and
	$|V_0|\le \varepsilon n$. Therefore
	$\delta(R)\ge \frac{\delta(G)-(d+2\varepsilon)n}{m}$. Since
	$\delta(G)\ge \alpha n$, $d,\varepsilon\ll \alpha$, and $n=km+|V_0|$,
	\[
	\delta(R)\ge \frac{\alpha}{2}k=\frac{\alpha}{2}|R|.
	\]
	
		It remains to prove that $R$ is a robust $(\nu/2,2\tau)$-expander; this part
	does not use the minimum-degree assumption.
	Let $S\subseteq V(R)$ with $2\tau k\le |S|\le (1-2\tau)k$, and let $S_G$
	be the union of all clusters corresponding to the vertices of $S$. Then
	$|S_G|=|S| m$ and hence, since $|V_0|\le \varepsilon n$ and
	$\varepsilon\ll \tau$, we have $\tau n\le |S_G|\le (1-\tau)n$. Since $G$ is a
	robust $(\nu,\tau)$-expander, $|RN_{\nu, G}(S_G)| \ge |S_G|+\nu n$.
	
	For any vertex $x\in RN_{\nu,G}(S_G)$, we have
	$|N_{G'}(x)\cap S_G| \ge|N_G(x)\cap S_G|-(d+\varepsilon)n \ge \nu n-(d+\varepsilon)n
	\ge \frac{\nu n}{2}$
	since $d,\varepsilon\ll \nu$. Thus $x\in RN_{\nu/2,G'}(S_G)$, so
	$RN_{\nu,G}(S_G)\subseteq RN_{\nu/2,G'}(S_G)$ and
	$|RN_{\nu/2,G'}(S_G)| \ge |S_G|+\nu n \ge |S|m+\nu mk$.
	
	Let $x\in RN_{\nu/2,G'}(S_G)\setminus V_0$ and suppose that $x\in V_j$.
	Since $|N_{G'}(x)\cap S_G|\ge \frac{\nu n}{2}$, the vertex $x$ has neighbors
	in at least $\frac{(\nu n/2)}{m}\ge \frac{\nu k}{2}$ clusters $V_i\in S$. By
	the definition of $R$, this implies that $V_j$ has at least $\nu k/2$
	neighbors in $S$, and so $V_j\in RN_{\nu/2,R}(S)$.
	
	It follows that every cluster containing a vertex of
	$RN_{\nu/2,G'}(S_G)\setminus V_0$ lies in $RN_{\nu/2,R}(S)$. Hence
	\begin{align*}
		|RN_{\nu/2,R}(S)|	 & \ge
		\frac{|RN_{\nu/2,G'}(S_G)|-|V_0|}{m} \\
		& \ge \frac{|S|m+\nu mk-|V_0|}{m} \\
		& \ge |S|+\frac{\nu k}{2},
	\end{align*}
	since $|V_0|\le \varepsilon n\ll \nu mk$. Therefore $R$ is a robust
	$(\nu/2,2\tau)$-expander.
\end{proof}

 We will also 
need the following version of the Chernoff bound (see
e.g.~\cite[Theorem A.1.16]{MR2437651}).

\begin{LEM}\label{lem:Chernoffbound}
	Let $X_1,\ldots, X_n$ be mutually independent random variables that satisfy
	$E(X_i)=0$ and $|X_i| \le 1$ for each $i\in [n]$. Set
	$S=X_1+\ldots+X_n$. Then for any $a>0$,
	$$
	\pr(|S|>a)<2e^{-a^2/2n}.
	$$
\end{LEM}

\begin{LEM}\label{lem:bipartite-robust-slicing}
	Let $M',n_0$ be positive integers and let $\varepsilon,d,\nu,\tau$ be positive
	constants such that
	\[
	\frac1{n_0}\ll \varepsilon \ll d \ll \nu \le \tau < 1
	\qquad\text{and}\qquad
	M'\ll n_0.
	\]
	Let $G$ be a robust $(\nu,\tau)$-expander on $n\ge n_0$ vertices, where $n$
	is even, and let $v_1,v_2\in V(G)$ be two specified vertices. Then there
	exists a partition $\{A,B\}$ of $V(G)$ such that
	\begin{enumerate}[(i)]
		\item  $|A| =|B|$ and $|A\cap \{v_1,v_2\}|=1$;
		\item $|d_G(v,A)-d_G(v,B)|\le n^{2/3}$ for every $v\in V(G)$;
		\item $H=G[A,B]$ is a bipartite robust $(d\nu^2/40,3\tau)$-expander.
	\end{enumerate}
\end{LEM}

\begin{proof}
	Apply the degree form of Szemer\'edi's regularity lemma to $G-\{v_1,v_2\}$.
	Since deleting two vertices changes robust expansion only negligibly,
	$G-\{v_1,v_2\}$ is still a robust $(0.9\nu,1.1\tau)$-expander for $n$
	sufficiently large.
	
	Let $R$ be the reduced graph of $G-\{v_1,v_2\}$ with parameters
	$\varepsilon,d,M'$, where $V_1,\dots,V_k$ are the clusters and $V_0$ is the
	exceptional set. By Lemma~\ref{lem:reduced-robust}, the graph $R$ is a robust
	$(0.45\nu,2.2\tau)$-expander.
	
	By moving one vertex from each cluster of odd size to $V_0$ if necessary, we
	may assume that every cluster has even size. Thus we may write
	\[
	|V_i|=2\ell \qquad\text{for all }i\in[k],
	\]
	and still
	\[
	|V_0|\le \varepsilon n + k \le 2\varepsilon n
	\]
	since $k$ is bounded in terms of $\varepsilon$.
	
	\medskip
	\noindent
	\textbf{Construction of the partition.}
	Pair $v_1$ with $v_2$, and for each $i\in [0,k]$, partition the vertices of
	$V_i$ into arbitrary pairs. For each pair, assign one vertex to $A$ and the
	other to $B$ uniformly at random. Every partition obtained in this way
	automatically satisfies~(i).
	
	We show some such partition also satisfies~(ii). Fix any vertex $v\in V(G)$.
	Let $\{x_1,y_1\}, \ldots, \{x_{n/2}, y_{n/2}\}$ be the pairs used in the
	construction. For each $i$, define
	\[
	X_i=
	\begin{cases}
		e_G(v,x_i)-e_G(v,y_i), &\text{if }x_i\in A,\ y_i\in B,\\[1mm]
		e_G(v,y_i)-e_G(v,x_i), &\text{if }x_i\in B,\ y_i\in A.
	\end{cases}
	\]
	Then $X_i\in\{-1,0,1\}$, so $|X_i|\le 1$. Moreover, the expectation of $X_i$
	is $0$ for every $i$, and the random variables $X_1,\dots,X_{n/2}$ are
	mutually independent, since the random choices on distinct pairs are
	independent.
	
	Let $S_v=X_1+\cdots+X_{n/2}$. By construction, $S_v=d_G(v,A)-d_G(v,B)$.
	Applying Lemma~\ref{lem:Chernoffbound},
	\[
	\Pr\bigl(|d_G(v,A)-d_G(v,B)|>n^{\frac 23}\bigr)
	=
	\Pr(|S_v|>n^{\frac 23})
	\le 2e^{-n^{\frac 13}}
	\]
	for all sufficiently large $n$. Taking a union bound over all $v\in V(G)$,
	the probability that condition~(ii) fails for some vertex is at most
	$2n e^{-n^{\frac 13 }}<1$ for all sufficiently large $n$. Hence there exists a
	partition $\{A,B\}$ produced by the above construction satisfying both~(i)
	and~(ii). Fix such a partition.
	
	For each $i\in[k]$, let $A_i= A\cap V_i$ and $B_i= B\cap V_i$. Then
	$|A_i|=|B_i|=\ell$. We show that $H=G[A,B]$ is a bipartite robust
	$(d\nu^2/40,3\tau)$-expander.  The argument is adapted from the proof of Lemma~6.3 in~\cite{KO2014}.
	
	We prove the expansion property only for subsets of $A$, since the argument
	for subsets of $B$ is symmetric.  
	 Let $S\subseteq A$ satisfy
	$3\tau |A|\le |S|\le (1-3\tau)|A|$ and let $\gamma =0.45\nu$. 
	 For each $i\in [k]$, let
	\begin{align*}
		S_R&=\{V_i: i\in [k], 	|S\cap A_i|\ge \frac{\gamma\ell}{5}\}.
	\end{align*}
	Then
	\begin{align}
		|S_R| & \ge \frac{|S|-\gamma \ell k/5-|V_0|-2}{\ell} \nonumber \\
		& \ge \frac{|S|}{\ell}-\frac{\gamma k}{4} \ge 2.2\tau k.  \label{eqn:S_R-size}
	\end{align}
	The last inequality holds since $|S|\ge 3\tau |A|$, $|A| \ge k\ell$, and
	$\gamma \ll \tau$.
	
	Since $R$ is a robust $(\gamma,2.2\tau)$-expander, if
	$|S_R|\le (1-2.2\tau)k$, then  $|RN_{\gamma,R}(S_R)|\ge |S_R|+\gamma k$. 
	If $|S_R|>(1-2.2\tau)k$, then applying robust expansion to an arbitrary
	subset of $S_R$ of size $(1-2.2\tau)k$ gives $|RN_{\gamma,R}(S_R)|\ge (1-2.2\tau+\gamma)k$.

	Let
	\[
	N_H=\bigcup_{V_j\in RN_{\gamma, R}(S_R)} B_j.
	\]
	Since $|B_j| =\ell$ for $j\in [k]$, if $|S_R| \le (1-2.2\tau)k$, then we
	have
	\[
	|N_H|
	\ge |S_R|\ell+\gamma k\ell
	\ge |S| +\frac{3\gamma k\ell }{4}
	\ge |S|+\frac{\gamma |A|}{2}.
	\]
	If $|S_R| > (1-2.2\tau)k$, then since $|S|\le (1-3\tau)|A|$,
	\[
	|N_H|
	\ge  (1-2.2\tau+\gamma)k\ell
	\ge |S|+\frac{\gamma |A|}{2}.
	\]
	Thus in either case,
	\begin{equation} \label{eqn:N_H-size}
		|N_H|\ge |S|+\frac{\gamma |A|}{2}.
	\end{equation}
	
	Now let $N^*=N_H\setminus RN_{d\gamma^2/8,H}(S)$. We claim that
	$|N^*|\le \frac{\gamma |A|}{4}$. Together with~\eqref{eqn:N_H-size}, this
	implies
	\[
	|RN_{d\gamma^2/8,H}(S)|\ge |S|+\frac{\gamma |A|}{4}\ge |S|+\frac{d\gamma^2}{4}|A| >|S|+\frac{d\nu^2}{40}|A|=|S|+\frac{d\nu^2}{40}|B|,
	\]
	since $d\ll \gamma$ and  $\gamma =0.45\nu$. 
	Since  $d\gamma^2/8>d\nu^2/40$, we then get 
	$$|RN_{d\nu^2/40,H}(S)|>|S|+\frac{d\nu^2}{40}|B|.$$

	Suppose for a contradiction that $|N^*|\ge \frac{\gamma |A|}{4}$. Then there
	exists a cluster $V_j$ such that $|B_j\cap N^*|\ge \frac{\gamma \ell}{5}$. In
	particular, $V_j\in RN_{\gamma, R}(S_R)$. Let $B_j^*=B_j\cap N^*$. Since no
	vertex in $B_j^*$ lies in $RN_{d\gamma^2/8,H}(S)$, we have
	\begin{align}
		e_H(S,B_j^*) &
	<\frac{d\gamma^2}{8} |A| |B_j^*|. \label{eqn:e_H(S,B_j^*)-upper}
	\end{align}
	
	On the other hand, since $V_j\in RN_{\gamma, R}(S_R)$, the cluster $V_j$ has
	at least $\gamma k$ neighbors in $S_R$ in the reduced graph. Let $S_R^*$ be
	the set of these neighboring clusters. Then $|S_R^*|\ge \gamma k$. For every
	$V_h\in S_R^*$ we have $|A_h\cap S|\ge \frac{\gamma\ell}{5}$.
	
	Moreover, since $V_hV_j\in E(R)$, the pair $(V_h,V_j)$ is $\varepsilon$-regular
	in the pure graph of $G-\{v_1,v_2\}$ with density at least $d$. Now
	$|B_j^*|\ge \frac{\gamma\ell}{5}\ge \varepsilon(2\ell+1)$, and
	$|A_h\cap S|\ge \frac{\gamma\ell}{5}\ge \varepsilon(2\ell+1)$, because
	$\varepsilon\ll \gamma$. Therefore, by $\varepsilon$-regularity, the number of
	edges in $G$ from $A_h\cap S$ to $B_j^*$ is at least
	\[
	(d-\varepsilon)|A_h\cap S||B_j^*|
	\ge \frac{5d}{6}\cdot \frac{\gamma\ell}{5}|B_j^*|
	= \frac{d\gamma\ell}{6}|B_j^*|.
	\]
	Summing over all $V_h\in S_R^*$ shows that
	\[
	e_H(S,B_j^*)
	\ge \gamma k\cdot \frac{d\gamma\ell}{6}|B_j^*|
	\ge \frac{d\gamma^2}{6}k\ell|B_j^*|
	\ge \frac{d\gamma^2}{7}|A||B_j^*|,
	\]
	a contradiction to~\eqref{eqn:e_H(S,B_j^*)-upper}. This completes the proof.
\end{proof}

\subsection{Proof of Theorem~\ref{thm:robust-expander2}}

We will also need the following result, 
due to K\"uhn and Osthus~\cite{KO2014}. 
It has 
served as a powerful tool in establishing other results.

\begin{THM}[\cite{KO2014}]\label{thm:regular-robust-decomposition} For every
	$\alpha>0$ there exists $\tau=\tau(\alpha)> 0$ such that for every $\nu> 0$
	there exists an integer $N_0=N_0(\alpha, \nu, \tau)$ for which the following
	holds. Suppose that
	\begin{enumerate}[(i)]
		\item $G$ is an $r$-regular graph on $n\geq N_0$ vertices, where
		$r\geq \alpha n$ is even;
		
		\item $G$ is a robust $(\nu, \tau)$-expander.
	\end{enumerate}
	Then $G$ has a Hamilton decomposition. Moreover, this decomposition can be
	found in polynomial time in $n$.
\end{THM}

\begin{proof}[Proof of Theorem~\ref{thm:robust-expander2}]
	Let $\tau^*=\tau(\alpha/3)$ be defined as in
	Theorem~\ref{thm:regular-robust-decomposition}. We   choose an additional constant $d$ such that
	\begin{equation}\label{eqn:parameters}
		0<1/n_0\ll \eta\ll d\ll \nu\le \tau\ll \tau^*, \alpha<1. 
	\end{equation}
	Let $G$ be a regular multigraph on $n\ge n_0$ vertices with $n>\Delta(G) \ge \alpha n$,
	where $n$ is even.  
	Suppose that $G$ satisfies the following
	conditions.
	\begin{enumerate}[(1)]
		\item $G$ is the edge-disjoint union of three multigraphs $D,L$, and $N$,
		all on the vertex set $\{v_1,\ldots,v_n\}$, such that
		$\Delta(D)=\Delta(G)$ and $G=D\cup L\cup N$ is a  canonical completion
		of $D$  defined in  Definition~\ref{def:canonical-completion}. In
		addition, the following hold. 
	\begin{enumerate}[(a)]
	\item $e_D(v_1,v_i) \le 2\eta^2 n$ for each $i\in [3,n]$.  
	\item $d_G^s(v_i)\ge \frac{i-2}{i-1}\Delta(G)-6\eta n$ for every
	$i\in [3,n]$. 
	
	\item $|N_G^s(v_2)|\ge \Delta(G)-e_G(v_1,v_2)-13\eta n$, and
	$|N_G^s(v_i)|\ge \Delta(G)-e_G(v_1,v_i)-6\eta n$ for every
	$i\in [3,n]$. 
\end{enumerate}		
		\item The underlying simple graph of $G$ is a robust $(\nu,\tau)$-expander.
	\end{enumerate}
	
	Throughout the proof, let
	$\Delta=\Delta(G)$. 
	We prove that $\chi'(G) =\Delta$ if and only if $G$ contains no $\Delta$-overfull subgraph.

	If $\chi'(G)=\Delta$, then $G$ contains no $\Delta$-overfull subgraph.
	Thus assume that $G$ contains no $\Delta$-overfull subgraph; we show that
	$\chi'(G)=\Delta$.
	
	We split into two cases according to $e_G(v_1,v_2)$. Case~1 treats the situation
	where the edges incident with $v_1$ are highly concentrated in
	$E_G(v_1,v_2)$, which is an obstacle for the approach used in Case 2.

	{\bf \noindent Case 1: $e_G(v_1,v_2) \ge \Delta -\eta^{\frac{1}{7}} n$. }
	
	\medskip
	
	We first remove a small number of Hamilton cycles
	to cover  all edges in  $E_G(\{v_1,v_2\},V(G)\setminus\{v_1,v_2\})$.

Let
	$N_G(\{v_1,v_2\})\setminus\{v_1,v_2\}=\{w_1,\ldots,w_m\}$. For each
	$i\in[m]$, let $c_i=e_G(v_1,w_i)$ and $d_i=e_G(v_2,w_i)$. Then
	\[
	\sum_{i=1}^m c_i=d_G(v_1)-e_G(v_1,v_2)=d_G(v_2)-e_G(v_1,v_2)=\sum_{i=1}^m d_i.
	\]
	We claim that $c_i+d_i\le \sum_{j\ne i}(c_j+d_j)$ for every $i\in[m]$.
	Suppose not. Then for some $i\in[m]$,
	$c_i+d_i>\sum_{j\ne i}(c_j+d_j)$, and hence
	$c_i+d_i>\Delta-e_G(v_1,v_2)$. Thus
	$e_G(v_1,v_2)+c_i+d_i>\Delta$, so $G[\{v_1,v_2,w_i\}]$ is
	$\Delta$-overfull, a contradiction.
	
	By Lemma~\ref{lem:graphical-biparite2}, applied after splitting each $w_i$
	into two vertices $x_i$ and $y_i$, there is a bipartite multigraph $Q$ with
	parts $\{x_1,\ldots,x_m\}$ and $\{y_1,\ldots,y_m\}$ such that
	$d_Q(x_i)=c_i$, $d_Q(y_i)=d_i$, and $e_Q(x_i,y_i)=0$ for every $i\in[m]$.
	After identifying $x_i$ and $y_i$ back to $w_i$, each edge of $Q$ gives a pair
	of edges of $G$, one incident with $v_1$ and one incident with $v_2$, with
	distinct other endpoints. Let $q=e(Q)=\sum_{i=1}^m c_i$. By the assumption of
	Case~1, $q\le \eta^{\frac{1}{7}}n$.
	
	We process the edges of $Q$ one by one. Suppose that the current edge of $Q$
	corresponds to the unused copies of the edges $v_1w_i$ and $v_2w_j$, where
	$i\ne j$. Let $S=V(G)\setminus\{v_1,v_2\}$. After the Hamilton cycles chosen in
	earlier rounds have been deleted from $G[S]$, the remaining underlying simple
	graph on $S$ is still a robust $(\nu/2,2\tau)$-expander. Indeed, deleting the
	two vertices $v_1,v_2$ and then deleting the edges used in the previous rounds
	removes maximum degree at most $2q\le 2\eta^{\frac{1}{7}}n\ll \nu n$, so this follows
	from Condition~\eqref{item:thm4.2} and Lemma~\ref{lem:stability-expansion}.
	Moreover, by Condition~\eqref{item:thm4.1.b}, we have
	$\delta(G^s[S])\ge \alpha n/3$. Hence Lemma~\ref{lem:hamiltonicity-of-expander},
	with $k=1$, gives a Hamilton path $P$ in the current underlying simple graph on
	$S$ with ends $w_i$ and $w_j$. Choose an unused edge $f\in E_G(v_1,v_2)$; this
	is possible since $q=\Delta-e_G(v_1,v_2)\le \eta^{ \frac{1}{7}}n\ll \Delta$. Then
	$v_1w_iPw_jv_2fv_1$ is a Hamilton cycle of $G$. Repeating this for all edges of
	$Q$, we obtain $q$ edge-disjoint Hamilton cycles covering all edges in
	$E_G(\{v_1,v_2\},S)$.

	Let $G^{**}$ be the multigraph obtained from $G$ after deleting all edges
	covered by these $q$ Hamilton cycles. Then
	$e_{G^{**}}(\{v_1,v_2\},V(G^{**})\setminus\{v_1,v_2\})=0$, and $G^{**}$ is
	$r$-regular, where $r=\Delta-2q$. In particular,
	$e_{G^{**}}(v_1,v_2)=r$.
	
	Let  $R=G^{**}[S]$. Then $R$ is an 
	$r$-regular multigraph with even order $n-2$. Moreover, using Condition~\eqref{item:thm4.1.c}
	together with the fact that all edges from $\{v_1,v_2\}$ to $S$ have been
	deleted, we obtain 
	$\Delta(R-E(R^s))\le 6\eta n$.
	
	By Theorem~\ref{thm:chromatic-index}, we have 
	$\chi'(R-E(R^s))\le \Delta(R-E(R^s))+\mu(R-E(R^s))\le 12\eta n$. After edge-coloring $R-E(R^s)$ with at
	most $12\eta n$ colors, partition each
	color class greedily into submatchings of size at most $\nu^4n/81$.  Thus $E(R-E(R^s))$ can be partitioned into matchings
	$M_1,\ldots,M_\ell$ such that $|M_i|\le \nu^4 n/81$ for every $i\in[\ell]$ and
	$\ell<\nu^2 n$. 
	
	Remove spanning linear forests from $R^s$ as follows. Let $R_0=R^s$. Suppose
	that $R_0,\ldots,R_{i-1}$ and $F_1,\ldots,F_{i-1}$ have already been
	constructed. Since
	$\Delta(F_1\cup\cdots\cup F_{i-1})\le 2(i-1)\le 2\nu^2 n$, the graph
	$R_{i-1}$ is still a robust $(\nu/3,3\tau)$-expander by
	Lemma~\ref{lem:stability-expansion}. Applying
	Lemma~\ref{lem:hamiltonicity-of-expander} to $R_{i-1}$ and the pairs given by
	the matching $M_i$, we obtain a spanning linear forest $F_i\subseteq R_{i-1}$
	such that $F_i\cup M_i$ is a Hamilton cycle  of $R$. Set
	$R_i=R_{i-1}-E(F_i)$.
	
	We show that $R_\ell$ is regular. For every $u\in S$ and every $i\in[\ell]$,
	we have $d_{F_i}(u)=1$ if $u\in V(M_i)$ and $d_{F_i}(u)=2$ otherwise. Since
	$M_1,\ldots,M_\ell$ partition $E(R-E(R^s))$, it follows that
	$\sum_{i=1}^\ell d_{F_i}(u)=2\ell-d_{R-E(R^s)}(u)$. Therefore
	\[
	d_{R_\ell}(u)=d_{R^s}(u)-\sum_{i=1}^\ell d_{F_i}(u)
	=(r-d_{R-E(R^s)}(u))-(2\ell-d_{R-E(R^s)}(u))=r-2\ell.
	\]
	Thus $R_\ell$ is a simple $(r-2\ell)$-regular graph. Moreover,
	$r-2\ell\ge \Delta-2\eta^{\frac{1}{7}}n-2\nu^2n\ge \alpha n/2$, and $R_\ell$ is
	still a robust $(\nu/3,3\tau)$-expander.
	
	If $r-2\ell$ is odd, then Lemma~\ref{lem:hamiltonicity-of-expander} gives a
	Hamilton cycle in $R_\ell$, and hence a perfect matching $M_0$ in $R_\ell$.
	Replace $R_\ell$ by $R_\ell-M_0$. The resulting graph is still a robust
	$(\nu/4,4\tau)$-expander and is even-regular with degree at least
	$\alpha n/3$. If $r-2\ell$ is even, set $M_0=\emptyset$ and keep $R_\ell$ as
	it is. Now Theorem~\ref{thm:regular-robust-decomposition} applies with
	parameters $\nu/4$, $\tau^*$, and $\alpha/3$. Since each Hamilton cycle on the
	even vertex set $S$ decomposes into two perfect matchings, the graph $R_\ell$
	has a decomposition into $r-2\ell$ perfect matchings, including $M_0$ if it was
	chosen.
	
	For each $i\in[\ell]$, the graph $F_i\cup M_i$ is a Hamilton cycle on the even
	vertex set $S$, and hence decomposes into two perfect matchings. Therefore $R$
	has a decomposition into $(r-2\ell)+2\ell=r$ perfect matchings. Since
	$e_{G^{**}}(v_1,v_2)=r$ and there are no other edges from $\{v_1,v_2\}$ to $S$
	in $G^{**}$, adding one distinct edge from $E_{G^{**}}(v_1,v_2)$ to each of
	these perfect matchings gives an edge $r$-coloring of $G^{**}$. Each of the $q$ Hamilton cycles removed earlier 
	is edge-$2$-colorable.  Thus we conclude 
	that $\chi'(G)=\Delta$.

	{\bf \noindent Case 2: $e_G(v_1,v_2)< \Delta-\eta^{\frac{1}{7}}n$. }

	In this case, we follow the five-step framework 
	below   and 
decompose the edges of $G$ into $\Delta$ edge-disjoint 1-factors.
The framework  records the purpose of each step before 
the detailed construction and estimates are given.

	\begin{enumerate}[Step 1.]
	\item By Lemma~\ref{lem:bipartite-robust-slicing}, partition $V(G)$ into
	two subsets $A$ and $B$ satisfying the following three properties:
	\begin{enumerate}[(i)]
		\item $|A|=|B|$ and $|A\cap\{v_1,v_2\}|=1$;
		\item $|d_{G^s}(v,A)-d_{G^s}(v,B)|\le n^{2/3}$ for every $v\in V(G)$;
		\item $H^s=G^s[A,B]$ is a bipartite robust $(d\nu^2/40,3\tau)$-expander.
	\end{enumerate}
	We then adjust the partition so that $v_1\in A$, $v_2\in B$, and every vertex
	$u\in V(G)\setminus\{v_1,v_2\}$ with $e_G(v_1,u)>6\eta n$ lies in $B$. This adjustment is needed to ensure that the high-multiplicity pairs involving $v_1$ are separated by the bipartition. We also arrange the two parts slightly so that
	$$
d_G(v_1,A)\le \Delta/2.
	$$

	\item Let $G_{A,B}$ be the subgraph of $G$ consisting of all edges inside
	$A$, all edges inside $B$, together with some 
	edges incident with $v_1$ that go between $A$ and $B$. These cross-edges
	are chosen so that, for every $u\in B$, at most $\lceil e_G(v_1,u)/2\rceil + 7\eta n$
	of the edges between $v_1$ and $u$ are put into $G_{A,B}$. In particular,
	each high-multiplicity bundle between $v_1$ and a vertex of $B$ is split as
	evenly as needed between $G_{A,B}$ and $H-E(G_{A,B})$. It can be shown that
	$\Delta(G_{A,B})\le \Delta/2+ 14 \eta n$. We equitably edge-color $G_{A,B}$
	with $k$ colors, where $k  =\Delta/2+O(\eta n)$.
	
	\item In this step, we use the expansion property of the partition $\{A,B\}$
	to extend each color class obtained from Step~2 into a 1-factor of $G$.
	Specifically, since $|V(G)|$ is even, for each color $i\in [k]$, the number
	of vertices of $G_{A,B}$ missing color $i$ is even. We pair up these
	vertices. Let $(u,v)$ be such a pair. We find a path $P$ connecting $u$ and
	$v$ which starts and ends with an uncolored edge between $A$ and $B$, and
	whose edges alternate between uncolored edges between $A$ and $B$ and edges
	colored $i$ inside $A$ or inside $B$. We then swap color $i$ and the
	null-color along the edges of $P$. This increases the size of the color
	class $i$ by  one  and makes color $i$ present at both $u$ and $v$. We
	continue this process for all other vertices missing color $i$, and then for
	all other color classes. After this, each of the $k$ color classes becomes a
	1-factor. The vertices $v_1$ and $v_2$, which may have small simple degree,
	are treated with priority when we deal with colors missing at them. The
	existence of such paths is guaranteed by the robust expansion property of
	the partition. During this process, we make sure that the subgraphs inside
	$A$ and inside $B$, denoted by $R_A$ and $R_B$, respectively, induced by
	the set of edges that become uncolored in the process, have maximum degree
	$O(\eta^{\frac{1}{3}}n)$.
	
	\item We equitably edge-color $R_A\cup R_B$ with another $\ell$ colors. The
	equitable coloring ensures that, for each color $i\in [k+1,k+\ell]$, the
	number of edges colored $i$ inside $A$ is the same as the number of edges
	colored $i$ inside $B$. We then extend each color class obtained in this
	step into a 1-factor by finding a matching between $A$ and $B$ that
	saturates all vertices still missing color $i$. Again, the existence of such
	a matching is guaranteed by the robust expansion property of the partition.
	Repeating this for each new color class turns all of them into 1-factors.
	
	\item After Steps~3 and~4, all edges of $G$ inside $A$ and inside $B$ have
	been colored. The uncolored edges are only between $A$ and $B$, and they
	form a $(\Delta-k-\ell)$-regular bipartite multigraph. Hence this remaining
	bipartite multigraph has a 1-factorization by Theorem~\ref{konig}.
\end{enumerate}

We now provide the details for each step. 
	
	\begin{center}
		{\bf Step 1: Partition $V(G)$ into two desired subsets $A$ and $B$}
	\end{center}
	
	Let
	$Z_0=\{u\in V(G)\setminus\{v_1,v_2\}: e_G(v_1,u)>6\eta n\}$.
	Since $d_G(v_1)=\Delta < n$, we have
	$|Z_0|  \le \Delta/(6\eta n)<1/(6\eta)\ll \eta n$.
	
	Apply Lemma~\ref{lem:bipartite-robust-slicing} to $G^s$, the underlying
	simple graph of $G$, to obtain a bipartition $(A,B)$ of $V(G)$ satisfying
	the following properties:
	\begin{enumerate}[(i)]
		\item $|A|=|B|$ and $|A\cap\{v_1,v_2\}|=1$;
		\item $|d_{G^s}(v,A)-d_{G^s}(v,B)|\le n^{2/3}$ for every $v\in V(G)$;
		\item $H^s=G^s[A,B]$ is a bipartite robust $(d\nu^2/40,3\tau)$-expander.
	\end{enumerate}
	After interchanging the names of $A$ and $B$ if necessary, assume that
	$v_1\in A$ and $v_2\in B$.
	
	Since $|Z_0|\ll \eta n$, we can choose, for every $z\in Z_0$, a distinct vertex
	$z'\in V(G)\setminus (Z_0\cup\{v_1,v_2\})$ lying in the opposite part from
	$z$. We call $z'$ the partner of $z$, and let $Z_0'$ be the set of all these
	partners. Since $Z_0'\cap Z_0=\emptyset$, we have
	$d_G(v_1,Z_0')\le 6\eta n|Z_0'|=6\eta n|Z_0|<d_G(v_1,Z_0)$.
	
	For every vertex $z\in Z_0\cap A$, swap the positions of $z$ and its partner
	$z'$. We still denote the resulting parts by $A$ and $B$. After these swaps,
	we have $|A|=|B|$, $v_1\in A$, $v_2\in B$, $Z_0'\subseteq A$, and
	$Z_0\subseteq B$. We will swap vertices of $Z_0$ and $Z_0'$
	so that the following holds:
	\begin{eqnarray}
		 d_G(v_1, A\setminus (Z_0' \cup \{v_1\})) & \le & d_G(v_1, B\setminus (Z_0 \cup \{v_2\})), \quad \text{and} \label{eqn:v1-degree1}\\
		  d_G(v_1, A )  &\le & d_G(v_1, B).   \label{eqn:v1-degree2}
	\end{eqnarray}
	Note that~\eqref{eqn:v1-degree2} is a consequence of~\eqref{eqn:v1-degree1} 
	since $d_G(v_1,Z_0' \cup \{v_1\})<d_G(v_1,Z_0 \cup \{v_2\})$ by the definition of $Z_0$. 
	
	If $d_G(v_1, A\setminus (Z_0' \cup \{v_1\})) \le d_G(v_1, B\setminus (Z_0 \cup \{v_2\}))$, then we keep this partition. Otherwise, suppose
	the contrary. Let
	$A_0=A\setminus(\{v_1\}\cup Z_0')$ and
	$B_0=B\setminus(\{v_2\}\cup Z_0)$. 
	In this case, replace the partition by
	$A=\{v_1\}\cup Z_0'\cup B_0$ and
	$B=\{v_2\}\cup Z_0\cup A_0$. Then $|A|=|B|$, $v_1\in A$, $v_2\in B$, and
	$Z_0\subseteq B$. 
	Thus, in all cases, the final partition satisfies $|A|=|B|$, $v_1\in A$,
	$v_2\in B$, $Z_0\subseteq B$, and the partition satisfies both~\eqref{eqn:v1-degree1} and~\eqref{eqn:v1-degree2}.
	
	We now compare the final partition with the partition originally obtained from
	Lemma~\ref{lem:bipartite-robust-slicing}. In the first round of swaps, we move
	at most $2|Z_0|$ vertices, namely the vertices of $Z_0\cap A$ and their
	partners. In the second case above, we interchange only the smaller exceptional
	sets $\{v_1\}\cup Z_0'$ and $\{v_2\}\cup Z_0$, relative to the opposite
	orientation of the original partition. Thus, although the procedure may involve
	up to $4|Z_0|+2$ vertex moves in total, a vertex that is swapped twice is simply
	returned to the side it occupies in one of the two orientations of the original
	partition. Consequently, the final partition differs from one of the two
	orientations of the partition given by Lemma~\ref{lem:bipartite-robust-slicing}
	on at most $2|Z_0|+2$ vertices.
	Hence, for every $v\in V(G)$,
	\begin{equation}\label{eqn:partition-deg}
		|d_{G^s}(v,A)-d_{G^s}(v,B)|
		\le n^{2/3}+2|Z_0|+2
		\le n^{2/3}+\eta n.
	\end{equation}
	Moreover, this modification changes $H^s=G^s[A,B]$ only through edges incident
	with at most $2|Z_0|+2$ vertices. Therefore, by Lemma~\ref{lem:bipartite-stability}\eqref{lem:bipartite-stability-swap}, the final bipartite graph
	$H^s=G^s[A,B]$ is a bipartite robust $(d\nu^2/50,4\tau)$-expander.

\begin{center}
	{\bf Step 2: Form $G_{A,B}$ and edge-color it}
\end{center}

Let
\[
G_A=G[A], \qquad G_B=G[B], \qquad \text{and} \qquad H=G[A,B].
\]
We will move a carefully chosen set of edges of $H$ incident with $v_1$ into
$G_A\cup G_B$. The purpose of doing this is twofold. First, it allows us to
make the degrees of all vertices in the resulting graph within $O(\eta n)$ of
$\Delta/2$. Second, when the simple degree of $v_1$ is small, we need to keep
enough available edges from $v_1$ to one or two suitable neighbors for the
coloring extension in Step~4.

We first choose a possible exceptional vertex. Suppose that there exists a unique 
$w^*\in V(G)\setminus\{v_1,v_2\}$ such that
$e_G(v_1,w^*)>\eta^{\frac 16}n$ and $e_G(v_2,w^*)\ge 1$. Then, by the choice of
the partition, $w^*\in B$. In this case, set $q=e_G(v_2,w^*)$ and $W^*=\{w^*\}$. Since $G$ has no
$\Delta$-overfull subgraph, we have
$e_G(v_2,w^*)\le \Delta-e_G(v_1,v_2)-e_G(v_1,w^*)$. Moreover, by
Condition~\eqref{item:thm4.1.c}, we have $q\le 13\eta n$. If no such vertex
uniquely exists, then we let $q=0$ and $W^*=\emptyset$.

Let 
\[
Z_1=Z_0\cup\{v_2\}
\quad\text{and}\quad
Z_1'=Z_0'\cup\{v_1\}.
\]
For each $u\in Z_1$, let $u'\in Z_1'$ be its partner, where the partner of
$v_2$ is $v_1$. Thus $e_G(v_1,u')=0$ when $u=v_2$.

Let $a=d_G(v_1,A)$. We choose a set $M_0$ of
$\max\{q-a,0\}$ edges from
$E_G(v_1,B\setminus(\{v_2\}\cup W^*))$ as follows. First give priority to the
edges in $E_G(v_1,Z_0\setminus W^*)$; if these edges are not enough, choose
the remaining edges from $E_G(v_1,B\setminus Z_1)$. Thus no edge of $M_0$ is
incident with $v_2$ or with a vertex of $W^*$.

This is possible. Indeed, if $q\le a$, then $M_0=\emptyset$. If $q>a$, then
$q>0$ and $W^*=\{w^*\}$. Since $G$ has no $\Delta$-overfull subgraph, we have
\[
q=e_G(v_2,w^*)
\le \Delta-e_G(v_1,v_2)-e_G(v_1,w^*).
\]
Consequently,
\[
\begin{aligned}
	q-a
	&\le
	\Delta-e_G(v_1,v_2)-e_G(v_1,w^*)-d_G(v_1,A) \\
	&=
	e_G(v_1,B\setminus\{v_2,w^*\}),
\end{aligned}
\]
so there are enough available edges to choose $M_0$.

Let
\[
t=\left\lfloor\frac{d_G(v_1,B)-d_G(v_1,A)-|M_0|}{2}\right\rfloor, 
\]
and let
\[
M_0^1=M_0\cap E_G(v_1,Z_1)
\quad\text{and}\quad
M_0^2=M_0\cap E_G(v_1,B\setminus Z_1).
\]
For each $u\in B$, let $m_u=e_{G[M_0]}(v_1,u)$.
Let $Z_2=\{z\in Z_1:e_G(v_1,z)-e_G(v_1,z')-m_z\ge0\}$, 
let $Z_2'$ be the set of partners of the vertices in $Z_2$, and let $M_0^*$
be the set of edges of $M_0$ between $v_1$ and $Z_2$. Note that
$\{v_2\}\cup W^*\subseteq Z_2$.
Set $s=|M_0| -|M_0^*|$.  

We choose $M^1_1\subseteq E_G(v_1,B)\setminus M_0$ in two stages. In the first
stage, we choose edges from the remaining $v_1u$-bundles with $u\in Z_2$.
The choices are made so that, for
every $u\in Z_2$ with partner $u'$,
\[
\left\lfloor
\frac{e_G(v_1,u)-e_G(v_1,u')-m_u}{2}
\right\rfloor
-s_u
\le
e_{G[M_1^1]}(v_1,u)
\le
\left\lceil
\frac{e_G(v_1,u)-m_u}{2}
\right\rceil,
\]
where the nonnegative integers $s_u$ satisfy $\sum_{u\in Z_2}s_u\le s+\eta n$.

In particular, if 
$\sum_{u\in  Z_2}
\left\lceil\frac{e_G(v_1,u)-e_G(v_1, u')-m_u}{2}\right\rceil \ge t$, then we choose $M_1^1$ so that $|M_1^1|=t$. 
This is possible by the following estimate, since 
\[
\begin{aligned}
	d_G(v_1,B)-d_G(v_1,A)-|M_0|
	&=
	d_G(v_1,Z_2)-d_G(v_1,Z_2')-|M_0^*| \\
	&\quad+
	d_G(v_1,Z_1\setminus Z_2) -d_G(v_1,Z'_1\setminus Z'_2)
	-\bigl(|M_0^1|-|M_0^*|\bigr) \\
	&\quad+
	d_G(v_1,B\setminus Z_1)
	-d_G(v_1,A\setminus Z_1')-|M_0^2| \\
	&\ge
	d_G(v_1,Z_2)-d_G(v_1,Z_2')-|M_0^*|-s, 
\end{aligned}
\]
where the last inequality follows by $d_G(v_1,Z_1\setminus Z_2)  \ge d_G(v_1,Z'_1\setminus Z'_2)$
from  the definition of $Z_0$
and  by $d_G(v_1,B\setminus Z_1) \ge d_G(v_1,A\setminus Z_1')$ from~\eqref{eqn:v1-degree1}. 

If $\sum_{u\in  Z_2}
\left\lceil\frac{e_G(v_1,u)-e_G(v_1, u')-m_u}{2}\right\rceil <t$, 
then we choose $M_1^1$ so that  $|M_1^1| =\sum_{u\in  Z_2}
\left\lceil\frac{e_G(v_1,u)-e_G(v_1, u')-m_u}{2}\right\rceil $.

In the second stage,  if $|M_1^1|=t$, then we let $M_1^2 =\emptyset$. 
Otherwise, we choose $M_1^2 \subseteq E_G(v_1,B\setminus Z_1)\setminus(M_0\cup M_1^1)$ consisting of 
  $t-|M_1^1|$ further edges  
 so that, for every $u\in B\setminus Z_1$,
\[
e_{G[M^2_1]}(v_1,u)
\le
\left\lceil\frac{e_G(v_1,u)-m_u}{2}\right\rceil.
\]
The second-stage choice is possible. Since
\[
\sum_{u\in  Z_1}
\left\lceil\frac{e_G(v_1,u)-e_G(v_1, u')-m_u}{2}\right\rceil+\sum_{u\in B\setminus Z_1}
\left\lceil\frac{e_G(v_1,u)-m_u}{2}\right\rceil
\ge
\left\lfloor\frac{d_G(v_1,B)-d_G(v_1,A)-|M_0|}{2}\right\rfloor
\ge t,
\]
we have 
\[
\sum_{u\in  Z_2}
\left\lceil\frac{e_G(v_1,u)-e_G(v_1, u')-m_u}{2}\right\rceil+\sum_{u\in B\setminus Z_1}
\left\lceil\frac{e_G(v_1,u)-m_u}{2}\right\rceil
\ge t. 
\]
As 
$|M^1_1| = \sum_{u\in Z_2}
\left\lceil\frac{e_G(v_1,u)-e_G(v_1,u')-m_u}{2}\right\rceil$, 
we   get 
  $\sum_{u\in B\setminus Z_1}\left\lceil\frac{e_G(v_1,u)-m_u}{2}\right\rceil \ge t-|M_1^1|$.

Let $M_1=M_1^1\cup M_1^2$. Then $|M_1|=t$.
Let $M=M_0\cup M_1$, and define
\[
G_{A,B}=G_A\cup G_B+M.
\]
We call the edges of $M$ the middle edges.

For every $u\in B$, by construction,
\[
e_{G[M]}(v_1,u)
\le
m_u+\left\lceil\frac{e_G(v_1,u)-m_u}{2}\right\rceil
\le
\left\lceil\frac{e_G(v_1,u)}{2}\right\rceil+ \left \lceil \frac{|M_0|}{2} \right\rceil.
\]
Moreover, for every $u\in Z_2$ with partner $u'$,
\[
e_{G[M]}(v_1,u)
\ge
m_u+
\left\lfloor
\frac{e_G(v_1,u)-e_G(v_1,u')-m_u}{2}
\right\rfloor
-s-\eta n.
\]
In particular, when $u=v_2$, the partner is $v_1$ and $m_{v_2}=0$, so
\[
e_{G[M]}(v_1,v_2)
\ge
\left\lfloor\frac{e_G(v_1,v_2)}{2}\right\rfloor
-s-\eta n.
\]

In the exceptional case $q>0$, the triangle involving $\{v_1,v_2,w^*\}$ is
still controlled. Indeed, the edges of $M_0$ avoid both $v_2$ and $w^*$, and
$d_G(v_1,A)+|M_0|\ge q=e_G(v_2,w^*)$. Hence
\begin{equation}\label{eqn:triangle-v1-controlled}
	e(G_{A,B}[\{v_1,v_2,w^*\}])\le d_{G_{A,B}}(v_1).
\end{equation}
If $q=0$, then no exceptional vertex is designated, and no exceptional triangle
is treated separately.

\begin{CLA}\label{claim:degrees-in-G-A-B}
	For every $v\in V(G)$,
	\[
	\frac{1}{2}\Delta-21.6\eta n
	\le d_{G_{A,B}}(v)
	\le \frac{1}{2}\Delta+14\eta n.
	\]
\end{CLA}

\begin{proof}
	First consider $v=v_1$. Let $a=d_G(v_1,A)$ and $b=d_G(v_1,B)$. Then
	$a+b=\Delta$ and, by the choice of the partition, $a\le b$. Since
	$|M_0|=\max\{q-a,0\}\le q\le 13\eta n$, by the definition of $t$ and the
	choice of $M_1$ we have
	$d_{G_{A,B}}(v_1)=a+|M_0|+t$. Hence
	\[
	\frac{\Delta}{2}-1
	\le d_{G_{A,B}}(v_1)
	\le \frac{\Delta}{2}+\frac{|M_0|}{2}
	\le \frac{\Delta}{2}+6.6\eta n,
	\]
	and the desired bound holds.
	
	Now let $v\ne v_1$, and let $C\in\{A,B\}$ be the part containing $v$. Put
	$q_v=e_G(v_1,v)$ and
	$r_v=d_{G-v_1}(v)-d^s_{G-v_1}(v)$. Thus $r_v$ counts the excess
	multiplicity on the edges from $v$ to vertices other than $v_1$. By
	Condition~\eqref{item:thm4.1.c}, we have $r_v\le 13\eta n$ if $v=v_2$,
	and $r_v\le 6\eta n$ otherwise. Also,
	$d_{G^s}(v)=\Delta-(q_v-1)^+-r_v$, where
	$(q_v-1)^+=\max\{q_v-1,0\}$.
	
	Suppose first that $C=A$. Since $Z_0\subseteq B$ and $v_2\in B$, we have
	$q_v\le 6\eta n$ and $r_v\le 6\eta n$. Using
	\eqref{eqn:partition-deg}, we obtain
	\[
	\begin{aligned}
		d_{G_{A,B}}(v)
		&\le \frac{d_{G^s}(v)+n^{2/3}+\eta n}{2}+(q_v-1)^++r_v \\
		&= \frac{\Delta+(q_v-1)^++r_v+n^{2/3}+\eta n}{2}
		\le \frac{\Delta}{2}+7\eta n.
	\end{aligned}
	\]
	For the lower bound, the edges from $v$ to its own part contain the whole
	$v_1v$-bundle. Thus
	\[
	\begin{aligned}
		d_{G_{A,B}}(v)
		&\ge \frac{d_{G^s}(v)-n^{2/3}-\eta n}{2}+(q_v-1)^+ \\
		&= \frac{\Delta+(q_v-1)^+-r_v-n^{2/3}-\eta n}{2}
		\ge \frac{\Delta}{2}-4\eta n.
	\end{aligned}
	\]
	
	Now suppose that $C=B$. Let $m_v=e_{G[M_0]}(v_1,v)$. By the choice of the
	middle edges,
	\[
	e_{G[M]}(v_1,v)
	\le \left\lceil\frac{q_v-m_v}{2}\right\rceil+m_v
	\le \left\lceil\frac{q_v}{2}\right\rceil+ \left\lceil \frac{|M_0|}{2}\right \rceil.
	\]
	Since $|M_0|\le 13\eta n$, we have
	$e_{G[M]}(v_1,v)\le \lceil q_v/2\rceil+6.6\eta n$. Therefore
	\[
	\begin{aligned}
		d_{G_{A,B}}(v)
		&\le
		\frac{1}{2}\bigl(\Delta-(q_v-1)^+-r_v+n^{2/3}+\eta n\bigr)
		+r_v+\left\lceil\frac{q_v}{2}\right\rceil+6.6\eta n \\
		&\le \frac{\Delta}{2}+14\eta n,
	\end{aligned}
	\]
	where we use $r_v\le 13\eta n$ and $n^{2/3}\ll\eta n$.
	
	For the lower bound, first suppose that $v\notin Z_1$. Then
	$q_v\le 6\eta n$ and $r_v\le 6\eta n$. Since
	$e_{G[M]}(v_1,v)\ge 0$, the lower estimate from
	\eqref{eqn:partition-deg} gives
	\[
	\begin{aligned}
		d_{G_{A,B}}(v)
		&\ge
		\frac{1}{2}\bigl(\Delta-(q_v-1)^+-r_v-n^{2/3}-\eta n\bigr) \\
		&\ge
		\frac{\Delta}{2}
		-\frac{1}{2}\bigl(6\eta n+6\eta n+n^{2/3}+\eta n\bigr)
		\ge \frac{\Delta}{2}-7\eta n.
	\end{aligned}
	\]
	
	Then suppose that $v\in Z_1\setminus Z_2$. 
	By the construction of $M$, we have $e_{G[M]}(v_1,v)=e_{G[M_0]}(v_1,v)$,
	and $q_v < q_{v'}+m_v \le 6\eta n+m_v$. 
	Thus, 
	\[
	\begin{aligned}
		d_{G_{A,B}}(v)
		&\ge
		\frac{1}{2}\bigl(\Delta-(q_v-1)^+-r_v-n^{2/3}-\eta n\bigr)
		+m_v  \\
		&\ge
		\frac{\Delta}{2}
		-\frac{r_v}{2}
		+\frac{m_v}{2}
		-4\eta n \ge \frac{\Delta}{2} -7\eta n. 
	\end{aligned}
	\]
	
	It remains to consider $v\in Z_2$. Let $v'$ be the partner of $v$, where
	the partner of $v_2$ is $v_2$.
	From the first stage choice, 
 there are nonnegative integers
	$s_u$, $u\in Z_2$, such that $s_u \le 14 \eta n-|M_0^*|$. 
	Consequently,
	\[
	\begin{aligned}
		d_{G_{A,B}}(v)
		&\ge
		\frac{1}{2}\bigl(\Delta-(q_v-1)^+-r_v-n^{2/3}-\eta n\bigr)
		+\frac{q_v-q_{v'}-m_v}{2}-1-s_v \\
		&\ge
		\frac{\Delta}{2}
		-\frac{r_v}{2}
		-\frac{q_{v'}}{2}
		-\eta n-s_v,
	\end{aligned}
	\]
	where we used the fact that $m_v \le |M_0^*|$. 
	
	If $v\in Z_0$, then $v'\in Z_0'$, and hence
	$q_{v'}\le 6\eta n$. Also, $r_v\le 6\eta n$. Therefore
	\[
	d_{G_{A,B}}(v)
	\ge
	\frac{\Delta}{2}
	-\frac{6\eta n}{2}
	-\frac{6\eta n}{2}
	-\eta n-14\eta n
	\ge \frac{\Delta}{2}-21\eta n.
	\]
	If $v=v_2$, then $v'=v_1$, so $q_{v'}=0$, while $r_v\le 13\eta n$. Hence
	\[
	d_{G_{A,B}}(v_2)
	\ge
	\frac{\Delta}{2}
	-\frac{13\eta n}{2}
	-\eta n-14\eta n
	\ge \frac{\Delta}{2}-21.6\eta n.
	\]
	This proves the claim.
\end{proof}

	Let $k=\left\lceil \frac{\Delta}{2}+28\eta n\right\rceil$. By
	Claim~\ref{claim:degrees-in-G-A-B} and $\mu(G_{A,B}-v_1)\le 13\eta n$ from
Condition~\eqref{item:thm4.1.c}, we have
	$\Delta(G_{A,B})+\mu(G_{A,B}-v_1)\le k$. Thus $G_{A,B}$ has an edge
	$k$-coloring by Lemma~\ref{lem:multi-version-star-multigraph}. Applying
	Lemma~\ref{lem:equa-edge-coloring}, we obtain an equalized edge $k$-coloring
	$\varphi$ of $G_{A,B}$.
	
	In particular, for every $v\in V(G_{A,B})$, by
	Claim~\ref{claim:degrees-in-G-A-B}, 
	\begin{equation}\label{eqn:number-of-missing-colors}
		|\overline{\varphi}(v)|
		=
		k-d_{G_{A,B}}(v)
	< 49.7\eta n. 
	\end{equation}
	Hence
	\begin{equation}\label{eqn:total-number-of-missing-colors}
		\sum_{v\in V(G_{A,B})}|\overline{\varphi}(v)|
	< 49.7\eta n^2.
	\end{equation}
	Moreover, since $\varphi$ is equalized, the numbers of vertices missing any
	two colors differ by at most $2$. Therefore, using the total bound on missing
	colors and the fact that $k\ge \Delta/2\ge \alpha n/2$,
	\begin{equation}\label{eqn:upper-bound-vertices-missing-i}
		|\overline{\varphi}^{-1}(i)|
		\le \frac{49.7\eta n^2}{k}+2
		\le \frac{100\eta n}{\alpha}
		\qquad\text{for every } i\in [k],
	\end{equation}
	for $n$ sufficiently large. This partial edge coloring $\varphi$ of $G$ will
	be extended in the next step.
	
	\begin{center}
		{\bf Step 3: Extend the $k$ color classes from Step 2 into 1-factors}
	\end{center}
	
	Let $\gamma=d\nu^2/50$. We modify the partial edge coloring of $G$ obtained
	in Step~2 by exchanging alternating paths, that is, by swapping ``null-color''
	and a given color $i$ on the edges of such a path. The coloring will be
	updated frequently during this step. For clarity, let
	$\pbar_{\rm in}(v)$ denote the set of colors missing at $v$ in the coloring
	obtained at the end of Step~2. Missing colors used to define MCC-pairs below
	are always taken with respect to the current coloring.
	
	Upon the completion of Step~3, each of the $k$ color classes will be a
	1-factor of $G$. In the process of Step~3, a few edges of $H-E(G_{A,B})$
	will be colored and a few edges of $G_{A,B}$ will be uncolored. We denote
	by $R_A$ and $R_B$ respectively the subgraphs induced by the edges of $G_A$
	and $G_B$ that become uncolored in this step. They are empty initially. Each
	time we exchange colors on an alternating path, at most $2/\gamma$ edges are
	added to $R_A\cup R_B$, and hence at most $2/\gamma$ edges are added to each
	of $R_A$ and $R_B$. The following conditions will be satisfied at the
	completion of this step.
	\begin{enumerate}[S3.1]
		\item The total number of uncolored edges in each of $R_A$ and $R_B$ is
		less than $ \eta^{\frac{1}{2}}n^2$.
		Furthermore, $R_A$ and $R_B$ have the same number of uncolored edges.
		\label{s31}
		
		\item $\Delta(R_A)$ and $\Delta(R_B)$ are less than $5\eta^{\frac{1}{3}}n$.
		\label{s32}
		
		\item 
		Let
		\[
		Z=
		\begin{cases}
			N_{G^s}(v_1), & \text{if } d_{G^s}(v_1)< \eta^{\frac{1}{7}}n,\\
			\emptyset, & \text{otherwise.}
		\end{cases}
		\]
		
		Every vertex $v\in V(G)\setminus ( \{v_1,v_2\} \cup Z)$ is incident in $G$
		with fewer than
		$|\pbar_{\rm in}(v_1)|+|\pbar_{\rm in}(v_2)|+|\pbar_{\rm in}(v)|+5\eta^{\frac{1}{3}}n$
		colored edges of $H-E(G_{A,B})$; every vertex $v\in \{v_1,v_2\} \cup Z$ is
		incident in $G$ with at most
		$|\pbar_{\rm in}(v_1)|+|\pbar_{\rm in}(v_2)|+|\pbar_{\rm in}(v)|$ colored edges of
		$H-E(G_{A,B})$. \label{s33}
	\end{enumerate}
	
	To ensure that Condition~(S3.\ref{s32}) is satisfied, we say that an edge
	$e=uv\in E(G_{A,B})$ is \emph{good} if $e\notin E(R_A\cup R_B)$ and the
	degrees of both $u$ and $v$ in $R_A\cup R_B$ are less than
	$5\eta^{\frac{1}{3}}n-1$. Equivalently, when $uv\in E(G_A)$, the relevant
	degrees are those in $R_A$, while the degrees of $u$ and $v$ in $R_B$ are
	zero; the analogous statement holds when $uv\in E(G_B)$. Thus a good edge
	can be added to $R_A$ or $R_B$ without violating (S3.\ref{s32}).
	
	We call a pair of distinct vertices $(a,b)$ a \emph{missing-common-color
		pair}, or an \emph{MCC-pair} for short, with respect to a color $i$ if $i$
	is missing at both $a$ and $b$ with respect to the current coloring. Since
	$|\pbar^{-1}(i)\cap A|-|\pbar^{-1}(i)\cap B|$ is even by the Parity Lemma
	(Lemma~\ref{lem:parity-lemma}), for each color $i\in [k]$, we can first pair
	vertices from $\pbar^{-1}(i)\cap A$ with vertices from $\pbar^{-1}(i)\cap B$,
	and then pair the remaining unpaired vertices inside $\pbar^{-1}(i)\cap A$ or
	inside $\pbar^{-1}(i)\cap B$. Thus we can form in total $|\pbar^{-1}(i)|/2$
	MCC-pairs with respect to $i$.
	
	For every MCC-pair $(a,b)$ with respect to a color $i\in [k]$, we will
	exchange colors on an alternating path $P$ from $a$ to $b$ with at most
	$2/\gamma$ edges, where $P$ starts and ends with uncolored edges of
	$H-E(G_{A,B})$ and alternates between uncolored edges and good edges colored
	by $i$. After the exchange on $P$, the vertices $a$ and $b$ will each be
	incident with an edge colored by $i$, while the good edges of $P$ that were
	colored by $i$ become uncolored and are added to $R_A\cup R_B$. In
	particular, at most $2/\gamma$ edges are added to each of $R_A$ and $R_B$ in
	this exchange.  Thus each exchange eliminates one MCC-pair for color $i$, preserves the properness of the coloring, and uncolors only good edges of $G_A\cup G_B$, which are then recorded in $R_A\cup R_B$.
	Before proving the existence of such paths, we verify that
	Conditions~(S3.\ref{s31}), (S3.\ref{s32}), and~(S3.\ref{s33}) hold at the end of
	Step~3.

For (S3.\ref{s31}): The Initial Stage uncolors at most two edges in $G_B$
for each color initially missing at $v_1$, and at most one edge in $G_A$ for
each color which is still missing at $v_2$. Hence it contributes fewer than
$150\eta n$ edges to $R_A\cup R_B$, and fewer than $100\eta n$ edges to each
of $R_A$ and $R_B$. The only time two edges of $G_B$ may be uncolored for a
color missing at $v_1$ is in the special recoloring used to avoid uncoloring
an edge between $v_2$ and the largest $v_1$-bundle outside $v_2$.

Moreover, each ordinary step of the Initial Stage only transfers a missing
color at $v_1$ or $v_2$ to one endpoint of an uncolored edge, and so does not
increase the total number of missing-color incidences. In the special
recoloring step, if the auxiliary color $\alpha$ is missing at both $v_2$ and
$w^*$, or is present at exactly one of them, then the total number of
missing-color incidences again does not increase. The only possible increase
occurs when $\alpha$ is present at both $v_2$ and $w^*$; in this case, as
explained in the Initial Stage, which is provided shortly after, the total number of missing-color incidences
increases by at most two. Since this special recoloring is used at most
$|\overline{\varphi}_{\rm in}(v_1)|$ times, the Initial Stage increases the
total number of missing-color incidences by at most
$2|\overline{\varphi}_{\rm in}(v_1)|<100\eta n$.

Therefore, after the Initial Stage, the total number of missing colors over
all vertices in $A\cup B$ is still less than $50\eta n^2$. Hence there are
fewer than $25\eta n^2$ MCC-pairs. For each MCC-pair, at most $2/\gamma$
edges are added to each of $R_A$ and $R_B$ when we exchange colors on the
corresponding alternating path. Hence the number of newly uncolored edges in
each of $R_A$ and $R_B$ is always less than
\[
100\eta n+50\frac{\eta}{\gamma}n^2<\eta^{\frac 12}n^2.
\]

It remains to justify $e(R_A)=e(R_B)$. Upon completion of Step~3, each color
class is a $1$-factor of $G$. Since $|A|=|B|$, every $1$-factor contains the
same number of edges inside $A$ as inside $B$. Therefore the total number of
colored edges in $G_A$ equals the total number of colored edges in $G_B$ at
the end of Step~3. Since $G$ is regular and the balanced bipartition $(A,B)$
gives $e(G_A)=e(G_B)$, the number of uncolored edges left in $G_A$ equals the
number left in $G_B$. These are precisely $e(R_A)$ and $e(R_B)$, so
(S3.\ref{s31}) follows.

	For (S3.\ref{s32}): This holds throughout, since we add an edge to $R_A$ or
	$R_B$ only when that edge is good, and by definition adding a good edge cannot
	create a vertex of degree at least $5\eta^{\frac{1}{3}}n$ in either $R_A$ or
	$R_B$.

For (S3.\ref{s33}): During Step~3, the number of newly colored edges of
$H-E(G_{A,B})$ incident with a vertex
$u\in V(G)\setminus(\{v_1,v_2\}\cup Z)$ is bounded by the number of
exchanged alternating paths containing $u$. The number of such paths in which
$u$ is used as a starting vertex is at most
$|\pbar_{\rm in}(v_1)|+|\pbar_{\rm in}(v_2)|+|\pbar_{\rm in}(u)|$. If $u$
is contained in an alternating path but is not used as a starting vertex, then
the path uses a good edge incident with $u$, and this edge is subsequently
added to $R_A\cup R_B$. Hence the number of such paths is at most
$d_{R_A\cup R_B}(u)$, which is less than $5\eta^{\frac 13}n$ by (S3.\ref{s32}).
Therefore the number of edges of $H-E(G_{A,B})$ colored in Step~3 and incident
with $u$ is less than
\[
|\pbar_{\rm in}(v_1)|+|\pbar_{\rm in}(v_2)|+|\pbar_{\rm in}(u)|
+5\eta^{\frac 13 }n.
\]

Now let $u\in \{v_1,v_2\}\cup Z$. Such a vertex is used only as a starting
vertex of alternating paths. The number of such paths is at most
$|\pbar_{\rm in}(v_1)|+|\pbar_{\rm in}(v_2)|+|\pbar_{\rm in}(u)|$.
Thus Condition~(S3.\ref{s33}) is satisfied at the end of Step~3.

	\medskip

	\noindent\textbf{Initial Stage}. \quad
	
	Since  $d^s_G(v_1)=2$ is possible, we mainly deal with missing colors at $v_1$
	in this stage together with the missing colors at $v_2$.  If such a vertex exists uniquely, let  $w^*\in V(G)\setminus\{v_1,v_2\}$ such that
	$e_G(v_1,w^*)>\eta^{\frac 16}n$ and $e_G(v_2,w^*)\ge 1$, and let $W^*=\{w^*\}$.
	If no such vertex exists, or if there is more than one such vertex,
	then we let  $W^*=\emptyset$. 
	Let $\varphi_0$ denote the coloring at the beginning of the Initial
	Stage. We first deal with the colors missing at $v_1$. For $v_1$ and each
	$i\in \pbar_{\rm in}(v_1)$, choose an uncolored edge
	$e\in E_G(v_1,u)\cap E(H)$ for some $u\in B$. Such an edge exists, since
	$d_{G_{A,B}}(v_1)\le \Delta/2+14\eta n$ by~Claim~\ref{claim:degrees-in-G-A-B}, and so  the number of edges of $H-E(G_{A,B})$ incident with $v_1$ is larger than $|\pbar_{\rm in}(v_1)|$.
	Before treating the color $i$, we have used fewer than $|\pbar_{\rm in}(v_1)|$ such edges. If $i\in \pbar(u)$, then
	we simply color $e$ with $i$. Otherwise, $i$ is present at $u$. Since $i$ is
	missing at $v_1$, the edge colored $i$ and incident with $u$ cannot join $u$
	to $v_1$. Moreover, by the construction of the coloring before this Initial
	Stage, and by the fact that all edges of $H$ newly colored so far are incident
	with $v_1$, this edge must lie inside $B$. Hence there exists $y\in B$ such
	that an edge $f\in E_G(u,y)$ is colored with $i$.

If $y\notin W^*$ or $u\ne v_2$, then we uncolor $f$ and color $e$ with $i$.
In this case, the missing color $i$ is transferred from $v_1$ to $y$, so the
total number of missing-color incidences does not increase.

Thus we may assume that $y\in W^*$ and $u=v_2$. Hence $W^*=\{w^*\}$ and
$f\in E_G(v_2,w^*)$. In this case we avoid uncoloring the edge $f$. We claim
that, throughout the Initial Stage, we can choose distinct colors
\[
\alpha\in [k]\setminus \{\varphi_0(g):g\in E(G_{A,B}[\{v_1,v_2,w^*\}])\}.
\]
Indeed, by~\eqref{eqn:triangle-v1-controlled}, we have
\begin{equation}\label{eqn:alpha-available-initial}
	k-e(G_{A,B}[\{v_1,v_2,w^*\}])
	\ge k-d_{G_{A,B}}(v_1)=|\pbar_{\rm in}(v_1)|.
\end{equation}
The number of times this special recoloring is used is at most
$|\pbar_{\rm in}(v_1)|$. Hence we can choose a color $\alpha$ from the above
set which has not been used before in this special recoloring. Moreover,
$\alpha\ne i$: if $f$ had color $i$ under $\varphi_0$, then $i$ is excluded
from the above set; while if the color $i$ was put on $f$ by an earlier special
recoloring, then $i$ is excluded by the distinctness of the colors chosen for
the special recolorings.

By the choice of $\alpha$, if $\alpha$ is present at $v_2$, then the edge
colored $\alpha$ incident with $v_2$ is not an edge to $v_1$ or to $w^*$; and
if $\alpha$ is present at $w^*$, then the edge colored $\alpha$ incident with
$w^*$ is not an edge to $v_1$ or to $v_2$.

If $\alpha$ is missing at both $v_2$ and $w^*$, then we recolor $f$ by
$\alpha$ and color $e$ by $i$. In this case, the missing color $i$ is
transferred from $v_1$ to $w^*$, and the total number of missing-color
incidences does not increase.

If $\alpha$ is present at $v_2$ but missing at $w^*$, then we uncolor the
edge at $v_2$ colored $\alpha$, recolor $f$ by $\alpha$, and then color $e$
by $i$. The case where $\alpha$ is present at $w^*$ but missing at $v_2$ is
symmetric. In these cases, the color $i$ is transferred from $v_1$ to $w^*$,
and the missing incidence of $\alpha$ is transferred from one endpoint of $f$
to the other endpoint of the uncolored $\alpha$-edge. Thus the total number of
missing-color incidences does not increase.

Lastly, suppose that $\alpha$ is present at both $v_2$ and $w^*$. We uncolor
the two $\alpha$-colored edges incident with $v_2$ and $w^*$, recolor $f$ by
$\alpha$, and color $e$ by $i$. Then the color $i$ is no longer missing at
$v_1$, but becomes missing at $w^*$. The color $\alpha$ remains present at
both $v_2$ and $w^*$ because of the recolored edge $f$, but it becomes missing
at the other endpoints of the two edges originally colored $\alpha$. Thus this
step increases the total number of missing-color incidences by two.

After all colors missing at $v_1$ have been treated, we deal with the colors still missing at $v_2$ with respect to the current coloring. Let $i\in \pbar(v_2)$. We choose an uncolored edge $e\in E_H(v_2,A)$, say $e\in E_H(v_2,u)$, where $u\ne v_1$, as follows. Since we are in Case~2, $e_G(v_1,v_2)<\Delta-\eta^{\frac{1}{7}}n$. Hence $d_G(v_2,V(G)\setminus\{v_1\})>\eta^{\frac 17}n$. By Condition~\eqref{item:thm4.1.c}, $v_2$ has at least $\eta^{\frac 17}n-13\eta n$ simple neighbors outside $v_1$. Using \eqref{eqn:partition-deg}, at least $\eta^{\frac 17 }n/3$ of these neighbors lie in $A$. After excluding the vertices incident with edges of $H$ already colored in the Initial Stage, there remains an available uncolored edge $v_2u\in E_H(v_2,A)$ with $u\ne v_1$ such that $u$
is a simple neighbor of $v_2$. 

If $i\in \pbar(u)$, then color $v_2u$ with $i$. Thus we may assume that $i$ is present at $u$. Let $f$ be the edge colored $i$ incident with $u$. If $f$ is not incident with $v_1$, then uncolor $f$ and color $v_2u$ with $i$. It remains to consider the case where $f \in E_G(v_1,u)$. In this case the color $i$ is already present at $v_1$. By the preceding availability estimate, choose another uncolored edge $v_2u'\in E_H(v_2,A)$, with $u'\ne v_1,u$, avoiding all previously used choices. Since $i$ is already present at $v_1$ on the edge $f$, the edge colored $i$ incident with $u'$ cannot be from $E_G(u',v_1)$. Therefore, if $i\notin \pbar(u')$, the edge colored $i$ incident with $u'$ is not incident with $v_1$; we uncolor that edge and color $v_2u'$ with $i$. If $i\in \pbar(u')$, we simply color $v_2u'$ with $i$. Thus, in all cases, no edge incident with $v_1$ is uncolored while treating a color missing at $v_2$.
	
	After this Initial Stage, every color in $[k]$ is present at both $v_1$ and
	$v_2$, and $R_A$ contains no edge that is incident with $v_1$. 
	\medskip
	
	\noindent\textbf{Second Stage}. 
	
	We now show the existence of alternating paths for the current MCC-pairs.
	
	Fix a color $i\in [k]$, and let  $W_i$ be the set of vertices $x\in V(G)$ such that at least one of the
	following holds:
	\begin{itemize}
		\item $x$ is incident with a non-good edge colored $i$;
		\item $x$ is missing the color $i$;
		\item $x\in Z \cup \{v_1,v_2\}$;
		\item $x$ is incident with a good edge colored $i$ whose other endpoint lies
		in $Z \cup\{v_1,v_2\}$.
	\end{itemize}
The last item
	ensures that $W_i$ is closed under the good color-$i$ matching on the
	vertices in $Z \cup \{v_1, v_2\}$. Consequently, if a vertex outside $W_i$ is incident with a
	good edge colored $i$, then the other endpoint of this good edge is also
	outside $W_i$.
	
	By (S3.\ref{s31}), there are fewer than $\eta^{\frac{1}{2}}n^2$ edges in
	each of $R_A$ and $R_B$, so in each part there are fewer than
	$\frac{2\eta^{\frac{1}{2}}n^2}{5\eta^{\frac{1}{3}}n-1}<\eta^{\frac{1}{6}}n$
	vertices of degree at least $5\eta^{\frac{1}{3}}n-1$ in the corresponding
	auxiliary graph. Since each non-good edge colored $i$ is accounted for
	through one or two such vertices, there are fewer than $2\eta^{\frac{1}{6}}n$
	vertices incident with a non-good edge colored $i$. Moreover,
	by~\eqref{eqn:upper-bound-vertices-missing-i}, the number of vertices missing
	the color $i$ is at most $\frac{100\eta n}{\alpha}$. Since
	$|Z|<\eta^{\frac{1}{7}}n$, we have
	\begin{equation}\label{eqn:Wi-small}
		|W_i|
		\le 2\eta^{\frac{1}{6}}n+\frac{100\eta n}{\alpha}+4+2\eta^{\frac{1}{7}}n
		<3\eta^{\frac{1}{7}}n.
	\end{equation}
	
	For $v\in C$, where $C\in\{A,B\}$, let $D$ denote the other part. Define
	\[
	S_0^i(v)=\{u\in D:  \text{$vu$ is an uncolored edge}\}\setminus W_i.
	\]
	For $j\ge 0$, after $S_j^i(v)$ has been defined, let $T_j^i(v)$ be the set
	of vertices joined to vertices of $S_j^i(v)$ by a good edge colored $i$.
	Since $W_i$ is closed under the good color-$i$ matching, every vertex of
	$T_j^i(v)$ also avoids $W_i$. In particular, each vertex of $S_j^i(v)$ is
	matched by a unique good edge colored $i$ to a vertex outside $W_i$, and
	hence $|T_j^i(v)|=|S_j^i(v)|$. For $j\ge 1$, define
	\[
	S_j^i(v)=RN_{\gamma,H}(T_{j-1}^i(v))\setminus
	\big(\{u\in V(H):  \text{$vu$ was already colored}\} \cup W_i\big).
	\]
	Figure~\ref{fig:def-S_j-T_j} demonstrates the definition of these sets
	$S_j^i(v)$ and $T_j^i(v)$.
	
	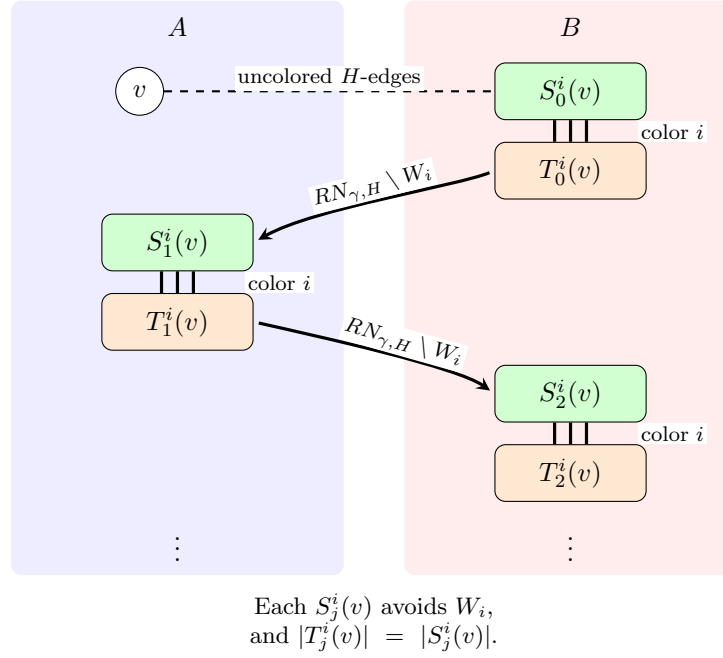
\begin{figure}[ht]
		\centering
		\begin{tikzpicture}[
			x=1cm,y=1cm,
			>=stealth,
			font=\small,
			vtx/.style={circle, draw, fill=white, inner sep=1.5pt, minimum size=18pt},
			sset/.style={draw, rounded corners, fill=green!18, minimum width=2.0cm, minimum height=0.75cm},
			tset/.style={draw, rounded corners, fill=orange!18, minimum width=2.0cm, minimum height=0.75cm},
			uncol/.style={dashed, thick},
			coli/.style={very thick},
			rn/.style={->, very thick},
			lab/.style={font=\scriptsize, fill=white, inner sep=1pt}
			]
			
			\fill[blue!7, rounded corners] (-2.2,-4.4) rectangle (2.2,3.2);
			\fill[red!7, rounded corners]  (3.0,-4.4) rectangle (7.4,3.2);
			
			\node at (0,2.85) {$A$};
			\node at (5.2,2.85) {$B$};
			
			\node[vtx] (v) at (-0.5,2.0) {$v$};
			
			\node[sset] (S0) at (5.2,2.0) {$S_0^i(v)$};
			\node[tset] (T0) at (5.2,0.95) {$T_0^i(v)$};
			
			\node[sset] (S1) at (0,0.0) {$S_1^i(v)$};
			\node[tset] (T1) at (0,-1.05) {$T_1^i(v)$};
			
			\node[sset] (S2) at (5.2,-2.0) {$S_2^i(v)$};
			\node[tset] (T2) at (5.2,-3.05) {$T_2^i(v)$};
			
			\node at (0,-3.95) {$\vdots$};
			\node at (5.2,-3.95) {$\vdots$};
			
			\draw[uncol] (v.east) -- node[lab,above] {uncolored $H$-edges} (S0.west);
			
			\draw[coli] ([xshift=-6pt]S0.south) -- ([xshift=-6pt]T0.north);
			\draw[coli] (S0.south) -- (T0.north);
			\draw[coli] ([xshift=6pt]S0.south) -- ([xshift=6pt]T0.north);
			\node[lab] at (6.55,1.47) {color $i$};
			
			\draw[rn, shorten <=2pt, shorten >=2pt]
			(T0.west) .. controls (3.7,0.75) and (1.8,0.45) ..
			node[lab,above,sloped,pos=0.52] {$RN_{\gamma,H}\setminus W_i$}
			(S1.east);
			
			\draw[coli] ([xshift=-6pt]S1.south) -- ([xshift=-6pt]T1.north);
			\draw[coli] (S1.south) -- (T1.north);
			\draw[coli] ([xshift=6pt]S1.south) -- ([xshift=6pt]T1.north);
			\node[lab] at (1.35,-0.53) {color $i$};
			
			\draw[rn, shorten <=2pt, shorten >=2pt]
			(T1.east) .. controls (1.9,-1.25) and (3.7,-1.65) ..
			node[lab,above,sloped,pos=0.55] {$RN_{\gamma,H}\setminus W_i$}
			(S2.west);
			
			\draw[coli] ([xshift=-6pt]S2.south) -- ([xshift=-6pt]T2.north);
			\draw[coli] (S2.south) -- (T2.north);
			\draw[coli] ([xshift=6pt]S2.south) -- ([xshift=6pt]T2.north);
			\node[lab] at (6.55,-2.53) {color $i$};
			
			\node[align=center, font=\footnotesize, text width=5.0cm] at (2.6,-5.02)
			{Each $S_j^i(v)$ avoids $W_i$, and $|T_j^i(v)|=|S_j^i(v)|$.};
			
		\end{tikzpicture}
		\caption{Illustration of the construction of $S_j^i(v)$ and $T_j^i(v)$,
			assuming $v\in A$. Dashed cross-part edges are uncolored edges of $H$, the
			bundles of solid edges represent good edges of color $i$, and the curved
			arrows indicate the robust-neighborhood step with the forbidden set $W_i$
			removed.}
		\label{fig:def-S_j-T_j}
	\end{figure}
	
	We record a lower bound on $|S_0^i(v)|$. Let $v$ be an endpoint of a current
	MCC-pair. By the Initial Stage, $v\notin\{v_1,v_2\}$. Hence, by
	\eqref{item:thm4.1.b},
	$d^s_G(v)\ge \frac{\Delta}{2}-6\eta n$.
	Using~\eqref{eqn:partition-deg}, the number of simple neighbors of $v$ in the
	opposite part is at least
	$\frac{1}{2}\left(\frac{\Delta}{2}-6\eta n-n^{\frac{2}{3}}-\eta n\right)>\frac{1}{5}\alpha n$.
	Moreover, at most
	$|\pbar_{\rm in}(v_1)|+|\pbar_{\rm in}(v_2)|+|\pbar_{\rm in}(v)|+5\eta^{\frac{1}{3}}n$
	edges of $H-E(G_{A,B})$ incident with $v$ have already been colored in
	Step~3, and by~\eqref{eqn:Wi-small} we exclude at most $3\eta^{\frac{1}{7}}n$
	vertices from $W_i$. Therefore
	\[
	|S_0^i(v)|
	>
	\frac{1}{5}\alpha n
	-\left(150\eta n+5\eta^{\frac{1}{3}}n\right)
	-3\eta^{\frac{1}{7}}n
	> \frac{1}{6}\alpha n.
	\]
	In particular, $|S_0^i(v)|>\frac{1}{6}\alpha n>4\tau |A|$.
	
	Since the partition is balanced, write $m=|A|=|B|$. Grow the sets $S_j^i(v)$
	using robust expansion. Suppose that $|S_j^i(v)|\le (1-4\tau)m$. Since
	$T_j^i(v)$ lies in the part opposite to $S_j^i(v)$ and
	$|T_j^i(v)|=|S_j^i(v)|$, the bipartite robust $(\gamma,4\tau)$-expansion of
	$H$ gives $|RN_{\gamma,H}(T_j^i(v))|>|S_j^i(v)|+\gamma m$. When forming
	$S_{j+1}^i(v)$, we delete the vertices in $W_i$ and the vertices excluded
	because of previously colored edges of $H$. By~\eqref{eqn:Wi-small} and the
	bounds from Step~3, these deletions have size less than $\frac{1}{2}\gamma m$.
	Hence $|S_{j+1}^i(v)|>|S_j^i(v)|+\frac{1}{2}\gamma m$. Let $\ell$ be the
	smallest integer such that $|S_\ell^i(v)|>(1-4\tau)m$. Since the size
	increases by at least $\frac{1}{2}\gamma m$ at each step, $\ell<2/\gamma$.
	
	Let $(v,u)$ be a current MCC-pair. Let $C$ be the part containing $v$, and
	let $D$ be the other part. Suppose first that $S_\ell^i(v)$ and $u$ lie in
	opposite parts. Since $|S_\ell^i(v)|>(1-4\tau)m$ and
	$|T_0^i(u)|=|S_0^i(u)|>\frac{1}{6}\alpha n$, we have
	$|T_0^i(u)\cap S_\ell^i(v)|>\frac{1}{7}\alpha n$. Choose
	$v_\ell^1\in T_0^i(u)\cap S_\ell^i(v)$ and choose
	$u_0^1\in S_0^i(u)$ such that $v_\ell^1u_0^1$ is colored $i$.
	
	Since $v_\ell^1\in S_\ell^i(v)\subseteq RN_{\gamma,H}(T_{\ell-1}^i(v))$,
	the vertex $v_\ell^1$ has at least $\gamma m-|W_i|>2/\gamma$ neighbors in
	$T_{\ell-1}^i(v)$ available for the construction. Choose
	$v_{\ell-1}^2\in T_{\ell-1}^i(v)\cap N_G(v_\ell^1)$ not selected before,
	and then choose $v_{\ell-1}^1\in S_{\ell-1}^i(v)$ such that
	$v_{\ell-1}^1v_{\ell-1}^2$ is colored $i$. Repeating backwards, for each
	$j \in \{\ell-2,\ldots,0\}$, choose
	$v_j^2\in T_j^i(v)\cap N_G(v_{j+1}^1)$ not selected before, and then choose
	$v_j^1\in S_j^i(v)$ such that $v_j^1v_j^2$ is colored $i$. It follows that
	\[
	vv_0^1v_0^2v_1^1v_1^2\cdots v_{\ell-1}^1v_{\ell-1}^2v_\ell^1u_0^1u
	\]
	is a $(v,u)$-alternating path. See Figure~\ref{fig:alternating-path} Case~(a)
	for the case $\ell=4$, where $S_4^i(v)\subseteq D$ and $u\in C$.
	
	Now suppose that $S_\ell^i(v)$ and $u$ lie in the same part. Use one more
	step from $u$. By the same argument,
	$|T_1^i(u)\cap S_\ell^i(v)|>\frac{1}{7}\alpha n$. Choose
	$v_\ell^1\in T_1^i(u)\cap S_\ell^i(v)$ and choose
	$u_1^1\in S_1^i(u)$ such that $u_1^1v_\ell^1$ is colored $i$. Then choose
	$u_0^2\in T_0^i(u)\cap N_G(u_1^1)$ not selected before, and choose
	$u_0^1\in S_0^i(u)$ such that $u_0^1u_0^2$ is colored $i$. Combining this
	with the backwards construction from $v_\ell^1$ to $v$, we obtain the
	$(v,u)$-alternating path
	\[
	vv_0^1v_0^2v_1^1v_1^2\cdots v_{\ell-1}^1v_{\ell-1}^2v_\ell^1u_1^1u_0^2u_0^1u.
	\]
	See Figure~\ref{fig:alternating-path} Case~(b) for the case $\ell=4$, where
	$S_4^i(v)\subseteq D$ and $u\in D$.
	
	\begin{figure}[ht]
		\centering
		\captionsetup[subfigure]{skip=8pt}
		
		\begin{subfigure}[t]{0.95\textwidth}
			\centering
			\begin{tikzpicture}[
				x=1cm,y=1cm,
				font=\small,
				>=stealth,
				partA/.style={draw, rounded corners, fill=blue!7, minimum width=12.2cm, minimum height=1.25cm},
				partB/.style={draw, rounded corners, fill=red!7, minimum width=12.2cm, minimum height=1.25cm},
				vtx/.style={circle, draw, fill=white, inner sep=1.2pt, minimum size=18pt},
				meet/.style={circle, draw, fill=yellow!22, inner sep=1.2pt, minimum size=18pt},
				uncol/.style={dashed, thick},
				coli/.style={very thick}
				]
				
				\node[partA] at (6.1,2.1) {};
				\node[partB] at (6.1,0.1) {};
				
				\node at (0.6,2.1) {$C$};
				\node at (0.6,0.1) {$D$};
				
				\node[vtx] (v)   at (1.3,2.1) {$v$};
				\node[vtx] (v11) at (4.3,2.1) {$v_1^1$};
				\node[vtx] (v12) at (5.4,2.1) {$v_1^2$};
				\node[vtx] (v31) at (7.5,2.1) {$v_3^1$};
				\node[vtx] (v32) at (8.6,2.1) {$v_3^2$};
				\node[vtx] (u)   at (11.0,2.1) {$u$};
				
				\node[vtx]  (v01) at (2.4,0.1) {$v_0^1$};
				\node[vtx]  (v02) at (3.5,0.1) {$v_0^2$};
				\node[vtx]  (v21) at (5.9,0.1) {$v_2^1$};
				\node[vtx]  (v22) at (7.0,0.1) {$v_2^2$};
				\node[meet] (v41) at (9.1,0.1) {$v_4^1$};
				\node[vtx]  (u01) at (10.2,0.1) {$u_0^1$};
				
				\draw[uncol] (v) -- (v01);
				\draw[coli]  (v01) -- (v02);
				\draw[uncol] (v02) -- (v11);
				\draw[coli]  (v11) -- (v12);
				\draw[uncol] (v12) -- (v21);
				\draw[coli]  (v21) -- (v22);
				\draw[uncol] (v22) -- (v31);
				\draw[coli]  (v31) -- (v32);
				\draw[uncol] (v32) -- (v41);
				\draw[coli]  (v41) -- (u01);
				\draw[uncol] (u01) -- (u);
				
				\node[font=\scriptsize] at (2.95,-0.22) {$i$};
				\node[font=\scriptsize] at (4.85,2.42) {$i$};
				\node[font=\scriptsize] at (6.45,-0.22) {$i$};
				\node[font=\scriptsize] at (8.05,2.42) {$i$};
				\node[font=\scriptsize] at (9.65,-0.22) {$i$};
				
				\node[align=center] at (6.1,-1.25)
				{Case (a): $u\in C$. Then $S_4^i(v)\subseteq D$, and the path ends through $T_0^i(u)$,\\
					so $v_4^1\in T_0^i(u)\cap S_4^i(v)$.};
				
			\end{tikzpicture}
		\end{subfigure}
		
		\vspace{1.2em}
		
		\begin{subfigure}[t]{0.95\textwidth}
			\centering
			\begin{tikzpicture}[
				x=1cm,y=1cm,
				font=\small,
				>=stealth,
				partA/.style={draw, rounded corners, fill=blue!7, minimum width=12.2cm, minimum height=1.25cm},
				partB/.style={draw, rounded corners, fill=red!7, minimum width=12.2cm, minimum height=1.25cm},
				vtx/.style={circle, draw, fill=white, inner sep=1.2pt, minimum size=18pt},
				meet/.style={circle, draw, fill=yellow!22, inner sep=1.2pt, minimum size=18pt},
				uncol/.style={dashed, thick},
				coli/.style={very thick}
				]
				
				\node[partA] at (6.1,2.1) {};
				\node[partB] at (6.1,0.1) {};
				
				\node at (0.6,2.1) {$C$};
				\node at (0.6,0.1) {$D$};
				
				\node[vtx] (v)   at (1.3,2.1) {$v$};
				\node[vtx] (v11) at (4.3,2.1) {$v_1^1$};
				\node[vtx] (v12) at (5.4,2.1) {$v_1^2$};
				\node[vtx] (v31) at (7.1,2.1) {$v_3^1$};
				\node[vtx] (v32) at (8.2,2.1) {$v_3^2$};
				\node[vtx] (u02) at (9.8,2.1) {$u_0^2$};
				\node[vtx] (u01) at (10.9,2.1) {$u_0^1$};
				
				\node[vtx]  (v01) at (2.4,0.1) {$v_0^1$};
				\node[vtx]  (v02) at (3.5,0.1) {$v_0^2$};
				\node[vtx]  (v21) at (5.9,0.1) {$v_2^1$};
				\node[vtx]  (v22) at (7.0,0.1) {$v_2^2$};
				\node[meet] (v41) at (8.8,0.1) {$v_4^1$};
				\node[vtx]  (u11) at (9.9,0.1) {$u_1^1$};
				\node[vtx]  (u)   at (11.1,0.1) {$u$};
				
				\draw[uncol] (v) -- (v01);
				\draw[coli]  (v01) -- (v02);
				\draw[uncol] (v02) -- (v11);
				\draw[coli]  (v11) -- (v12);
				\draw[uncol] (v12) -- (v21);
				\draw[coli]  (v21) -- (v22);
				\draw[uncol] (v22) -- (v31);
				\draw[coli]  (v31) -- (v32);
				\draw[uncol] (v32) -- (v41);
				\draw[coli]  (v41) -- (u11);
				\draw[uncol] (u11) -- (u02);
				\draw[coli]  (u02) -- (u01);
				\draw[uncol] (u01) -- (u);
				
				\node[font=\scriptsize] at (2.95,-0.22) {$i$};
				\node[font=\scriptsize] at (4.85,2.42) {$i$};
				\node[font=\scriptsize] at (6.45,-0.22) {$i$};
				\node[font=\scriptsize] at (7.75,2.42) {$i$};
				\node[font=\scriptsize] at (9.35,-0.22) {$i$};
				\node[font=\scriptsize] at (10.35,2.42) {$i$};
				
				\node[align=center] at (6.1,-1.25)
				{Case (b): $u\in D$. Then $S_4^i(v)\subseteq D$, and we use one extra step from $u$,\\
					so $v_4^1\in T_1^i(u)\cap S_4^i(v)$.};
				
			\end{tikzpicture}
		\end{subfigure}
		
		\caption{Illustration of the two alternating-path constructions for $\ell=4$.
			Here $\{C,D\}=\{A,B\}$ denotes the bipartition. Dashed edges are uncolored
			edges of $H$, and solid edges are edges colored $i$.}
		\label{fig:alternating-path}
	\end{figure}
	
	In both constructions, no internal vertex of the path is in $\{v_1,v_2\}$.
	Moreover, in the case $d_{G^s}(v_1)<\eta^{\frac{1}{7}}n$, 
	no internal vertex
	lies in $Z$.
	
	We exchange along this alternating path $P$ by coloring the edges of
	$E(P)\cap E(H)$ with color $i$ and uncoloring the edges of
	$E(P)\cap E(G_A\cup G_B)$. After the exchange, the color $i$ appears on
	edges incident with both $u$ and $v$. The newly uncolored edges of $G_A$ are
	added to $R_A$, and the newly uncolored edges of $G_B$ are added to $R_B$.
	
	By finding such paths for all MCC-pairs with respect to color $i$, we extend
	color class $i$ to a 1-factor of $G$. Repeating for all $i\in[k]$ completes
	Step~3.
	
	\begin{center}
		{\bf Step 4: Edge-color $R_A$ and $R_B$ and extend the new color classes
			into 1-factors}
	\end{center}
	
Each of the color classes for the colors in $[k]$ is now a 1-factor of $G$.
By Conditions~(S3.\ref{s31}) and~(S3.\ref{s32}), each of $R_A$ and $R_B$ has
fewer than $\eta^{\frac12}n^2$ edges and
$\Delta(R_A),\Delta(R_B)<5\eta^{\frac13}n$.

We claim that $v_1\notin V(R_A)$. In the Initial Stage of Step~3, no edge
incident with $v_1$ was added to $R_A$; and throughout the Second Stage of
Step~3 the vertex $v_1$ was avoided, so no edge of $G_A$ incident with $v_1$
was placed in $R_A$ either. Hence $R_A$ contains no edge incident with $v_1$.
Since $v_1\notin V(R_A)$ and $v_1\in A$, no edge of $R_A\cup R_B$ is incident
with $v_1$, so $R_A\cup R_B$ is a subgraph of $G-v_1$. As $v_1$ is the only
vertex at which $G$ may have high multiplicity,
Condition~\eqref{item:thm4.1.c} gives
$\mu(R_A\cup R_B)\le\mu(G-v_1)\le 13\eta n$.

We color the edges of $R_A\cup R_B$ in two stages, and then extend each new
color class to a 1-factor by adding a matching in $H$.
	
	\medskip
	\noindent{\bf Stage I: Coloring the selected Initial Stage edges.}
	
	Let $R_B^0$ be the subgraph of $R_B$ consisting of the edges that were added
	to $R_B$ in the Initial Stage of Step~3 and are incident with a vertex of
	$B\cap Z$ or with $v_2$.  Let $\ell_1=e(R_B^0)$. Every color initially
	missing at $v_1$ contributes at most two edges to $R_B^0$, while the colors
	treated at $v_2$ contribute none to $R_B^0$. Hence $\ell_1 <100\eta n$. If
	$\ell_1\ge 1$,
	then we color all edges of $R_B^0$ with distinct colors from
	$[k+1,k+\ell_1]$. For these same colors, we choose $\ell_1$ distinct edges
	from $R_A$ and color them distinctly, one edge for each color in
	$[k+1,k+\ell_1]$. 
	
	The alternating-path part of Step~3 avoids all vertices of
	$Z \cup \{v_1, v_2\}$ through the forbidden set $W_i$. Moreover, every
	Initial Stage edge of $R_B$ incident with $B\cap Z$ or with $v_2$ has just
	been placed into $R_B^0$ and colored in Stage~I. Thus, after Stage~I, no
	remaining edge of $R_B$ which was created in the Initial Stage is incident with
	$B\cap Z$ or with $v_2$. 
	
	After this, delete the colored edges from $R_A$ and $R_B$, and still denote
	the remaining uncolored subgraphs by $R_A$ and $R_B$. Since the same number
	of edges was colored in $R_A$ as in $R_B$, we still have $e(R_A)=e(R_B)$.
	Moreover, the maximum-degree bound and the multiplicity bound
	$\mu(R_A\cup R_B)\le 13\eta n$ remain valid.
	
	\medskip
	\noindent{\bf Stage II: Coloring the remaining edges of $R_A\cup R_B$.}
	
	Set $\ell_2=\lceil 5 \eta^{\frac13}n+13\eta n \rceil$. By
	Theorem~\ref{thm:chromatic-index} and Lemma~\ref{lem:equa-edge-coloring},
	we edge-color $R_A$ and $R_B$ separately with the common color set
	$[k+\ell_1+1,k+\ell_1+\ell_2]$, and choose both colorings equalized. Since
	$e(R_A)=e(R_B)$, the two equalized colorings have the same multiset of color
	class sizes. Thus, after renaming the colors in one of the two colorings if
	necessary, each color in $[k+\ell_1+1,k+\ell_1+\ell_2]$ appears on the same
	number of edges in $R_A$ as in $R_B$. Each such color appears on at most
	\[
	\left\lceil\frac{e(R_A)}{\ell_2}\right\rceil
	\le \frac{\eta^{\frac12}n^2}{5\eta^{\frac13}n+13\eta n}+1
	<\eta^{\frac16}n
	\]
	edges in each of $R_A$ and $R_B$.
	
	\medskip
	\noindent{\bf Extending the new color classes.}
	
The extensions of the Stage I and Stage II colorings are discussed together below, since they share some common arguments.

We process the colors $i\in [k+1,k+\ell_1+\ell_2]$ in increasing order. Suppose that, for every $j\in [k+1,i-1]$, the color class $j$ has already been extended to a 1-factor. Let $A_i$ and $B_i$ be the sets
	of vertices in $A$ and $B$, respectively, that are incident with an edge
	colored $i$ in $R_A\cup R_B$. If $i\in [k+1,k+\ell_1]$, then
	$|A_i|=|B_i|=2$. If $i\in [k+\ell_1+1,k+\ell_1+\ell_2]$, then
	$|A_i|=|B_i|\le 2\eta^{\frac16}n$. Let $H_i$ be the subgraph of $H$
	obtained by deleting the vertices in $A_i\cup B_i$ and removing all edges of
	$H$ that have already been colored.
	
	Initialize $Q_i=\emptyset$. We will build a small prescribed matching $Q_i$ to
	handle $v_1$ and $v_2$ before finding a matching saturating  the remaining vertices of $H_i$.
	
	We first deal with $v_1$.   As argued at the beginning
	of Step~4, $v_1\notin V(R_A)$, so $v_1$ is incident with no colored edge of
	$R_A\cup R_B$; hence $v_1\notin A_i$, and $v_1$ has not yet been covered by
	$Q_i$.

	First suppose that $d_{G^s}(v_1)\ge \eta^{\frac17}n$. By
	\eqref{eqn:partition-deg},
	$d_{G^s}(v_1,B)\ge \frac{1}{2}\left(d_{G^s}(v_1)-n^{\frac23}-\eta n\right)>\frac{1}{3}\eta^{\frac17}n$.
	Thus, if $i\in [k+\ell_1+1,k+\ell_1+\ell_2]$, then
	$|B_i|\le 2\eta^{\frac16}n$, and if $i\in [k+1,k+\ell_1]$, then
	$|B_i|=2$. In either case, after also excluding $v_2$ and the vertices
	corresponding to already colored edges of $H$ incident with $v_1$, there
	remains a  neighbor
	$u\in (N_{G^s}(v_1)\cap B)\setminus (B_i\cup\{v_2\})$ such that some edge in
	$E_G(v_1,u)$ is still uncolored. Indeed, at this point at most
	$150\eta n+\ell_1+\ell_2$  neighbors of $v_1$ in $B$ can be blocked by
	already colored edges of $H$, and in either case $|B_i|\le 2\eta^{\frac16}n$.
	Hence the number of available  neighbors is positive.  Choose one such edge and add it to $Q_i$.

Then suppose that $d_{G^s}(v_1)< \eta^{\frac17}n$.
	If $i\in [k+1,k+\ell_1]$, let $e \in E_G(u,w)$  for some $u,w\in B$ be the edge of $R_B^0$ colored $i$ in
	Stage~I.   If there is a unique vertex 
	$w'\in V(G)\setminus\{v_1,v_2\}$ such that
	$e_G(v_1,w')>\eta^{\frac 16}n$, 
	then either  $e_G(v_2, w')=0$  or the construction in the Initial Stage of Step 3 indicates
	that $R_B^0$ contains no edge with   endvertices  $v_2$ and 
	$w'$ (so $w'=w^*$ when $e_G(v_2, w') \ge 1$). 
	 Hence the edge $e$ cannot have both
	$v_2$ and $w'$ as its endvertices. Choose
	$x\in \{v_2,w'\}\setminus \{u,w\}$.

	We claim that there is an available edge from $v_1$ to $x$ in $H_i$. If
	$x=w'$, then $e_G(v_1,x)>\eta^{\frac16}n$. If $x=v_2$, then by
	Conditions~\eqref{item:thm4.1.a-new} and~\eqref{item:thm4.1.c}, together with
	Lemma~\ref{lem:v1-v2-edges}, we have
	$e_G(v_1,x)=e_G(v_1,v_2)\ge \eta^{\frac16}n-7\eta n$.
	
	By the construction of the middle edges, for every $x\in B$,
	\[
	e_{H-E(G_{A,B})}(v_1,x)
	=
	e_G(v_1,x)-e_{G[M]}(v_1,x)
	\ge
	\left\lfloor \frac{e_G(v_1,x)}{2}\right\rfloor-7\eta n.
	\]
	Therefore, in either case,
	\[
	e_{H-E(G_{A,B})}(v_1,x)
	\ge \frac12\eta^{\frac16}n-12\eta n.
	\]
	At this point, fewer than $150\eta n+\ell_1+\ell_2$ edges of $H$ incident with
	$v_1$ have already been colored. Since
	\[
	150\eta n+\ell_1+\ell_2<\frac12\eta^{\frac16}n-12\eta n,
	\]
	we have $e_{H_i}(v_1,x)>0$. Thus we can choose an available edge from $v_1$ to
	$x$ and add it to $Q_i$.
	
	This handles the case where such a unique vertex $w'$ exists. It remains to
	consider the case where either $e_G(v_1,w')\le \eta^{\frac16}n$ for every
	$w'\in V(G)\setminus\{v_1,v_2\}$, or there are at least two distinct vertices
	$w_1,w_2\in V(G)\setminus\{v_1,v_2\}$ such that
	$e_G(v_1,w_j)>\eta^{\frac16}n$ for $j\in[2]$.
In the first case, since $e_G(v_1,v_2) < \Delta -\eta^{\frac 1 7} n$, 
we know that $e_G(v_1, B\setminus \{u,w\}) \ge \eta^{\frac 1 7} n -2\eta^{\frac 16}n$. 
Thus $e_{H_i}(v_1, B\setminus \{u,w\}) \ge \frac{1}{2} (\eta^{\frac 1 7} n -2\eta^{\frac 16}n)-13\eta n -(150\eta n+\ell_1+\ell_2) >0$. 
Thus we can choose an available edge between  $v_1$ 
and a vertex of $B\setminus \{u,w\}$ and add this edge to $Q_i$. 

If there are at least two distinct vertices $w_1, w_2 \in V(G)\setminus\{v_1,v_2\}$
such that $e_G(v_1,w_j) > \eta^{\frac 16}n$ for $j\in [2]$,  then $w_1, w_2\in Z_0\subseteq B$ and 
 we again get $e_G(v_1,v_2) \ge \eta^{\frac 16}n-7\eta n$. 
Now there is an available edge between $v_1$ and 
a vertex of $\{v_2, w_1, w_2\}\setminus \{u,w\}$.  Choose such an edge and add it to $Q_i$.

If instead $i\in [k+\ell_1+1,k+\ell_1+\ell_2]$, then $i$ is a Stage~II
color. By Stage~I, no remaining edge of $R_B$ is incident with a vertex of
$(Z\cap B)\cup\{v_2\}$. Since we are now in the case
$d_{G^s}(v_1)<\eta^{\frac17}n$, every neighbor of $v_1$ in $B$ lies in
$(Z\cap B)\cup\{v_2\}$. Hence
\[
e_G(v_1,B)=e_G(v_1,(Z\cap B)\cup\{v_2\}).
\]
Moreover, the number of edges from $v_1$ to $B$ that were put into
$G_{A,B}$ is exactly $|M|=d_{G_{A,B}}(v_1)-d_G(v_1,A)$, and the number of
edges of $H$ incident with $v_1$ that have already been colored is at most
$150\eta n+\ell_1+\ell_2$. Therefore
\[
\begin{aligned}
	e_{H_i}(v_1,(Z\cap B)\cup\{v_2\})
	&\ge e_G(v_1,B)-|M|-(150\eta n+\ell_1+\ell_2)\\
	&=\Delta-d_{G_{A,B}}(v_1)-(150\eta n+\ell_1+\ell_2)\\
	&\ge \Delta/2-14\eta n-(150\eta n+\ell_1+\ell_2)>0.
\end{aligned}
\]
Thus we can choose an available edge between $v_1$ and a vertex of
$(Z\cap B)\cup\{v_2\}$ and add it to $Q_i$.

	Next we consider $v_2$. 
	If $v_2\in B_i$, or if $v_2$ has already been covered by the edge between
	$v_1$ and $v_2$ in $Q_i$, then no further edge incident with $v_2$ is needed.
	Suppose $v_2\notin B_i$ and $v_2$ has not yet been covered by $Q_i$. Since
	we are in Case~2, $e_G(v_1,v_2)<\Delta-\eta^{\frac17}n$, and hence
	$d_G(v_2,V(G)\setminus\{v_1\})>\eta^{\frac17}n$. By Condition~\eqref{item:thm4.1.c},
	$d_{G^s}(v_2,V(G)\setminus\{v_1\})>\eta^{\frac17}n-13\eta n$. Since
	$v_2\in B$, \eqref{eqn:partition-deg} gives
	\[
	d_{G^s}(v_2,A\setminus\{v_1\})
	\ge \frac{1}{2}\left(\eta^{\frac17}n-13\eta n-n^{\frac23}-\eta n\right)-1
	> \frac{1}{3}\eta^{\frac17}n.
	\]
	After excluding the vertices in $A_i$ and those corresponding to already
	colored edges of $H$ incident with $v_2$, at least one neighbor
	$y\in A\setminus A_i$ remains such that some edge in $E_G(v_2,y)$ is still
	uncolored. Indeed, the number of excluded neighbors  of $v_2$ is at most
	$|A_i|+(150\eta n+\ell_1+\ell_2)$, and
	$|A_i|\le 2\eta^{\frac16}n$.  Choose one such edge and add it to $Q_i$.
	
	Let $H_i^*$ be obtained from $H_i$ by deleting the endpoints of the edges in
	$Q_i$. We show $H_i^*$ has a perfect matching. Suppose not. By Hall's theorem,
	there exists $X\subseteq A\cap V(H_i^*)$ with $|N_{H_i^*}(X)|<|X|$. Since
	$H_i^*$ is obtained from $H$ by deleting at most $2\eta^{\frac16}n+2$ vertices
	from each side and removing only already colored edges of $H$, we estimate
	its minimum degree. By the choice of $Q_i$, the vertices $v_1$ and $v_2$ are deleted from
	$H_i^*$ whenever they were not already deleted in the definition of $H_i$. Thus Condition~\eqref{item:thm4.1.b} gives
	$d^s_G(v)\ge \frac{\Delta}{2}-6\eta n$. Thus,
	by~\eqref{eqn:partition-deg}, every remaining vertex of $H_i^*$ has at least
	$\frac{1}{2}\left(\frac{\Delta}{2}-6\eta n-n^{\frac23}-\eta n\right)>\frac{1}{5}\alpha n$
	neighbors in the opposite part in $H$ before the deletions. Hence every
	remaining vertex of $H_i^*$ has degree at least
	\[
	\delta(H_i^*)>
	\frac{1}{5}\alpha n
	-2
	-2\eta^{\frac16}n
	-\left(150\eta n+5\eta^{\frac13}n+\ell_1+\ell_2\right)
	>8\tau |A|.
	\]
	Thus $|N_{H_i^*}(X)|>\frac{1}{5}\alpha n$, in particular $|X|>8\tau |A|$.
	On the other hand, since the vertices in $(B\cap V(H_i^*))\setminus N_{H_i^*}(X)$
	have no neighbors in $X$, the same minimum-degree bound implies
	$|(A\cap V(H_i^*))\setminus X|>8\tau |A|$. Hence
	$8\tau |A|<|X|<(1-8\tau)|A|$.
	
	Since the underlying simple graph of $H_i^*$ is obtained from $H^s$ by deleting at most $3\eta^{\frac16}n$
	vertices from each side and at most
	$150\eta n+5\eta^{\frac13}n+\ell_1+\ell_2$ edges incident with each vertex,
	and $H^s$ is a bipartite robust $(\gamma,4\tau)$-expander with
	$\eta\ll\gamma\ll\tau$, by Lemma~\ref{lem:bipartite-stability}\eqref{lem:bipartite-stability-a} and~\eqref{lem:bipartite-stability-b},  $H_i^*$ is still a bipartite robust
	$(\gamma/2,8\tau)$-expander. Applying robust expansion to $X$ gives
	$|N_{H_i^*}(X)|\ge |X|$, contradicting the choice of $X$. Therefore $H_i^*$
	has a perfect matching, say $M_i$.
	
	Then $M_i\cup Q_i$ is a perfect matching of $H_i$. Color all edges of
	$M_i\cup Q_i$ with color $i$, extending it to a 1-factor.
	
	Repeating this procedure for every color in $[k+1,k+\ell_1+\ell_2]$, we
	obtain  an additional  $\ell_1+\ell_2$ 1-factors. Together with the $k$ 1-factors
	already obtained in Step~3, this gives $k+\ell_1+\ell_2$ pairwise
	edge-disjoint 1-factors of $G$.
	
	\begin{center}
		{\bf Step 5: Edge-color the remaining bipartite multigraph}
	\end{center}
	
	Let $R$ be the subgraph of $G$ consisting of the remaining uncolored edges.
	By the construction in the previous steps, all edges of $G[A]$ and $G[B]$
	have already been colored, so every edge of $R$ belongs to $H=G[A,B]$. Thus
	$R$ is bipartite with parts $A$ and $B$.
	
	Since we have $k+\ell_1+\ell_2$ pairwise edge-disjoint 1-factors, each
	vertex of $G$ is incident with exactly one edge of each of these color
	classes. Therefore the remaining multigraph $R$ is regular of degree
	$\Delta(R)=\Delta(G)-k-\ell_1-\ell_2$.
	
	By Theorem~\ref{konig}, $R$ admits a proper edge-coloring with $\Delta(R)$
	colors, using the colors in $[k+\ell_1+\ell_2+1,\Delta]$. Since $R$ is
	$\Delta(R)$-regular, every vertex of $R$ is incident with exactly one edge of
	each of these new colors, so each new color class is a 1-factor of $G$.
	
	Together with the $k+\ell_1+\ell_2$ 1-factors constructed in the previous
	steps, this gives an edge-coloring of $G$ with exactly $\Delta$ colors, and
	every color class is a 1-factor. Hence $G$ is 1-factorizable.
	
\end{proof}

\section*{Declaration of Use of AI Tools}

During the preparation of this manuscript, the authors used ChatGPT 5.5 Plus 
 to assist with language
editing, grammar, clarity, and formatting. The authors reviewed and verified
all AI-assisted edits and take full responsibility for the content of the
manuscript.
	
	\section*{Acknowledgments}
	The third author would like to dedicate this paper to the memory of Professor Bing Yao of Northwest Normal University, China, who passed away on May 24, 2026. His passion for research and teaching, together with his selfless support of students, will always be remembered.
	
	\bibliographystyle{abbrv}
	\bibliography{SSL-BIB_08-19}
\end{document}